\documentclass[letterpaper,11pt]{article}
\usepackage[margin=1in]{geometry}
\usepackage{amsmath,amsthm,amsfonts}
\usepackage{graphicx}
\usepackage[pdfborder={0 0 1}]{hyperref}
\usepackage[capitalize]{cleveref}
\usepackage{subfigure}
\usepackage{color}
\usepackage{bbm}
\usepackage{ifthen}
\usepackage{tikz}
\usetikzlibrary{positioning}
\usetikzlibrary{calc}
\usepackage{enumitem}
\usepackage{ulem}

\newif\ifdraft
\drafttrue
\newtheorem{theorem}{Theorem}[section]
\newtheorem{lemma}[theorem]{Lemma}
\newtheorem{corollary}[theorem]{Corollary}
\newtheorem{definition}[theorem]{Definition}
\newtheorem{proposition}[theorem]{Proposition}
\newtheorem{assumption}[theorem]{Assumption}
\newtheorem{claim}[theorem]{Claim}

\newcommand{\W}{W}
\newcommand{\Wh}{\hat{W}}
\newcommand{\Ind}{I}
\newcommand{\Vsize}{N}
\newcommand{\pmax}{p}

\newcommand{\Slc}{\mathbf{S}}
\newcommand{\Tcons}{T_{\mathrm{cons}}}
\newcommand{\Vs}{U}
\newcommand{\Es}{F}
\newcommand{\Bthree}{\mathrm{Bo3}}
\newcommand{\Btwo}{\mathrm{Bo2}}
\newcommand*{\absp}[1]{\langle #1 \rangle_+}

\tikzset{every picture/.style={font issue=\footnotesize},
            font issue/.style={execute at begin picture={#1\selectfont}}
           }

\crefname{equation}{}{}
\crefname{item}{}{}
\crefname{assumption}{Assumption}{Assumption}
\crefname{figure}{Figure}{Figure}

\renewcommand\Pr{\mathop{\mathbf{Pr}}}
\DeclareMathOperator{\E}{\mathbf{E}}
\newcommand{\Var}{\mathop{\mathbf{Var}}}%add
\def\defeq{\mathrel{\mathop:}=}%add
\makeatletter

\newcommand{\Proofname}{Proof}
\makeatother
\makeatletter
\def\BOXSYMBOL{\RIfM@\bgroup\else$\bgroup\aftergroup$\fi
  \vcenter{\hrule\hbox{\vrule height.6em\kern.6em\vrule}\hrule}\egroup}
\makeatother
\newcommand{\BOX}{%
  \ifmmode\else\leavevmode\unskip\penalty9999\hbox{}\nobreak\hfill\fi
  \quad\hbox{\BOXSYMBOL}}
\newcommand\QED{\BOX}

\makeatother

\title{Phase Transitions of Best-of-Two and Best-of-Three\\ on Stochastic Block Models}

\author{Nobutaka Shimizu\thanks{The University of Tokyo, Japan. {\ttfamily nobutaka\_shimizu@mist.i.u-tokyo.ac.jp}} \and Takeharu Shiraga\thanks{Chuo University, Japan. {\ttfamily shiraga.076@g.chuo-u.ac.jp}}}

\date{\today}

\begin{document}
\maketitle

%\tableofcontents

\begin{abstract}
This paper is concerned with voting processes on graphs where each vertex holds one of two different opinions.
In particular, we study the \emph{Best-of-two} and the \emph{Best-of-three}.
Here at each synchronous and discrete time step, each vertex updates its opinion to match the majority among the opinions of two random neighbors and itself (the Best-of-two) or the  opinions of three random neighbors (the Best-of-three).
Previous studies have explored these processes on complete graphs and expander graphs, 
but we understand significantly less about their properties on graphs with more complicated structures.

In this paper, we study the Best-of-two and the Best-of-three on the stochastic block model $G(2n,p,q)$, which is a random graph consisting of two distinct Erd\H{o}s-R\'enyi graphs $G(n,p)$ joined by random edges with density $q\leq p$.
We obtain two main results.
First, if $p=\omega(\log n/n)$ and $r=q/p$ is a constant, we show that there is a phase transition in $r$ with threshold $r^*$ (specifically, $r^*=\sqrt{5}-2$ for the Best-of-two, and $r^*=1/7$ for the Best-of-three).
If $r>r^*$, the process reaches consensus within $O(\log \log n+\log n/\log (np))$ steps for any initial opinion configuration with a bias of $\Omega(n)$.
By contrast, if $r<r^*$, then there exists an initial opinion configuration with a bias of $\Omega(n)$ from which the process requires at least $2^{\Omega(n)}$ steps to reach consensus.
Second, if $p$ is a constant and  $r>r^*$, we show that, for any initial opinion configuration, the process reaches consensus within $O(\log n)$ steps.
To the best of our knowledge, this is the first result concerning multiple-choice voting for arbitrary initial opinion configurations on non-complete graphs.

\ 


\end{abstract}

%\thispagestyle{empty}
%\clearpage
%\addtocounter{page}{-1}

\section{Introduction}
This paper is concerned with {\em voting processes} %(a.k.a.~the {\em opinion dynamics}) 
on distributed networks.
Consider an undirected connected graph $G=(V,E)$ where each vertex $v\in V$ initially holds an opinion from a finite set.
A~voting~process is defined by a local updating rule:~Each vertex updates its opinion according to the rule.
Voting processes appear as simple mathematical models in a wide range of fields, e.g.~social behavior, physical phenomena and biological systems~\cite{MNT14,Liggett85,AABHBB11}.
In distributed computing, voting~processes are known as a simple approach for consensus problems~\cite{FLM86,GK10}.

\subsection{Previous work}
The synchronous {\em pull voting} (a.k.a.~the {\em voter model}) is a simple and well-studied voting process~\cite{NIY99,HP01}.
In the pull voting, at each synchronous and discrete time step, each vertex adopts the opinion of a randomly selected neighbor. 
Here, the main quantity of interest is the {\em consensus time}, which is the number of steps required to reach consensus (i.e.~the configuration where all vertices hold the same opinion).
Hassin and Peleg~\cite{HP01} showed that the expected consensus time is $O(n^3\log n)$ for all non-bipartite graphs and for all initial opinion configurations, where $n$ is the number of vertices.
Note that, for bipartite graphs, there exists an initial opinion configuration that never reaches consensus.

The pull voting has been extended to develop voting processes where each vertex queries multiple neighbors at each step.
%Voting processes in which each vertex queries multiple neighbors at each step are studied as a simple extension of the pull voting. 
The simplest multiple-choice voting process is the {\em Best-of-two} ({\em two sample voting}, or {\em 2-Choices}), where each vertex $v\in V$ randomly samples two neighbors (with replacement) and, if both hold the same opinion, adopts it\footnote{If the graph initially involves two possible opinions, this definition matches the rule described in Abstract.}.
Doerr et al.~\cite{DGMSS11} showed that, for complete graphs initially involving two possible opinions, the consensus time of the Best-of-two is $O(\log n)$ with high probability\footnote{In this paper ``with high probability'' (w.h.p.) means probability at least $1-n^{-c}$ for a constant $c>0$.}.
Likewise, the {\em Best-of-three} (a.k.a.~{\em 3-Majority}) is another simple multiple-choice voting process where each vertex adopts the majority opinion among those of three randomly selected neighbors.
Several researchers have studied this model on complete graphs initially involving $k\geq 2$ opinions~\cite{BCNPT16, BCNPST17, BCEKMN17, GL18}.
%: Roughly speaking, it is known that the consensus time is $O(k\log n)$ if the number of opinions is $k$~\cite{GL18}.
For example, Ghaffari and Lengler~\cite{GL18} showed that the consensus time of the Best-of-three is $O(k\log n)$ if $k=O(n^{1/3}/\sqrt{\log n})$.
%{Here, the tie breaks down to the first one among the three random neighbors.}
%On the other direction, several previous works investigated the Best-of-two and the Best-of-three  with adversary~\cite{DGMSS11, BCNPT16, GL18}.

Several studies of multiple-choice voting processes on non-complete graphs have considered expander graphs with an {\em initial bias}, i.e.~a difference between the initial sizes of the largest and the second largest opinions.
Cooper et al.~\cite{CER14} showed that, for any regular expander graph initially involving two opinions, the Best-of-two reaches consensus within $O(\log n)$ steps w.h.p.~if the initial bias is $\Omega(n\lambda_2)$, where $\lambda_2$ is the second largest eigenvalue of the graph's transition matrix. 
This result was later extended to general expander graphs, including Erd\H{o}s-R\'enyi random graphs $G(n,p)$, under milder assumptions about the initial bias~\cite{CERRS15}.
Recall that the Erd\H{o}s-R\'enri graph $G(n,p)$ is a graph on $n$ vertices where each vertex pair is joined by an edge with probability $p$, independent of any other pairs.
In \cite{CRRS17}, the authors studied the Best-of-two and the Best-of-three on regular expander graphs initially involving more than two opinions.
In \cite{AD15,KR19}, the authors studied multiple-choice voting processes on non-complete graphs with random initial configuration.

Recently, the Best-of-two on richer classes of graphs involving two opinions have been studied.
Previous works proved interesting results which do not hold on complete graphs or expander graphs.
Cruciani et al.~\cite{CNNS18} studied the Best-of-two on the {\em core periphery network}, namely a graph consisting of core vertices and periphery vertices.
They showed that a phase transition can occur, depending on the density of edges between  core and periphery vertices: 
Either the process reaches consensus within $O(\log n)$ steps, 
or remains a configuration where both opinions coexist for at least $\Omega(n)$ steps.
Cruciani et al.~\cite{CNS19} studied the Best-of-two on the $(a,b)$-{\em regular stochastic block model}, which is a graph consisting of two $a$-regular graphs connected by a $b$-regular bipartite graph.
Under certain assumptions including $b/a=O(n^{-0.5})$, they showed that, starting from a random initial opinion configuration, the process reaches an almost {\em clustered} configuration (e.g.~both communities are in almost consensus but the opinions are distinct) within $O(\log n)$ steps with constant probability, then stays in that configuration for at least $\Omega(n)$ steps w.h.p.
They also proposed a distributed community detection algorithm based on this property. 
\subsection{Our results}
This paper considers the \emph{stochastic block model}, a well-known random graph model that forms multiple communities.
This model has been well-explored in a wide range of fields, including biology~\cite{CY06, MPNRYE99}, network analysis~\cite{BDLB17,GZFA10} and machine learning~\cite{AS15,Abbe18}, where it serves as a benchmark for community detection algorithms.
The study of the voting processes on the stochastic block model has a potential application in distributed community detection algorithms~\cite{BCMNPRT18,BCNPT17,CNS19}.
In this paper, we focus on the following model which admits two communities of equal size.
\begin{definition}[Stochastic block model]
\label{def:SBM}
For $n\in\mathbb{N}$ and $p,q\in[0,1]$ with $q\leq p$, the stochastic block model $G(2n,p,q)$ is a graph on a vertex set $V=V_1\cup V_2$, where $|V_1|=|V_2|=n$ and $V_1\cap V_2=\emptyset$.
In addition, each pair $\{u,v\}$ of distinct vertices $u\in V_i$ and $v\in V_j$ forms an edge with probability $\theta$, independent of any other edges, where
\begin{align*}
\theta = \begin{cases}
p & \text{if $i=j$},\\
q & \text{otherwise}.
\end{cases}
\end{align*}
\end{definition}
Note that $G(2n,p,q)$ is not connected w.h.p.~if $p=o(\log n/n)$~\cite{FK16}.
Throughout this paper, we assume $p=\omega(\log n/n)$, in which regime each community is connected w.h.p.

In this paper, we first generate a random graph $G(2n,p,q)$, 
and then set an initial opinion configuration from $\{1,2\}$.
Let $A^{(0)}, A^{(1)}, \ldots$ be a sequence of random vertex subsets where $A^{(t)}$ is the set of vertices of opinion $1$ at step $t$.
For any $A\subseteq V$, the consensus time $\Tcons(A)$ is defined as
\begin{align*}
\Tcons(A)\defeq \min\left\{t\geq 0\,:\,A^{(t)}\in\{\emptyset,V\},\,A^{(0)}=A\right\}.
\end{align*}
We obtain two main results, described below.
\paragraph*{Result I: phase transition.}
Observe that, if $p=q=1$, then $G(2n,1,1)$ is a complete graph and the consensus time of the Best-of-two is $O(\log n)$, from the results of \cite{DGMSS11}.
On the other hand, the graph $G(2n,1,0)$ consists of two disjoint complete graphs, each of size $n$,
meaning that, depending on the initial state, it may not reach consensus. 
%; hence it may not admit consensus depending on the initial state.
This naturally raises the following question:
{\em Where is the boundary between these two phenomena?}
This motivated us to study the consensus times of the Best-of-two and the Best-of-three on $G(2n,p,q)$ for a wide range of $r\defeq q/p$, and led us to propose the following answers.
\begin{theorem}[Phase transition of the Best-of-three on $G(2n,p,q)$] \label{thm:phasetransition_3M}
Consider the Best-of-three on $G(2n,p,q)$ such that $r\defeq \frac{q}{p}$ is a constant.

\begin{enumerate}[label=(\roman*)]
\item\label{state:abovethroshold_3M}
If $r>\frac{1}{7}$, then $G(2n,p,q)$ w.h.p.~satisfies the following property:
There exist two positive constants $C, C'>0$ such that
\begin{align*}
&\forall A\subseteq V\text{ of $\bigl||A|-|V\setminus A|\bigr|= \Omega(n)$}
:\\
&\Pr\left[\Tcons(A)\leq C\left(\log\log n+\frac{\log n}{\log(np)}\right)\right] \geq 1-O\left(n^{-C'}\right).
\end{align*}
%
%If $r>\frac{1}{7}$, then two positive constants $C,C'>0$ exist such that $G(2n,p,q)$ w.h.p.~satisfies
%\begin{align*}
%&\forall A\subseteq V\text{ of $\bigl||A|-|V\setminus A|\bigr|= \Omega(n)$}
%:\,
%\Pr\left[\Tcons(A)\leq C\left(\log\log n+\frac{\log n}{\log(np)}\right)\right] \geq 1-O\left(n^{-C'}\right).
%\end{align*}
\item\label{state:belowthroshold_3M} 
If $r<\frac{1}{7}$, then $G(2n,p,q)$ w.h.p.~satisfies the following property:
There exist a set $A\subseteq V$ with $\bigl||A|-|V\setminus A|\bigr|= \Omega(n)$ and two positive constants $C, C'>0$ such that
\begin{align*}
%&\exists A\subseteq V\text{ of $\bigl||A|-|V\setminus A|\bigr|= \Omega(n)$}
%:\,
\Pr\left[\Tcons(A) \geq \exp(Cn) \right] \geq 1-O\left(n^{-C'}\right).
\end{align*}
%
%If $r<\frac{1}{7}$, then two positive constants $C,C'>0$ exist such that $G(2n,p,q)$ w.h.p.~satisfies
%\begin{align*}
%&\exists A\subseteq V\text{ of $\bigl||A|-|V\setminus A|\bigr|= \Omega(n)$}
%:\,
%\Pr\left[\Tcons(A) \geq \exp(Cn) \right] \geq 1-O\left(n^{-C'}\right).
%\end{align*}
\end{enumerate}
\end{theorem}
\begin{theorem}[Phase transition of the Best-of-two on $G(2n,p,q)$] \label{thm:phasetransition_2C}
Consider the Best-of-two on $G(2n,p,q)$ such that $r\defeq \frac{q}{p}$ is a constant.

\begin{enumerate}[label=(\roman*)]
\item\label{state:abovethroshold_2C}
If $r>\sqrt{5}-2$, then $G(2n,p,q)$ w.h.p.~satisfies the following property:
There exist two positive constants $C, C'>0$ such that
\begin{align*}
&\forall A\subseteq V\text{ of $\bigl||A|-|V\setminus A|\bigr|= \Omega(n)$}
:\\
&\Pr\left[\Tcons(A)\leq C\left(\log\log n+\frac{\log n}{\log(np)}\right)\right] \geq 1-O\left(n^{-C'}\right).
\end{align*}

\item\label{state:belowthroshold_2C}
%If $r<\sqrt{5}-2$, then $G(2n,p,q)$ w.h.p.~satisfies the following property:
%There exist a set $A\subseteq V$ with $\bigl||A|-|V\setminus A|\bigr|= \Omega(n)$ and constants $C,C'>0$ such that 
%\begin{align*}
%\Pr\left[\Tcons(A) \geq \exp(Cn) \right] \geq 1-O\left(n^{-C'}\right).
%\end{align*}
If $r<\sqrt{5}-2$, then $G(2n,p,q)$ w.h.p.~satisfies the following property:
There exist a set $A\subseteq V$ with $\bigl||A|-|V\setminus A|\bigr|= \Omega(n)$ and two positive constants $C, C'>0$ such that
\begin{align*}
%&\exists A\subseteq V\text{ of $\bigl||A|-|V\setminus A|\bigr|= \Omega(n)$}
%:\,
\Pr\left[\Tcons(A) \geq \exp(Cn) \right] \geq 1-O\left(n^{-C'}\right).
\end{align*}
\end{enumerate}
\end{theorem}
Note that the upper bound $\Tcons(A)= O(\log\log n+\log n/\log(np))$ is tight up to a constant factor if $\log n / \log (np) \geq \log\log n$.
To see this, observe that there exists an $A\subseteq V$ such that $\Tcons(A)$ is at least half of the diameter.
In addition, it is easy to see that the diameter of $G(2n,p,q)$ is $\Theta(\log n / \log (np))$ w.h.p.~\cite{FK16}.

We also note that the consensus time of the pull voting is $O({\rm poly}(n))$ for any non-bipartite graph~\cite{HP01}.
%DISCverに合わせた 
To the best of our knowledge, \cref{thm:phasetransition_3M} and \cref{thm:phasetransition_2C} provide the first nontrivial graphs where the consensus time of a multiple-choice voting process is exponentially slower than that of the pull voting. 
%To the best of our knowledge, \cref{thm:phasetransition_3M} (\ref{thm:phasetransition_2C}) is the first example where the consensus time of a multiple-choice voting process is exponentially slower than that of the pull voting. 

%
\paragraph*{Result I\hspace{-.1em}I: worst-case analysis.}
The most central topic in voter processes is the \emph{symmetry breaking}, i.e.~the number of iterations required to cause a small bias starting from the half-and-half state.
Here, we are interested in the worst-case consensus time with respect to initial opinion configurations.
To the best of our knowledge, all current results on worst-case consensus time of multiple-choice voting processes deal with complete graphs~\cite{DGMSS11,BCNPST17, BCEKMN17, GL18}.
All previous work on non-complete graphs has involved some special bias setting (e.g.~an initial bias~\cite{CER14,CERRS15,CRRS17}, or a random initial opinion configuration~\cite{AD15,CNS19,KR19}).
In this paper, we present the following first worst-case analysis of non-complete graphs.
\begin{theorem}[Worst-case analysis of the Best-of-three on $G(2n,p,q)$] \label{thm:worst_case_3M}
Consider the Best-of-three on  $G(2n,p,q)$ such that $p$ and $q$ are positive constants.
If $\frac{q}{p}>\frac{1}{7}$, then $G(2n,p,q)$ w.h.p.~satisfies the following property:
There exist two positive constants $C,C'>0$ such that %for any $A\subseteq V$,
\begin{align*}
\forall A\subseteq V\, :\, \Pr\left[\Tcons(A)\leq C\log n \right] \geq 1-O\left(n^{-C'}\right).
\end{align*}
\end{theorem}
\begin{theorem}[Worst-case analysis of the Best-of-two on $G(2n,p,q)$] \label{thm:worst_case_2C}
Consider the Best-of-two on $G(2n,p,q)$ such that $p$ and $q$ are positive constants.
If $\frac{q}{p}>\sqrt{5}-2$, then $G(2n,p,q)$ w.h.p.~satisfies the following property:
There exist two positive constants $C,C'>0$ such that %for any $A\subseteq V$,
\begin{align*}
\forall A\subseteq V\, :\, \Pr\left[\Tcons(A)\leq C\log n \right] \geq 1-O\left(n^{-C'}\right).
\end{align*}
\end{theorem}

Based on these theorems, an immediate but important corollary follows.
\begin{corollary}
For any constant $p>0$, the Best-of-two and the Best-of-three on the Erd\H{o}s-R\'enyi graph $G(n,p)$ reach  consensus within $O(\log n)$ steps w.h.p.~for all initial opinion configurations.
\end{corollary}
Recall that the Best-of-two and the Best-of-three on $G(n,p)$ has been extensively studied in previous works but these works put aforementioned assumptions on initial bias.
\subsection{Strategy}
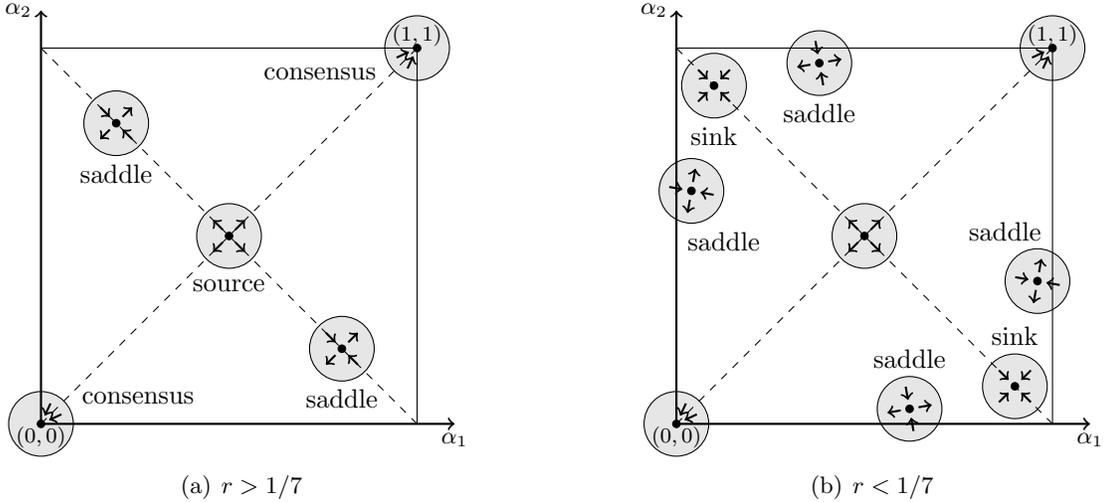
\begin{figure}[t]
\centering
	\begin{tabular}{c}
		\begin{minipage}{0.5\hsize}
		\centering
\subfigure[$r>1/7$]{\begin{tikzpicture}
	\colorlet{Area}{white!60!lightgray}
	\def\L{5};
	\def\R{0.43};
	\def\r{0.05};
	\def\over{0.5};
	\def\D{1.5};
	\def\rout{0.3};
	\def\rin{0.12};
	\def\angle{25};

	\coordinate (BB) at (0,0);
	\coordinate (AB) at (\L,0);
	\coordinate (BA) at (0,\L);
	\coordinate (AA) at (\L,\L);
	\coordinate (HH) at (\L/2, \L/2);
	\coordinate (S1) at (\L/2-\D, \L/2+\D);	
	\coordinate (S2) at (\L/2+\D, \L/2-\D);	
		
	\draw[fill=Area] (BB) circle (\R) {};  
	\draw[fill=Area] (AA) circle (\R) {};
	\draw[fill=Area] (HH) circle (\R) {};
	\draw[fill=Area] (S1) circle (\R) {};
	\draw[fill=Area] (S2) circle (\R) {};
	\draw[fill=Area] node[below = \R, anchor=north] at (S1) {\small{saddle}};
	\draw[fill=Area] node[below = \R, anchor=north] at (S2) {\small{saddle}};	
	\draw[fill=Area] node[below = \R*1.1, anchor=north] at (HH) {\small{source}};
	\draw[fill=Area] node[above right = \R*1.25 and \R, anchor=north west] at (BB.north east) {\small{consensus}};
	\draw[fill=Area] node[below left = \R*1.25 and \R, anchor=south east] at (AA.south west) {\small{consensus}};

	\draw node[below=0.17, anchor=center] at (BB.center) {\scriptsize{$(0,0)$}};
	\draw node[above=0.17, anchor=center] at (AA.center) {\scriptsize{$(1,1)$}};	
	
	\draw[fill] (BB) circle (\r) {};  
	\draw[fill] (AA) circle (\r) {};
	\draw[fill] (HH) circle (\r) {};
	\draw[fill] (S1) circle (\r) {};
	\draw[fill] (S2) circle (\r) {};

	\draw[->,thick] (BB) -- (AB) -- +(\over,0) node [anchor=north]{$\alpha_1$};
	\draw[->,thick] (BB) -- (BA) -- +(0,\over) node [anchor=east]{$\alpha_2$};
	\draw (AB) -- (AA) -- (BA);

	\draw[-,dashed] (BB)--(AA);
	\draw[-,dashed] (AB)--(BA);	

	\draw[->,thick] ({\L/2+\rin*cos(45)},{\L/2+\rin*sin(45)}) -- ({\L/2+\rout*cos(45)},{\L/2+\rout*sin(45)});
	\draw[->,thick] ({\L/2-\rin*cos(45)},{\L/2-\rin*sin(45)}) -- ({\L/2-\rout*cos(45)},{\L/2-\rout*sin(45)});
	\draw[->,thick] ({\L/2+\rin*cos(45)},{\L/2-\rin*sin(45)}) -- ({\L/2+\rout*cos(45)},{\L/2-\rout*sin(45)});
	\draw[->,thick] ({\L/2-\rin*cos(45)},{\L/2+\rin*sin(45)}) -- ({\L/2-\rout*cos(45)},{\L/2+\rout*sin(45)});
	\draw[->,thick] ({\L/2-\D+\rin*cos(45)},{\L/2+\D+\rin*sin(45)}) -- ({\L/2-\D+\rout*cos(45)},{\L/2+\D+\rout*sin(45)});
	\draw[->,thick] ({\L/2-\D-\rin*cos(45)},{\L/2+\D-\rin*sin(45)}) -- ({\L/2-\D-\rout*cos(45)},{\L/2+\D-\rout*sin(45)});	
	\draw[<-,thick] ({\L/2-\D+\rin*cos(135)},{\L/2+\D+\rin*sin(135)}) -- ({\L/2-\D+\rout*cos(135)},{\L/2+\D+\rout*sin(135)});
	\draw[<-,thick] ({\L/2-\D-\rin*cos(135)},{\L/2+\D-\rin*sin(135)}) -- ({\L/2-\D-\rout*cos(135)},{\L/2+\D-\rout*sin(135)});	
	\draw[->,thick] ({\L/2+\D+\rin*cos(45)},{\L/2-\D+\rin*sin(45)}) -- ({\L/2+\D+\rout*cos(45)},{\L/2-\D+\rout*sin(45)});
	\draw[->,thick] ({\L/2+\D-\rin*cos(45)},{\L/2-\D-\rin*sin(45)}) -- ({\L/2+\D-\rout*cos(45)},{\L/2-\D-\rout*sin(45)});	
	\draw[<-,thick] ({\L/2+\D+\rin*cos(135)},{\L/2-\D+\rin*sin(135)}) -- ({\L/2+\D+\rout*cos(135)},{\L/2-\D+\rout*sin(135)});
	\draw[<-,thick] ({\L/2+\D-\rin*cos(135)},{\L/2-\D-\rin*sin(135)}) -- ({\L/2+\D-\rout*cos(135)},{\L/2-\D-\rout*sin(135)});	
	\draw[->,thick] ({\rout*cos(\angle)},{\rout*sin(\angle)}) -- ({\rin*cos(\angle)},{\rin*sin(\angle)});
	\draw[->,thick] ({\rout*cos(90-\angle)},{\rout*sin(90-\angle)}) -- ({\rin*cos(90-\angle)},{\rin*sin(90-\angle)});
	\draw[->,thick] ({\L-\rout*cos(\angle)},{\L-\rout*sin(\angle)}) -- ({\L-\rin*cos(\angle)},{\L-\rin*sin(\angle)});
	\draw[->,thick] ({\L-\rout*cos(90-\angle)},{\L-\rout*sin(90-\angle)}) -- ({\L-\rin*cos(90-\angle)},{\L-\rin*sin(90-\angle)});	
\end{tikzpicture}
}
		\end{minipage}
		\begin{minipage}{0.5\hsize}
		\centering
		\subfigure[$r<1/7$]{\begin{tikzpicture}
	\colorlet{Area}{white!60!lightgray}
	\def\L{5};
	\def\R{0.43};
	\def\r{0.05};
	\def\over{0.5};
	\def\D{2};
	\def\rout{0.3};
	\def\rin{0.12};
	\def\angle{25};
	\def\len{0.7};
	\def\llen{0.15};	
	\def\aangle{10};

	\coordinate (BB) at (0,0);
	\coordinate (AB) at (\L,0);
	\coordinate (BA) at (0,\L);
	\coordinate (AA) at (\L,\L);
	\coordinate (HH) at (\L/2, \L/2);
	\coordinate (S1) at (\L/2-\D, \L/2+\D);	
	\coordinate (S2) at (\L/2+\D, \L/2-\D);
	
	\coordinate[above right = \llen*\D and \len*\D] (N11) at (S1);
	\coordinate[below left = \len*\D and \llen*\D] (N12) at (S1);
	\coordinate[above right = \len*\D and \D*\llen] (N21) at (S2);
	\coordinate[below left = \llen*\D and \len*\D] (N22) at (S2);

	\coordinate[above left = {\rout*cos(\aangle)} and {\rout*sin(\aangle)}] (N11s1) at (N11);
	\coordinate[above left = {\rin*cos(\aangle)} and {\rin*sin(\aangle)}] (N11t1) at (N11);
	\coordinate[above left = {\rout*cos(\aangle+90)} and {\rout*sin(\aangle+90)}] (N11s2) at (N11);
	\coordinate[above left = {\rin*cos(\aangle+90)} and {\rin*sin(\aangle+90)}] (N11t2) at (N11);
	\coordinate[above left = {\rout*cos(\aangle+180)} and {\rout*sin(\aangle+180)}] (N11s3) at (N11);
	\coordinate[above left = {\rin*cos(\aangle+180)} and {\rin*sin(\aangle+180)}] (N11t3) at (N11);
	\coordinate[above left = {\rout*cos(\aangle+270)} and {\rout*sin(\aangle+270)}] (N11s4) at (N11);
	\coordinate[above left = {\rin*cos(\aangle+270)} and {\rin*sin(\aangle+270)}] (N11t4) at (N11);

	\coordinate[above left = {\rout*cos(-\aangle)} and {\rout*sin(-\aangle)}] (N12s1) at (N12);
	\coordinate[above left = {\rin*cos(-\aangle)} and {\rin*sin(-\aangle)}] (N12t1) at (N12);
	\coordinate[above left = {\rout*cos(-\aangle+90)} and {\rout*sin(-\aangle+90)}] (N12s2) at (N12);
	\coordinate[above left = {\rin*cos(-\aangle+90)} and {\rin*sin(-\aangle+90)}] (N12t2) at (N12);
	\coordinate[above left = {\rout*cos(-\aangle+180)} and {\rout*sin(-\aangle+180)}] (N12s3) at (N12);
	\coordinate[above left = {\rin*cos(-\aangle+180)} and {\rin*sin(-\aangle+180)}] (N12t3) at (N12);
	\coordinate[above left = {\rout*cos(-\aangle+270)} and {\rout*sin(-\aangle+270)}] (N12s4) at (N12);
	\coordinate[above left = {\rin*cos(-\aangle+270)} and {\rin*sin(-\aangle+270)}] (N12t4) at (N12);
	
	\coordinate[above left = {\rout*cos(\aangle)} and {\rout*sin(\aangle)}] (N22s1) at (N22);
	\coordinate[above left = {\rin*cos(\aangle)} and {\rin*sin(\aangle)}] (N22t1) at (N22);
	\coordinate[above left = {\rout*cos(\aangle+90)} and {\rout*sin(\aangle+90)}] (N22s2) at (N22);
	\coordinate[above left = {\rin*cos(\aangle+90)} and {\rin*sin(\aangle+90)}] (N22t2) at (N22);
	\coordinate[above left = {\rout*cos(\aangle+180)} and {\rout*sin(\aangle+180)}] (N22s3) at (N22);
	\coordinate[above left = {\rin*cos(\aangle+180)} and {\rin*sin(\aangle+180)}] (N22t3) at (N22);
	\coordinate[above left = {\rout*cos(\aangle+270)} and {\rout*sin(\aangle+270)}] (N22s4) at (N22);
	\coordinate[above left = {\rin*cos(\aangle+270)} and {\rin*sin(\aangle+270)}] (N22t4) at (N22);

	\coordinate[above left = {\rout*cos(-\aangle)} and {\rout*sin(-\aangle)}] (N21s1) at (N21);
	\coordinate[above left = {\rin*cos(-\aangle)} and {\rin*sin(-\aangle)}] (N21t1) at (N21);
	\coordinate[above left = {\rout*cos(-\aangle+90)} and {\rout*sin(-\aangle+90)}] (N21s2) at (N21);
	\coordinate[above left = {\rin*cos(-\aangle+90)} and {\rin*sin(-\aangle+90)}] (N21t2) at (N21);
	\coordinate[above left = {\rout*cos(-\aangle+180)} and {\rout*sin(-\aangle+180)}] (N21s3) at (N21);
	\coordinate[above left = {\rin*cos(-\aangle+180)} and {\rin*sin(-\aangle+180)}] (N21t3) at (N21);
	\coordinate[above left = {\rout*cos(-\aangle+270)} and {\rout*sin(-\aangle+270)}] (N21s4) at (N21);
	\coordinate[above left = {\rin*cos(-\aangle+270)} and {\rin*sin(-\aangle+270)}] (N21t4) at (N21);

	\draw[fill=Area] (BB) circle (\R) {};  
	\draw[fill=Area] (AA) circle (\R) {};
	\draw[fill=Area] (HH) circle (\R) {};
	\draw[fill=Area] (S1) circle (\R) {};
	\draw[fill=Area] (S2) circle (\R) {};
	\draw[fill=Area] (N11) circle (\R) {};
	\draw[fill=Area] (N12) circle (\R) {};	
	\draw[fill=Area] (N21) circle (\R) {};
	\draw[fill=Area] (N22) circle (\R) {};	
	
	\draw node[below = \R, anchor=north] at (S1) {\small{sink}};
	\draw node[above = \R, anchor=south] at (S2) {\small{sink}};		
	\draw node[below = \R, anchor=north] at (N11) {\small{saddle}};	
	\draw node[above = \R, anchor=south] at (N22) {\small{saddle}};		
	\draw node[below right = \R and \R, anchor=north] at (N12) {\small{saddle}};
	\draw node[above left = \R and \R, anchor=south] at (N21) {\small{saddle}};			

%	\draw node[below = \R, anchor=north] at (HH) {\small{source}};
%	\draw node[above right = \R*2 and \R/2, anchor=north west] at (BB.north east) {\small{consensus}};
%	\draw node[below left = \R*2 and \R/2, anchor=south east] at (AA.south west) {\small{consensus}};	

	\draw[fill] (BB) circle (\r) {};  
	\draw[fill] (AA) circle (\r) {};
	\draw[fill] (HH) circle (\r) {};
	\draw[fill] (S1) circle (\r) {};
	\draw[fill] (S2) circle (\r) {};	
	\draw[fill] (N11) circle (\r) {};
	\draw[fill] (N12) circle (\r) {};	
	\draw[fill] (N21) circle (\r) {};
	\draw[fill] (N22) circle (\r) {};
	
	\draw node[below=0.17, anchor=center] at (BB.center) {\scriptsize{$(0,0)$}};
	\draw node[above=0.17, anchor=center] at (AA.center) {\scriptsize{$(1,1)$}};	

	\draw[->,thick] (BB) -- (AB) -- +(\over,0) node [anchor=north]{$\alpha_1$};
	\draw[->,thick] (BB) -- (BA) -- +(0,\over) node [anchor=east]{$\alpha_2$};
	\draw (AB) -- (AA) -- (BA);

	\draw[->,thick] (N11s1)--(N11t1);
	\draw[<-,thick] (N11s2)--(N11t2);
	\draw[->,thick] (N11s3)--(N11t3);
	\draw[<-,thick] (N11s4)--(N11t4);

	\draw[<-,thick] (N12s1)--(N12t1);
	\draw[->,thick] (N12s2)--(N12t2);
	\draw[<-,thick] (N12s3)--(N12t3);
	\draw[->,thick] (N12s4)--(N12t4);
	
	\draw[<-,thick] (N21s1)--(N21t1);
	\draw[->,thick] (N21s2)--(N21t2);
	\draw[<-,thick] (N21s3)--(N21t3);
	\draw[->,thick] (N21s4)--(N21t4);

	\draw[->,thick] (N22s1)--(N22t1);
	\draw[<-,thick] (N22s2)--(N22t2);
	\draw[->,thick] (N22s3)--(N22t3);
	\draw[<-,thick] (N22s4)--(N22t4);

	\draw[-,dashed] (BB)--(AA);
	\draw[-,dashed] (AB)--(BA);	

	\draw[->,thick] ({\L/2+\rin*cos(45)},{\L/2+\rin*sin(45)}) -- ({\L/2+\rout*cos(45)},{\L/2+\rout*sin(45)});
	\draw[->,thick] ({\L/2-\rin*cos(45)},{\L/2-\rin*sin(45)}) -- ({\L/2-\rout*cos(45)},{\L/2-\rout*sin(45)});
	\draw[->,thick] ({\L/2+\rin*cos(45)},{\L/2-\rin*sin(45)}) -- ({\L/2+\rout*cos(45)},{\L/2-\rout*sin(45)});
	\draw[->,thick] ({\L/2-\rin*cos(45)},{\L/2+\rin*sin(45)}) -- ({\L/2-\rout*cos(45)},{\L/2+\rout*sin(45)});
	\draw[<-,thick] ({\L/2-\D+\rin*cos(45)},{\L/2+\D+\rin*sin(45)}) -- ({\L/2-\D+\rout*cos(45)},{\L/2+\D+\rout*sin(45)});
	\draw[<-,thick] ({\L/2-\D-\rin*cos(45)},{\L/2+\D-\rin*sin(45)}) -- ({\L/2-\D-\rout*cos(45)},{\L/2+\D-\rout*sin(45)});	
	\draw[<-,thick] ({\L/2-\D+\rin*cos(135)},{\L/2+\D+\rin*sin(135)}) -- ({\L/2-\D+\rout*cos(135)},{\L/2+\D+\rout*sin(135)});
	\draw[<-,thick] ({\L/2-\D-\rin*cos(135)},{\L/2+\D-\rin*sin(135)}) -- ({\L/2-\D-\rout*cos(135)},{\L/2+\D-\rout*sin(135)});	
	\draw[<-,thick] ({\L/2+\D+\rin*cos(45)},{\L/2-\D+\rin*sin(45)}) -- ({\L/2+\D+\rout*cos(45)},{\L/2-\D+\rout*sin(45)});
	\draw[<-,thick] ({\L/2+\D-\rin*cos(45)},{\L/2-\D-\rin*sin(45)}) -- ({\L/2+\D-\rout*cos(45)},{\L/2-\D-\rout*sin(45)});	
	\draw[<-,thick] ({\L/2+\D+\rin*cos(135)},{\L/2-\D+\rin*sin(135)}) -- ({\L/2+\D+\rout*cos(135)},{\L/2-\D+\rout*sin(135)});
	\draw[<-,thick] ({\L/2+\D-\rin*cos(135)},{\L/2-\D-\rin*sin(135)}) -- ({\L/2+\D-\rout*cos(135)},{\L/2-\D-\rout*sin(135)});	
	\draw[->,thick] ({\rout*cos(\angle)},{\rout*sin(\angle)}) -- ({\rin*cos(\angle)},{\rin*sin(\angle)});
	\draw[->,thick] ({\rout*cos(90-\angle)},{\rout*sin(90-\angle)}) -- ({\rin*cos(90-\angle)},{\rin*sin(90-\angle)});
	\draw[->,thick] ({\L-\rout*cos(\angle)},{\L-\rout*sin(\angle)}) -- ({\L-\rin*cos(\angle)},{\L-\rin*sin(\angle)});
	\draw[->,thick] ({\L-\rout*cos(90-\angle)},{\L-\rout*sin(90-\angle)}) -- ({\L-\rin*cos(90-\angle)},{\L-\rin*sin(90-\angle)});
	
\end{tikzpicture}}
		\end{minipage}
	\end{tabular}	\caption{Four types of zero areas are illustrated. The sink areas do not appear if $r>1/7$. \label{fig:3M_proof_sketch}}
\end{figure}
\paragraph*{Known techniques and our technical contribution.}
Consider a voting process on a graph $G=(V,E)$ where each vertex holds an opinion from $\{1,2\}$, 
%Consider a voting process with two opinions on a graph $G=(V,E)$, 
and let $A$ be the set of vertices holding opinion $1$.
In general, a voting process with two opinions can be seen as a Markov chain with the state space $\{1,2\}^V$.
For $A\subseteq V$, let $A'$ denote the set of vertices that hold opinion $1$ in the next time step.
Then, $|A'|=\sum_{v\in V}\mathbbm{1}_{v\in A'}$ is the sum of independent random variables; thus, $|A'|$ concentrates on $\E[|A'|\mid A]$.

If the underlying graph is a complete graph, the state space can be regarded as $\{0,\ldots,n \}$ (each state represents $|A|$).
%Therefore, we have $\E[|A'| \mid A]=f(|A|)$ for some function $f:\mathbb{N}\to\mathbb{N}$.
%Doerr et al.~\cite{DGMSS11} exploit this idea for the Best-of-two and obtain the worst-case analysis for the consensus time on complete graphs.
%Note the same argument works for the Best-of-three on complete graphs.
Therefore, $\E[|A'| \mid A]$ is expressed as a function of $|A|$, e.g.~in the Best-of-two, 
$\E[|A'| \mid A]=f(|A|)\defeq |A|\bigl(1-\bigl(\frac{|A|}{n}\bigr)^2\bigr)+(n-|A|)\bigl(\frac{|A|}{n}\bigr)^2=n\bigl(3\bigl(\frac{|A|}{n}\bigr)^2-2\bigl(\frac{|A|}{n}\bigr)^3\bigr)$.
%Therefore, we have $\E[|A'| \mid A]=f(|A|)$, where $f(|A|)\defeq |A|\bigl(1-\bigl(\frac{|A|}{n}\bigr)^2\bigr)+(n-|A|)\bigl(\frac{|A|}{n}\bigr)^2=n\bigl(3\bigl(\frac{|A|}{n}\bigr)^2-2\bigl(\frac{|A|}{n}\bigr)^3\bigr)$ in the Best-of-two.
Doerr et al.~\cite{DGMSS11} exploited this idea for the Best-of-two and obtained the worst-case analysis for the consensus time on complete graphs.
Somewhat interestingly, we also have $\E[|A'| \mid A]=f(|A|)$ in the Best-of-three. % on complete graphs.

Cooper et al.~\cite{CER14} extended this approach to the Best-of-two on regular expander graphs.
Specifically, they proved that $\E[|A'|\mid A]=f(|A|)\pm O(\epsilon)$ for all $A\subseteq V$, where $\epsilon=\epsilon(n,\lambda_2)=o(n)$ is some function using the \emph{expander mixing lemma}.
This argument  assumes an initial bias of size $\Omega(\epsilon)$.
In another paper, Cooper et al.~\cite{CERRS15} improved this technique and proved more sophisticated results that hold for general (i.e.~not necessarily regular) expander graphs.
%{Unfortunately, their technique is currently not known to work for the Best-of-three.}

In this paper, we consider $G(2n,p,q)$ on the vertex set $V=V_1\cup V_2$.
Let $A_i\defeq A\cap V_i$ for $A\subseteq V$ and $i=1,2$.
We prove that $G(2n,p,q)$ w.h.p.~satisfies $\E[|A'_i|\mid A] = F_i(|A_1|,|A_2|) \pm O(\sqrt{n/p})$ for all $A\subseteq V$ in the Best-of-three, where $F_i:\mathbb{N}^2\to\mathbb{N}$ is some function ($i=1,2$). See \cref{thm:mainthm_E} for details. 
We show the same result for the Best-of-two~\cref{thm:mainthm_E2}.
Here, our key tool is the concentration method, specifically the Janson inequality (\cref{lem:Janson}) and the Kim-Vu concentration (\cref{lem:Kim-Vu}).

\paragraph*{High-level proof sketch.}

Consider the Best-of-three on $G(2n,p,q)$, and let $A^{(0)}, A^{(1)}, \ldots$ be a sequence of random vertex subsets determined by $A^{(t+1)}\defeq (A^{(t)})'$ for each $t\geq 0$.
Consider a stochastic process ${\boldsymbol \alpha}^{(t)}=(\alpha_1^{(t)}, \alpha_2^{(t)})\in [0,1]^2$ where $\alpha^{(t)}_i=|A^{(t)}\cap V_i|/n$ for $i=1,2$.
Our technical result in the previous paragraph approximates the stochastic process ${\boldsymbol \alpha}^{(t)}$ by the {\em deterministic} process $\mathbf{a}^{(t)}$ defined as $\mathbf{a}^{(t+1)}=H(\mathbf{a}^{(t)})$ and ${\boldsymbol \alpha}^{(0)}=\mathbf{a}^{(0)}$ for some function $H: [0,1]^2\to [0,1]^2$ (See \cref{eqn:H_def,fig:3M_vectorfield}).
The function $H$ induces a two-dimensional dynamical system, which we call the \emph{induced dynamical system}.
Using this, we obtain two results concerning ${\boldsymbol \alpha}^{(t)}$.

First, we show that, for any initial configuration, the process reaches one of the {\em zero areas} (a neighbor of a fixed point of $H$) within a constant number of steps.
To show this, in addition to the approximation result, we used the theory of \emph{competitive dynamical systems}~\cite{HS05}.

Second, we characterize the behavior of ${\boldsymbol \alpha}^{(t)}$ in zero areas.
The zero areas depend only on $r=q/p$, and are classified into four types using the Jacobian matrix:
consensus, sink, saddle and source areas (see \cref{fig:3M_proof_sketch} for a description).
In consensus areas, we show that the process reaches consensus within $O(\log \log n+\log n/\log(np))$ steps.
In sink areas, we show that the process remains there for at least $2^{\Omega(n)}$ steps, and also that sink areas only appear if $r<1/7$. 
In saddle and source areas, we show that the process escapes from there within $O(\log n)$ steps if $p$ is a constant by using techniques of \cite{DGMSS11}.
Intuitively speaking, in these two kinds of areas, there are drifts towards outside.
To apply the techniques of \cite{DGMSS11}, we show that $\Var[|A'_i|]=\Omega(n)$ in the area if $p$ is constant, which leads to our worst-case analysis result.
Indeed, any previous works working on expander graphs did not investigate the worst-case due to the lack of variance estimation.

These arguments also enable us to study the Best-of-two process, which implies \cref{thm:phasetransition_2C}.

\subsection{Related work}
The consensus time of the pull voting process is investigated via its dual process, known as \emph{coalescing random walks}~\cite{HP01, CEOR13, CR16}.
Recently coalescing random walks have been extensively studied, including the relationship with properties of random walks such as the hitting time and the mixing time \cite{KMS19, OP19}.

Other studies have focused on voting processes with more general updating rules.
Cooper and Rivera~\cite{CR16} studied the {\em linear voting model}, whose updating rule is characterized by a set of $n\times n$ binary matrices.
This model covers the synchronous pull and the asynchronous push/pull voting processes. 
However, it does not cover the Best-of-two and the Best-of-three.
Schoenebeck and Yu~\cite{SY18} studied asynchronous voting processes whose updating functions are {\em majority-like} (including the asynchronous {\em Best-of-$(2k+1)$} voting processes).
They gave upper bounds on the consensus times of such models on dense Erd\H{o}s-R\'enyi random graphs using a potential technique.

\paragraph*{Organization.}
First we set notation and precise definition of the Best-of-three in \cref{sec:preliminaries}.
After explaining key properties of the stochastic block model in \cref{sec:concentrationsbm},
we show some auxiliary results of the induced dynamical system in \cref{sec:auxiliary_results}.
Then we derive \cref{thm:phasetransition_3M,thm:worst_case_3M} in \cref{sec:proofofmainthm}.
Our general framework of voting processes and results of the general induced dynamical systems are given in \cref{sec:polynomialvoting,sec:SJacobian}, respectively.
Then we give proofs for key properties of the stochastic block model in \cref{sec:randomgraphs} and 
results of the general induced dynamical system in \cref{sec:generaljacobi}.
In \cref{sec:2Choices}, we show \cref{thm:phasetransition_2C,thm:worst_case_2C},
and we conclude this paper in \cref{sec:concluding}.

%\label{sec:polynomialvoting}
%\label{sec:SJacobian}
%\label{sec:randomgraphs}
%\label{sec:dynamical}
%\label{sec:generaljacobi}

%{2.1--2.4を, 2-5へ変更?2-5がそれぞれ後続の3-8へ対応している
%\section{Key ingredients exemplified on the Best-of-three}\label{sec:key}%}
%
\section{Best-of-three voting process}\label{sec:preliminaries}
%{Notationを移動}
For an $\ell\in \mathbbm{N}$, let $[\ell]\defeq \{1,2,\ldots, \ell\}$.
For a graph $G=(V,E)$ and $v\in V$, 
let $N(v)$ be the set of vertices adjacent to $v$.
Denote the degree of $v\in V$ by $\deg(v)=|N(v)|$.
%In addition, 
For $v\in V$ and $S\subseteq V$, let $\deg_S(v)=|S\cap N(v)|$.
Here, we study the Best-of-three with two possible opinions from $\{1,2\}$.
\begin{definition}[Best-of-three]
Let $G=(V,E)$ be a graph where each vertex holds an opinion from~$\{1,2\}$.
Let 
\begin{align*}
f^{\Bthree}(x)\defeq \binom{3}{3}x^3+\binom{3}{2}x^2(1-x)=3x^2-2x^3.
\end{align*}
For the set $A$ of vertices holding opinion $1$,
let $A'$ denote the set of vertices that hold opinion $1$ after an update.
In the Best-of-three, $A'= \{v\in V: X_v=1\}$ where $(X_v)_{v\in V}$ are independent binary random variables satisfying 
\begin{align*}
\Pr[X_v=1]=
f^{\Bthree}\left(\frac{\deg_A(v)}{\deg(v)}\right).
\end{align*}
%Let $A\subseteq V$ be the set of vertices with opinion $1$ 
%and let $A'\subseteq V$ denote the set of vertices with opinion $1$ after an update.
%An update of the Best-of-three is defined by $|V|$ independent binary random variables $(X_v)_{v\in V}$ such that 
%\begin{align*}
%\Pr[X_v=1]=
%f^{\Bthree}\left(\frac{\deg_A(v)}{\deg(v)}\right),
%\left(\frac{\deg_A(v)}{\deg(v)}\right)^3+3\left(\frac{\deg_A(v)}{\deg(v)}\right)\left(1-\left(\frac{\deg_A(v)}{\deg(v)}\right)\right),
%\end{align*}
\end{definition}
%In other words, $\Pr[v\in A']=\Pr[X_v=1]=\left(\frac{\deg_A(v)}{\deg(v)}\right)^3+3\left(\frac{\deg_A(v)}{\deg(v)}\right)^2\left(1-\left(\frac{\deg_A(v)}{\deg(v)}\right)\right)$ for $v\in V$.
%{Consensus timeの定義の位置を移動}
For a given vertex subset $A^{(0)}\subseteq V$, 
we are interested in the behavior of the Markov chain $(A^{(t)})_{t=0}^{\infty}$, i.e.~the sequence of random vertex subsets determined by $A^{(t+1)}\defeq (A^{(t)})'$ for each $t\geq 0$.
%For any $A\subseteq V$, the consensus time $\Tcons(A)$ is defined as
%\begin{align}
%\Tcons(A)\defeq \min\left\{t\geq 0\,:\,A^{(t)}\in\{\emptyset,V\},\,A^{(0)}=A\right\}. \label{eqn:Tcons_def}
%\end{align}
Let $A_i \defeq V_i\cap A$ for $A\subseteq V$ and $i=1,2$.
Since $|A_i'|=\sum_{v\in V_i}X_v$, the Hoeffding bound (\cref{lem:Hoeffding}) implies that the following holds w.h.p for $i=1,2$: 
\begin{align}
\left||A'_i|-\E[|A_i'|]\right|=O(\sqrt{n\log n}). \label{eqn:A'concentration}
\end{align}

\section{Concentration result for the stochastic block model}\label{sec:concentrationsbm}
In this paper, we consider the Best-of-three on the stochastic block model $G(2n,p,q)$ (\cref{def:SBM}). 
Then, $\E\left[|A'_i|\right]$ in \cref{eqn:A'concentration} is a random variable since $G(2n,p,q)$ is a random graph.
Here, our key ingredient is the following general concentration result for $G(2n,p,q)$. %of this kind of random variables for \emph{all} $A\subseteq V$.
\begin{definition}[$f$-good $G(2n,p,q)$]
For a given function $f:[0,1]\to [0,1]$, we say $G(2n,p,q)$ is $f$-good if $G(2n,p,q)$ satisfies the following properties.
\begin{enumerate}[label={\rm (P\arabic*)}]
\setlength{\itemsep}{-1.8pt}
\item It is connected and non-bipartite. \label{pro:connectednonbipartite}
\item A positive constant $C_1$ exists such that, for all $A, S\subseteq V$ and $i\in\{1,2\}$, 
\begin{align*}
\left| \sum_{v\in S\cap V_i}f\left(\frac{\deg_A(v)}{\deg(v)}\right) - |S\cap V_i|f\left( \frac{|A_i|p+|A_{3-i}|q}{n(p+q)} \right) \right| \leq  C_1\sqrt{\frac{n}{p}}. 
\end{align*}
\label{pro:gomiN}
\item A positive constant $C_2$ exists such that, for all $A\subseteq V$, $S\in\{A, V\setminus A, V\}$ and $i\in\{1,2\}$, 
\begin{align*}
\sum_{v\in S\cap V_i}f\left(\frac{\deg_A(v)}{\deg(v)}\right) \leq 
|S\cap V_i|f\left( \frac{|A_i|p+|A_{3-i}|q}{n(p+q)} \right) + C_2|A|\sqrt{\frac{\log n}{np}}.
\end{align*}
\label{pro:gomiA}
\end{enumerate}
\end{definition}
\begin{theorem}[Main technical theorem]
\label{thm:fgood}
Suppose that $f:[0,1]\to [0,1]$ is a polynomial function with constant degree, $p=\omega(\log n/n)$ and $q\geq \log n/n^2$. 
Then $G(2n,p,q)$ is $f$-good w.h.p.
\end{theorem}
Note that the proof of \ref{pro:connectednonbipartite} is not difficult since $p=\omega(\log n/n)$ and $q\geq \log n/n^2$~\cite{FK16}.
Proving \ref{pro:gomiN} and \ref{pro:gomiA}, however, is more challenging: we show these in \cref{sec:randomgraphs}.
%Compare to \ref{pro:gomiN}, \cref{pro:gomiA} gives better upper bound if $|A|$ is sufficiently small. This property plays a key role in the proof of \cref{prop:fastconsensus_3M}.
%Since $f^{\Bthree}$ is a polynomial function with three degree, 

From \cref{thm:fgood}, $G(2n,p,q)$ is $f^{\Bthree}$-good w.h.p.
Hence, we consider the Best-of-three on an {\em $f^{\Bthree}$-good} $G(2n,p,q)$.
From \ref{pro:gomiN} and \ref{pro:gomiA}, we have
\begin{align}
\E[|A_i'|]
&=\sum_{v\in V_i}f^{\Bthree}\left(\frac{\deg_A(v)}{\deg(v)}\right)
=nf^{\Bthree}\left(\frac{|A_i|p+|A_{3-i}|q}{n(p+q)} \right) 
\begin{cases}
\pm O\left(\sqrt{\frac{n}{p}}\right)\\
+O\left(|A|\sqrt{\frac{\log n}{np}}\right)
\end{cases}
\label{thm:mainthm_E}%thm:mainthm_E
\end{align}
for all $A\subseteq V$ and $i=1,2$.
Here, we remark that \ref{pro:gomiA} is stronger than \ref{pro:gomiN} if $|A|$ is sufficiently small.
This property will play a key role in the proof of \cref{prop:fastconsensus_3M}.

%{変更}
\paragraph*{Idea of the proof of \cref{thm:fgood}.}
%\noindent \textbf{Idea of the proof of \cref{thm:fgood}} (see \cref{sec:randomgraphs} for details). 
%\begin{proof}[Proof of \cref{thm:fgood} (sketch, see \cref{sec:randomgraphs} for details)]
We consider the property \ref{pro:gomiN}.
Note that we may assume $f(x)=x^k$ for some constant $k$ w.l.o.g.~since it suffices to obtain the concentration result for each term of $f$.
For simplicity, let us exemplify our idea on the special case of $k=3$. % and $p=q$.
It is known that $\deg(v)=n(p+q)\pm O(\sqrt{np\log n})$ holds for all $v\in V$ w.h.p.~(see, e.g.~\cite{FK16}).
This implies that $\sum_{v\in S}\left(\frac{\deg_A(v)}{\deg(v)}\right)^3 = \frac{1\pm O(\sqrt{\log n/np})}{(n(p+q))^3}\cdot \sum_{v\in S}\deg_A(v)^3$ holds for all $S,A\subseteq V$.
Indeed, it is not difficult to see that the term $O(\sqrt{\log n/np})$ can be improved to $O(\sqrt{1/np})$ (see \cref{sec:reduction}).

The core of the proof is the concentration of $\sum_{v\in S}\deg_A(v)^3$.
Note that $\sum_{v\in S}\deg_A(v)=\sum_{v\in S}\sum_{a\in A}\mathbbm{1}_{\{s,a\}\in E}$ counts the number of cut edges between $S$ and $A$.
For fixed $S$ and $A$, the Chernoff bound yields the concentration of it since each edge appears independently.
Similarly, the summation $\sum_{v\in S}\deg_A(v)^3=\sum_{v\in S}\sum_{a,b,c\in A}\mathbbm{1}_{\{v,a\},\{v,b\},\{v,c\}\in E}$ counts the number of ``crossing stars" between $S$ and $A$.
However, the Chernoff bound does not work here due to the dependency of the appearance of crossing stars.
Fortunately, we can obtain a strong lower bound using the Janson inequality (\cref{lem:Janson}) as follows:
For $S,A,B,C\subseteq V$, let $W(S;A,B,C)\defeq \sum_{v\in S}\deg_A(v)\deg_B(v)\deg_C(v)$.
From the Janson inequality and the union bound on $S,A,B,C\subseteq V$, we can show that $W(S;A,B,C)\geq \E[W(S;A,B,C)]-O(n^{3.5}p^{2.5})$ holds for all $S,A,B,C\subseteq V$ w.h.p.
On the other hand, it is easy to check that
\begin{align*}
W(S;A,B,C) &= W(V;V,V,V) - W(V;V,V,V\setminus C) - W(V;V,V\setminus B,C) \\
 & \ \ \ - W(V;V\setminus A,B,C)-W(V\setminus S;A,B,C).
\end{align*}
The Kim-Vu concentration (\cref{lem:Kim-Vu}) yields $W(V;V,V,V)\leq \E[W(V;V,V,V)] + O(n^{3.5}p^{2.5})$ since we do not consider the union bound here.
For the other terms, we apply the lower bound by the Janson inequality.
Then, we have a strong concentration result that $\sum_{v\in S}\deg_A(v)^3 = W(S;A,A,A) =  \E[W(S;A,A,A)] \pm O(n^{3.5}p^{2.5})$ holds for all $S,A\subseteq V$ w.h.p.
Finally, we estimate the gap between $\E[W(S\cap V_i,A,A,A)]$ and $|S\cap V_i|(|A_i|p+|A_{3-i}|q)^3$. %{$\E[W(S;A,A,A)]$ and $\sum_{v\in S}(|A|p)^3$}.
See \cref{sec:randomgraphs} for details.
%\end{proof}
%
\section{Induced dynamical system}\label{sec:auxiliary_results}
Let $\alpha_i \defeq \frac{|A_i|}{n}$, $\alpha'_i\defeq \frac{|A'_i|}{n}$ and $r \defeq \frac{q}{p}$.
Suppose that $r$ is a constant.
Then, for an $f^{\Bthree}$-good $G(2n,p,q)$, it holds w.h.p.~that
\begin{align}
\left|\alpha_i'-f^{\Bthree}\left(\frac{\alpha_i+r\alpha_{3-i}}{1+r}\right)\right|= O\left(\sqrt{\frac{1}{np}}+\sqrt{\frac{\log n}{n}}\right)
\label{eq:goodalphai}
\end{align}
for all $A\subseteq V$ and $i=1,2$ since \cref{eqn:A'concentration,thm:mainthm_E} hold.
\if0
From \cref{thm:mainthm_E}, we have
\begin{align}
\E[\alpha_i'\mid A]
&= \frac{1}{n}\E[|A'_i| \mid A]  = f\left(\frac{\alpha_i+r\alpha_{3-i}}{1+r}\right) + O\left(\frac{1}{\sqrt{np}}\right) \label{eqn:Ealpha}
\end{align}
for all $A\subseteq V$.
\fi

Throughout this paper, we use ${\boldsymbol \alpha}=(\alpha_1,\alpha_2)$ and ${\boldsymbol \alpha}'=(\alpha'_1,\alpha'_2)$ as vector-valued random variables.
Equation \cref{eq:goodalphai} leads us to the dynamical system $H$, where we define $H:\mathbb{R}^2\to \mathbb{R}^2$ as
\begin{align}
H:\mathbf{a} \mapsto (H_1(\mathbf{a}),H_2(\mathbf{a})), \label{eqn:H_def}
\end{align}
and $H_i(a_1,a_2)\defeq f^{\Bthree}\left(\frac{a_i+ra_{3-i}}{1+r}\right)$.

By combining \cref{eq:goodalphai} with the Lipschitz condition (see \cref{sec:lipschitzcondition}), it is not difficult to show the following result; see \cref{sec:polynomialvoting} for the proof. 
\begin{theorem} \label{thm:approximation_vectorfield}
Consider the Best-of-three on an $f^{\Bthree}$-good $G(2n,p,q)$, starting with the vertex set $A^{(0)}\subseteq V$ holding opinion $1$.
%Let $(A^{(0)})_{t=0}^{\infty}$ be a sequence of random vertex subsets determined by $A^{(t+1)}\defeq (A^{(t)})'$ for each $t\geq 0$, 
Let $({\boldsymbol \alpha}^{(t)})_{t=0}^{\infty}$ be a stochastic process given by ${\boldsymbol \alpha}^{(t)}=(\alpha^{(t)}_1,\alpha^{(t)}_2)$ and $\alpha^{(t)}_i=|A^{(t)}\cap V_i|/n$.
%Consider a stochastic process $({\boldsymbol \alpha}^{(t)})_{t=0}^\infty $, where ${\boldsymbol \alpha}^{(t)}=(\alpha^{(t)}_1,\alpha^{(t)}_2)$.
Let $H$ be the mapping \cref{eqn:H_def} and define $(\mathbf{a}^{(t)})_{t=0}^\infty$ as %$\mathbf{a}^{(0)}={\boldsymbol \alpha}^{(0)}$ and $\mathbf{a}^{(t+1)}=H(\mathbf{a}^{(t)})$ for each $t\geq 0$. 
\begin{align}
\label{eqn:veca_def}
\begin{cases}
\mathbf{a}^{(0)}={\boldsymbol \alpha}^{(0)},\\
\mathbf{a}^{(t+1)}=H(\mathbf{a}^{(t)}).
\end{cases}
\end{align}
Then there exists a positive constant $C>0$ such that
\begin{align*}
\forall 0\leq t\leq n^{o(1)},\forall A^{(0)}\subseteq V\,:\,\Pr\left[\|{\boldsymbol \alpha}^{(t)}-\mathbf{a}^{(t)}\|_\infty \leq C^t \left(\frac{1}{\sqrt{np}}+\sqrt{\frac{\log n}{n}}\right)\right] \geq 1-n^{-\Omega(1)}.
\end{align*}
\end{theorem}
Broadly speaking, \Cref{thm:approximation_vectorfield} approximates the behavior of ${\boldsymbol \alpha}^{(t)}$ by the orbit $\mathbf{a}^{(t)}$ of the corresponding dynamical system $H$.
We call the mapping $H$ the \emph{induced dynamical system}.
Indeed, the same results as \cref{thm:mainthm_E} hold for the Best-of-two voting.
Therefore, analogous results of 
\cref{thm:approximation_vectorfield} hold, which enable us to analyze the Best-of-two on $G(2n,p,q)$ via its induced dynamical system.
\begin{figure}[t]
\centering
	\begin{tabular}{c}
		\begin{minipage}{0.5\hsize}
		\centering
			\subfigure[$r=1/6$]{\includegraphics[height=7cm]{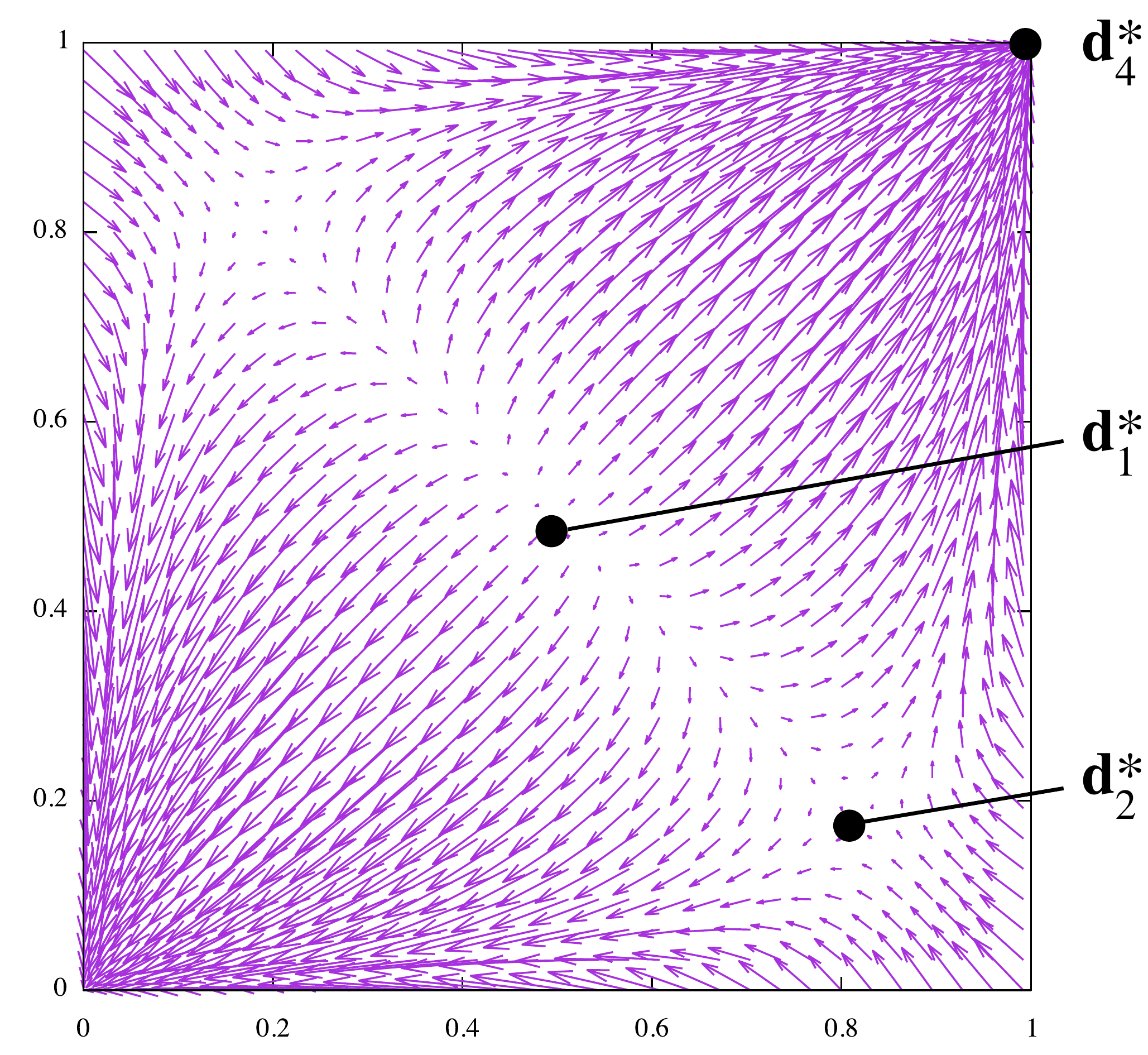}}
		\end{minipage}	
		\begin{minipage}{0.5\hsize}
		\centering
		   \subfigure[$r=1/9$]{\includegraphics[height=7cm]{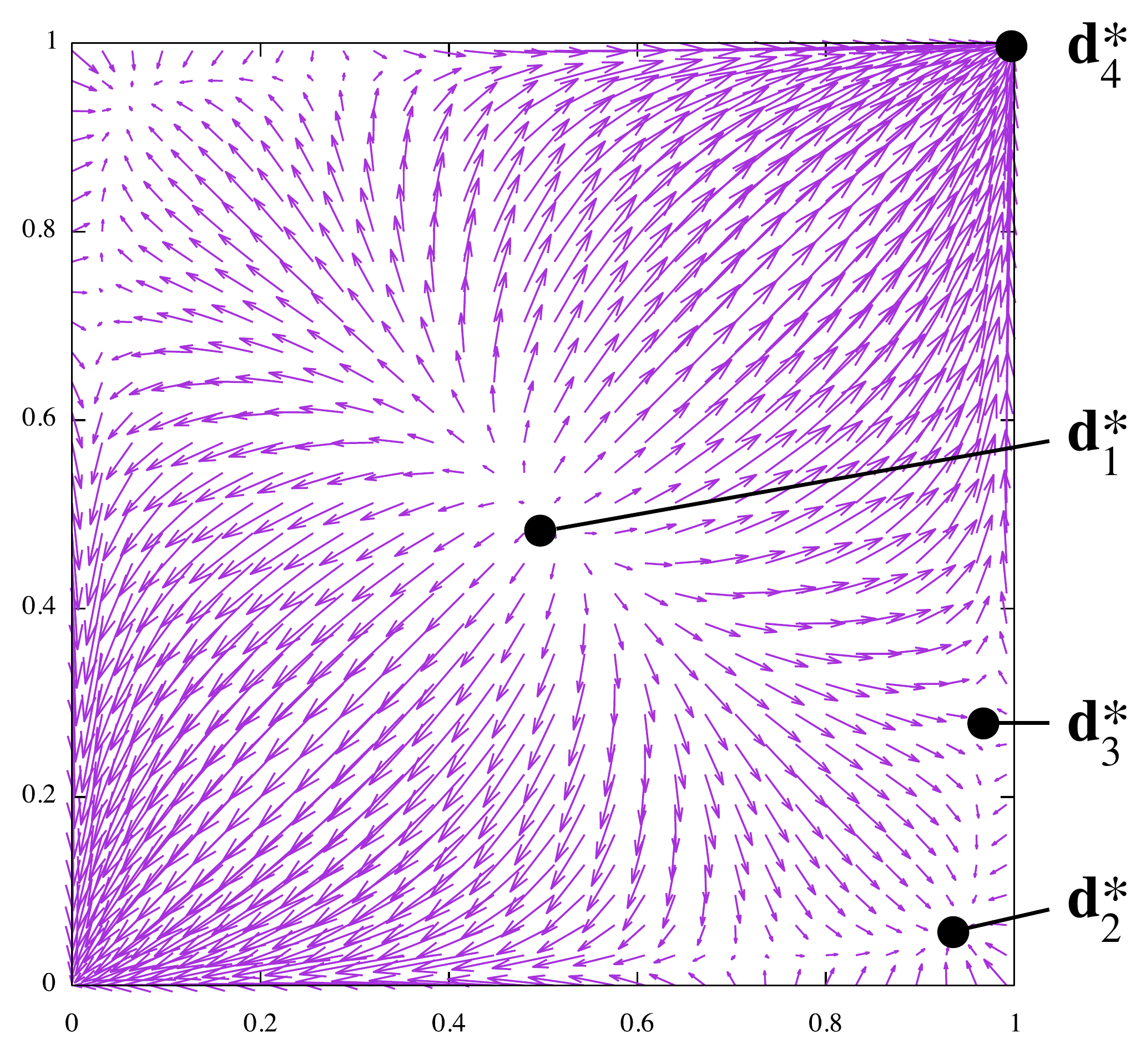}}
		\end{minipage}
	\end{tabular}
	\caption{The induced dynamical system $H$ of \cref{eqn:H_def}.
	The points $\mathbf{d}^*_i$ are the fixed points given in \cref{eqn:fixed_points_3M}.
	Here, the horizontal and vertical axes correspond to $\alpha_1$ and $\alpha_2$, respectively.
	We can observe two sink points in (b), but none in (a). \label{fig:3M_vectorfield}}
\end{figure}
The dynamical system $H$ of \cref{eqn:H_def} is illustrated in \cref{fig:3M_vectorfield}.
%
%\if0

To make the calculations more convenient, we change the coordinate of $H$ by
\begin{align*}
{\boldsymbol \delta}=(\delta_1, \delta_2)\defeq (\alpha_1-\alpha_2, \alpha_1+\alpha_2-1).
%&\delta_1\defeq \alpha_1-\alpha_2,\\
%&\delta_2\defeq \alpha_1+\alpha_2-1.
\end{align*}
Note that $\delta_1$ and $\delta_2$ axes are corresponding to the dotted lines of \cref{fig:3M_proof_sketch}.
%Note that $\delta_1,\delta_2$ are random variables.
Let $u\defeq\frac{1-r}{1+r}$. Then we have
\begin{align*}
\E[\delta'_i\mid A]=T_i(\delta_1,\delta_2)+O\left(\frac{1}{\sqrt{np}}\right),
\end{align*}
where
\begin{align}
%\bigl((T_1(d_1,d_2), T_2(d_1,d_2)\bigr)\defeq \left(\frac{ud_1}{2}\left(3-(ud_1)^2-3d_2^2\right),\frac{d_2}{2}\left(3-3(ud_1)^2-d_2^2\right) \right).
%&u\defeq\frac{1-r}{1+r},\nonumber\\
T_1(d_1,d_2)\defeq \frac{ud_1}{2}\left(3-(ud_1)^2-3d_2^2\right), \, %\label{eqn:3Mvecfielddef2}\\
T_2(d_1,d_2)\defeq\frac{d_2}{2}\left(3-3(ud_1)^2-d_2^2\right). \label{eqn:3Mvecfielddef}
\end{align}

This suggests another dynamical system $T(\mathbf{d})=(T_1(\mathbf{d}),T_2(\mathbf{d}))$.
Here, we use $\mathbf{d}=(d_1,d_2)$ as a specific point and ${\boldsymbol \delta}=(\delta_1,\delta_2)$ as a vector-valued random variable.
Consider ${\boldsymbol \delta^{(t)}}=(\delta^{(t)}_1,\delta^{(t)}_2)$ and $(\mathbf{d}^{(t)})_{t=0}^\infty$, where $\mathbf{d}^{(0)}={\boldsymbol \delta}^{(0)}$ and $\mathbf{d}^{(t+1)}=T(\mathbf{d}^{(t)})$ for each $t\geq 0$.
%\begin{align*}
%\begin{cases}
%\mathbf{d}^{(0)}={\boldsymbol \delta}^{(0)},\\
%\mathbf{d}^{(t+1)}=T(\mathbf{d}^{(t)}).
%\end{cases}
%\end{align*}
From \cref{thm:approximation_vectorfield}, it holds w.h.p.~that
\begin{align}
\|{\boldsymbol \delta}^{(t)}-\mathbf{d}^{(t)}\|_\infty \leq C^t \left(\frac{1}{\sqrt{np}}+\sqrt{\frac{\log n}{n}}\right) \label{eqn:delta_d_diff}
\end{align}
for sufficiently large constant $C>0$, any $0\leq t\leq n^{o(1)}$ and any initial configuration $A^{(0)}\subseteq V$.
For notational convenience, we use ${\boldsymbol \delta}'\defeq {\boldsymbol \delta}^{(t+1)}$ for ${\boldsymbol \delta}={\boldsymbol \delta}^{(t)}$.
Similarly, we refer $\mathbf{d}'$ to $T(\mathbf{d})$.
%\fi

Note that ${\boldsymbol \delta}$ satisfies $|\delta_1|+|\delta_2|\leq 1$. % since ${\boldsymbol \alpha}\in[0,1]^2$.
In addition, the dynamical system $T$ is symmetric: Precisely, $T_1(\pm d_1,\mp d_2) = \pm T_1(d_1,d_2)$ and $T_2(\pm d_1,\mp d_2)=\mp T_2(d_1,d_2)$ hold.
In \cref{lem:3Mvector_field_lem1}, we assert that the sequence $(\mathbf{d}^{(t)})_{t=0}^\infty$ is closed in
\begin{align*}
S\defeq \{(d_1,d_2)\in[0,1]^2:d_1+d_2\leq 1\}.
\end{align*}

From now on, we focus on $S$ and consider the behavior of ${\boldsymbol \delta}$ around fixed points.
A straightforward calculation shows that $\mathbf{d}'=\mathbf{d}\in S$ if and only if $\mathbf{d}\in\{\mathbf{d}^*_1,\mathbf{d}^*_2,\mathbf{d}^*_3,\mathbf{d}^*_4\}$, where
\begin{align} \label{eqn:fixed_points_3M}
\mathbf{d}^*_i \defeq \begin{cases}
(0,0) & \text{if $i=1$},\\
\left(\sqrt{\frac{3u-2}{u^3}},0\right) & \text{if $i=2$ and $u\geq\frac{2}{3}$},\\
\left(\sqrt{\frac{1}{4u^3}},\sqrt{\frac{4u-3}{4u}} \right) & \text{if $i=3$ and $u\geq\frac{3}{4}$},\\
(0,1) & \text{if $i=4$}.
\end{cases}
\end{align}%

%{section Auxiliary resultsを統合}% \subsection{Auxiliary results}\label{sec:auxiliary_results}}
Here, we provide auxiliary results needed for the proofs of \cref{thm:phasetransition_3M,thm:worst_case_3M}.
The proofs of these results are presented in \cref{sec:SJacobian,sec:generaljacobi}. %will be shown in later sections.
For $\mathbf{x}\in\mathbb{R}^2$ and $\epsilon>0$, let $B(\mathbf{x},\epsilon)=\{\mathbf{y}\in\mathbb{R}^2\,:\,\|\mathbf{x}-\mathbf{y}\|_\infty < \epsilon\}$ be the open ball.
For $\mathbf{d}=(d_1,d_2)\in\mathbb{R}^2$, let $\absp{\mathbf{d}}\defeq (|d_1|,|d_2|) \in \mathbb{R}^2$.
%{2.5+2.6, 2.7-2.9でまとめる?}
\begin{lemma}[$S$ is closed] \label{lem:3Mvector_field_lem1}
For any $\mathbf{d}\in S$, it holds that $\mathbf{d}'\in S$.
\end{lemma}

\begin{proposition}[Orbit convergence]\label{prop:3Mvectorfield}
For any sequence $(\mathbf{d}^{(t)})_{t=0}^\infty$, $\lim_{t\to\infty} \absp{\mathbf{d}^{(t)}}=\mathbf{d}^*_i$ for some $i\in\{1,2,3,4\}$.
In addition, if $u<\frac{3}{4}$ and a positive constant $\kappa>0$ exists such that the initial point $\mathbf{d}^{(0)}=(d^{(0)}_1,d^{(0)}_2)\in S$ satisfies $|d^{(0)}_2|>\kappa$, then $\lim_{t\to\infty} \absp{\mathbf{d}^{(t)}} =\mathbf{d}^*_4$.
\end{proposition}

\begin{proposition}[Dynamics around $\mathbf{d}^*_2$] \label{prop:sink_3M}
Consider the Best-of-three on an $f^{\Bthree}$-good $G(2n,p,q)$ such that $r=q/p<1/7$ is a constant.
Then there exists a positive constant $\epsilon=\epsilon(r)$ satisfying
\begin{align*}
\Pr\left[\absp{{\boldsymbol \delta'}}\not\in B(\mathbf{d}_2^*,\epsilon) \,\middle|\, \absp{{\boldsymbol \delta}}\in B(\mathbf{d}_2^*,\epsilon)\right] \leq \exp(-\Omega(n)).
\end{align*}

In particular, 
%let $\tau=\inf\{t:{\boldsymbol \delta}^{(t)}\not\in B(\mathbf{d}^*_2,\epsilon),{\boldsymbol \delta}^{(0)}\in B(\mathbf{d}^*_2,\epsilon)\}$ be the stopping time.
%Then, 
$\Tcons(A) = \exp(\Omega(n))$ w.h.p.~for any $A$ satisfying $\absp{{\boldsymbol \delta}} \in B(\mathbf{d}_2^*,\epsilon)$.
\end{proposition}

\begin{proposition}[Towards consensus] \label{prop:fastconsensus_3M}
Consider the Best-of-three on an $f^{\Bthree}$-good $G(2n,p,q)$ such that $r=q/p$ is a constant.
Then, there exists a universal constant $\epsilon=\epsilon(r)>0$ satisfying the following: $\Tcons(A) \leq O(\log\log n+\log n/\log (np))$ holds w.h.p.~for all $A\subseteq V$ with $\min\{|A|,2n-|A|\}\leq \epsilon n$.
\end{proposition}

\begin{proposition}[Escape from fixed points] \label{prop:escape_3M}
Consider the Best-of-three on an $f^{\Bthree}$-good $G(2n,p,q)$ such that $p$ and $q$ are constants.
%Suppose that $p$ and $q$ are constants and let $r=q/p$.
If $q/p>1/7$ and $|\delta_2^{(0)}|=o(1)$, then it holds w.h.p.~that $|\delta_2^{(\tau)}|>\kappa$ for some $\tau=O(\log n)$ and some constant $\kappa>0$.
\end{proposition}
%
%\noindent \textbf{Intuitive explanation for \cref{prop:sink_3M,prop:fastconsensus_3M,prop:escape_3M}. }
\paragraph*{Intuitive explanations for \cref{prop:sink_3M,prop:fastconsensus_3M,prop:escape_3M}.}
In \cref{prop:sink_3M,prop:fastconsensus_3M,prop:escape_3M}, we consider the behavior of ${\boldsymbol \alpha}^{(t)}$ around the fixed points \cref{eqn:fixed_points_3M}.
Let $H$ be the induced dynamical system and let $J$ be the Jacobian matrix of $H$ at a fixed point $\mathbf{a}^*$ with two eigenvalues $\lambda_1,\lambda_2$.
If the eigenvectors are linearly independent, we can rewrite $J$ as $J=U^{-1}\Lambda U$, where $\Lambda\defeq\mathrm{diag}(\lambda_1,\lambda_2)$ and $U$ is some nonsingular matrix.
Let ${\boldsymbol \beta}\defeq U({\boldsymbol \alpha}-\mathbf{a}^*)$.
Roughly speaking, if ${\boldsymbol \alpha}$ is closed to $\mathbf{a}^*$, the Taylor expansion at $\mathbf{a}^*$ (i.e.~$H({\boldsymbol \alpha})\approx \mathbf{a}^*+ J({\boldsymbol \alpha}-\mathbf{a}^*)$) yields
\begin{align*}
\E[{\boldsymbol \beta}' \mid A] = U(\E[{\boldsymbol \alpha}'\mid A]-\mathbf{a}^*) \approx U(H({\boldsymbol \alpha})-\mathbf{a}^*) \approx \Lambda {\boldsymbol \beta}.
\end{align*}
In other words, $\beta'_i\approx \lambda_i\beta_i$.
If $\max\{|\lambda_1|,|\lambda_2|\}<1-c$ for some constant $c>0$, we might expect that  $\|{\boldsymbol \beta}\|=\Theta(\|{\boldsymbol \alpha}-\mathbf{a}^*\|)$ is likely to keep being small.
Here, we do not restrict this argument on the Best-of-three.
We will prove \cref{prop:sinkpoint}, which is a generalized version of \cref{prop:sink_3M}.
If $\max\{|\lambda_1|,|\lambda_2|\}>1+c$ for some constant $c>0$, the norm $\|{\boldsymbol \beta}\|$ seems to become large in a small number of steps.
We will exploit this insight and prove \cref{prop:escape}, which immediately implies \cref{prop:escape_3M}.
Indeed, for consensus areas (i.e.~$\mathbf{a}^*\in\{(0,0),(1,1)\}$), the induced dynamical systems of the Best-of-three and the Best-of-two satisfy $\lambda_1=\lambda_2=0$.
Then, the Taylor expansion yields $\|{\boldsymbol \alpha}' - \mathbf{a}^*\| \approx O(\|{\boldsymbol \alpha}-\mathbf{a}^*\|^2)$.
This observation and the property~\ref{pro:gomiA} lead to the proof of \cref{prop:fastconsensus_Jacob} as well as \cref{prop:fastconsensus_3M}.
\section[Derive Theorem 1.2 and 1.4]{Derive \cref{thm:phasetransition_3M,thm:worst_case_3M}}\label{sec:proofofmainthm}
Here, we prove \cref{thm:phasetransition_3M,thm:worst_case_3M} using \cref{prop:3Mvectorfield,prop:sink_3M,prop:fastconsensus_3M,prop:escape_3M}.

\begin{proof}[Proof of \cref{thm:phasetransition_3M}]
If $r>\frac{1}{7}$ and $A^{(0)}\subseteq V$ satisfies $\bigl| |A^{(0)}|-n \bigr| = \Omega(n)$, then we have $|d^{(0)}_2|=|\delta^{(0)}_2|>\kappa$ for some constant $\kappa>0$.
Next, for any constant $\epsilon>0$, \cref{prop:3Mvectorfield} implies $\absp{\mathbf{d}^{(l)}} \in B(\mathbf{d}^*_4,\epsilon)$ for some constant $l=l(\epsilon)$.
From \cref{eqn:delta_d_diff}, we have $\absp{{\boldsymbol \delta}^{(l)}}\in B(\mathbf{d}^*_4,\epsilon)$ for sufficiently large $n$. 
Set $\epsilon$ be the constant mentioned in \cref{prop:fastconsensus_3M}. 
Then, from \cref{prop:fastconsensus_3M}, it holds w.h.p.~that $\Tcons(A^{(0)}) \leq l+\Tcons(A^{(l)}) \leq O(\log\log n+\log n/\log (np))$.

If $r<\frac{1}{7}$, \cref{prop:sink_3M} yields $\Tcons(A^{(0)}) \geq \exp(\Omega(n))$ w.h.p.~for any $A^{(0)}\subseteq V$ with ${\boldsymbol \delta}^{(0)} \in B(\mathbf{d}^*_2,\epsilon)$, where $\epsilon>0$ is the constant from \cref{prop:sink_3M}.
This completes the proof of \ref{state:belowthroshold_3M}.
\end{proof}

\begin{proof}[Proof of \cref{thm:worst_case_3M}]
If $|\delta^{(0)}|=o(1)$, then \cref{prop:escape_3M} yields that $|\delta^{(\tau)}|>\kappa$ for some constant $\kappa>0$ and some $\tau=O(\log n)$.
Then, from \cref{thm:phasetransition_3M}, we have $\Tcons(A^{(\tau)})\leq O(\log\log n+\log n/\log (np))$.
Thus, $\Tcons(A^{(0)})\leq \tau + \Tcons(A^{(\tau)}) \leq O(\log n)$.
\end{proof}
%
%{Lemma 5.2-5.5 to Lemma 5.1 to Theorem fgood(2), Lemma (5.5) and 5.6 to Theorem fgood(3)}
%
\begin{figure}[h]
\begin{center}
\begin{tikzpicture}[
           normal node/.style={
               draw,
               text height=0.25cm,
%               text depth=.0ex,
%               text width=13em,
               text centered
			},
           prop node/.style={
               draw,
               text height=0.2cm,
               text width=7em,               
%               text depth=.0ex,
%               text width=13em,
               text centered
			},			
           tool node/.style={
               draw,
               text width=13em,
               text centered,
               text height=0.25cm,
%               text depth=.0ex,
%               text width=13em,
               text centered
			}]
	\def\D{0.2};

  \node[normal node] (worst_case_3M) at (0cm, 0cm) {\cref{thm:worst_case_3M,thm:worst_case_2C}};

  \node[normal node, below = \D of worst_case_3M] (escape_3M) {\cref{prop:escape_3M,prop:escape_2C}};
%    \node[prop node, below = \D of worst_case_3M] (escape_3M) {\cref{prop:escape_3M,prop:escape_2C}};
  \node[normal node, right=\D of escape_3M] (sink_3M) {\cref{prop:sink_3M,prop:sink_2C}};
%  \node[prop node, right=\D of escape_3M] (sink_3M) {\cref{prop:sink_3M,prop:sink_2C}};
  \node[normal node, right=\D of sink_3M] (fastconsensus_3M) {\cref{prop:fastconsensus_3M,prop:fastconsensus_2C}};  
%  \node[prop node, right=\D of sink_3M] (fastconsensus_3M) {\cref{prop:fastconsensus_3M,prop:fastconsensus_2C}};  
  \node[normal node, above=\D of fastconsensus_3M] (phasetransition_3M) {\cref{thm:phasetransition_3M,thm:phasetransition_2C}};  
  \node[normal node, right=0.5cm of fastconsensus_3M] (3Mvectorfield) {\cref{prop:3Mvectorfield,prop:2Cvectorfield}};
%  \node[prop node, right=1cm of fastconsensus_3M] (3Mvectorfield) {\cref{prop:3Mvectorfield,prop:2Cvectorfield}};
  
  \node[normal node, below = \D of escape_3M] (escape) {\cref{prop:escape}}; 
  \node[normal node, below = \D of sink_3M] (sinkpoint) {\cref{prop:sinkpoint}};   
  \node[normal node, below = \D of fastconsensus_3M.south west, anchor = north west] (fastconsensus) {\cref{prop:fastconsensus_Jacob}};

  \node[normal node, below = \D of fastconsensus] (fgood) {\cref{thm:fgood,thm:approximation_vectorfield_general}};
  \node[normal node, below = \D of fgood] (keylemma) {\cref{lem:WIdiscrepancy,lem:WEWdiscrepancy,lem:EWIdiscrepancy,lem:technique,lem:maxmindegree,lem:WIdiscrepancy2}};
%  \node[normal node, left = 1cm of keylemma] (Wlemma) {\cref{lem:WEWdiscrepancy,lem:EWIdiscrepancy,lem:technique,lem:maxmindegree}};

  \node[draw, text width=11.7em, left = 1 of keylemma] (probtools) {The Janson inequality and the Kim-Vu concentration};
  \node[draw, below = 1.6 of 3Mvectorfield] (competitive) {Competitive dynamical systems};

  \coordinate[right=0.5cm] (D) at (phasetransition_3M.east);
%  \coordinate[above = \D+0.125] (E) at (phasetransition_3M.east);  
  \coordinate[above = \D+0.25] (E) at (3Mvectorfield.north);
  \coordinate[below = \D+0.5] (A) at (sinkpoint);
  \coordinate[below = \D+0.5] (B) at (escape);
  \coordinate[below = 2.15] (C) at (D);  

  \draw[->] (phasetransition_3M) to (worst_case_3M);
  \draw[->] (3Mvectorfield) -- (E) -- (phasetransition_3M);
  \draw[->] (escape_3M) to (worst_case_3M);
  \draw[->] (sink_3M) to (phasetransition_3M.south west); 
  \draw[->] (fastconsensus_3M) to (phasetransition_3M);
    
  \draw[->] (sinkpoint) to (sink_3M);
  \draw[->] (escape) to (escape_3M); 
  \draw[->] (fastconsensus) to (fastconsensus_3M); 
  
  \draw[->] (A) -- (B) -- (escape.south);
%  \draw[->] (fgood.west) -- +(-3.8,0) -- (sinkpoint.south);
  \draw[->] (fgood.west) -- (A) -- (sinkpoint.south);
  \draw[->] (fgood.north) to (fastconsensus.south);
  \draw[-] (fgood.east) -- (C) -- (D);
  \draw[->] (keylemma) to (fgood);
%  \draw[->] (Wlemma) to (keylemma);
  
  \draw[->] (probtools) to (keylemma);
  \draw[->] (competitive.north) -- (3Mvectorfield.south);
\end{tikzpicture}
\caption{The organization of our proofs. \label{fig:organization}}
\end{center}
\end{figure}
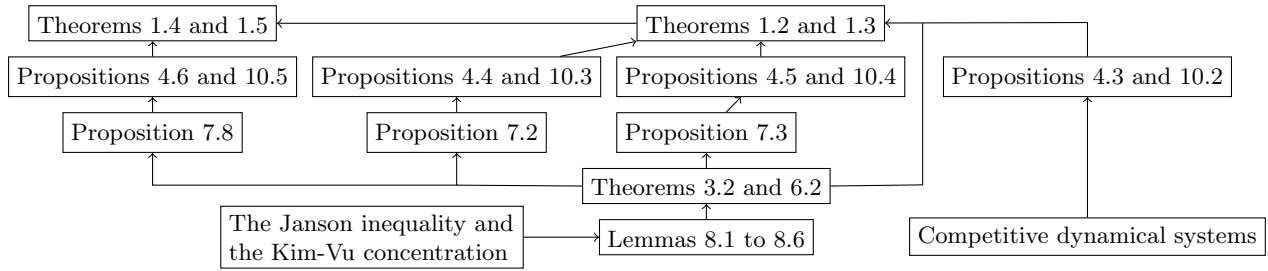
%

%For example, consider {\em Best of $2k+1$} voting process for a positive constant $k$.
%This process is defined by $f_1=f_2=f$, where $f(x)\defeq \Pr[{\rm Bin}(2k+1,x)\geq k+1]=%\sum_{i=k+1}^{2k+1}{2k+1\choose i}x^i(1-x)^{2k+1-i}$.
%Since $f$ is a polynomial function with degree $2k+1$, \cref{thm:approximation_vectorfield_general} guarantees $\|{\boldsymbol \alpha}^{(t)}-\mathbf{a}^{(t)}\|_\infty\leq C^t(\sqrt{1/np}+\sqrt{\log n/n})$ for this process.

\section{Polynomial voting processes
\label{sec:polynomialvoting}}
Using \cref{thm:fgood}, we can prove the same results as \cref{thm:approximation_vectorfield} for various models including the Best-of-two.
Hence, in this paper, we do not restrict our interest to the Best-of-three: Instead, we prove general results that hold for \emph{polynomial voting process} on $G(2n,p,q)$.
%%%%

%Formally, a consensus model $\mathcal{M}$ on two opinions is specified by two functions $f_1,f_2$, which satisfies
%\begin{align}
%\Pr[v\in A'\mid A,\text{$v$ has opinion $i$}] = f_i\left(\frac{\deg_A(v)}{\deg(v)}\right). \label{eqn:f_def}
%\end{align}
%\begin{definition}[polynomial voting process]
%A consensus model $\mathcal{M}$ is \emph{proper} if the functions $f_1(x),f_2(x)$ of \cref{eqn:f_def} are polynomials.
%\end{definition}
\begin{definition}[$(f_1,f_2)$-polynomial voting process]
Let $G=(V,E)$ be a graph where each vertex holds an opinion from $\{1,2\}$.
Let $f_1,f_2: [0,1]\to [0,1]$ be polynomials.
For the set $A$ of vertices with opinion $1$,
let $A'$ denote the set of vertices with opinion $1$ after an update.
In the $(f_1,f_2)$-polynomial voting process, $A'= \{v\in V: X_v=1\}$ where $(X_v)_{v\in V}$ are independent binary random variables satisfying
\begin{align*}
\Pr[X_v=1]&=
\begin{cases}
f_1\left(\frac{\deg_A(v)}{\deg(v)}\right) &(\textrm{$v\in A$, i.e.~$v$ has opinion $1$})\\
f_2\left(\frac{\deg_A(v)}{\deg(v)}\right) &(\textrm{$v\in V\setminus A$, i.e.~$v$ has opinion $2$})
\end{cases}.
\end{align*}
\end{definition}
In other words, 
\begin{align*}
\Pr[v\in A'\mid A,\text{$v$ has opinion $i$}] = f_i\left(\frac{\deg_A(v)}{\deg(v)}\right) %\label{eqn:f_def}
\end{align*}
for $i=1,2$.
Polynomial voting process includes several known voting models including the Best-of-two, the Best-of-three, and so on.
For example, $f_1(x)=f_2(x)=f^{\Bthree}(x)=3x^2-2x^3$ for the Best-of-three.
For the Best-of-two, $f_1(x)=2x(1-x)$ and $f_2(x)=x^2$.
%For a given vertex subset $A^{(0)}\subseteq V$, 
%
We can define induced dynamical system for any polynomial voting process on $G(2n,p,q)$ via the following result:
\begin{theorem}[\cref{thm:approximation_vectorfield} for polynomial voting processes]\label{thm:approximation_vectorfield_general}
Let $f_1$ and $f_2$ be polynomials with constant degree.
Consider an $(f_1,f_2)$-polynomial voting process, on an $f_1$-good and $f_2$-good $G(2n,p,q)$ starting with vertex set $A^{(0)}\subseteq V$ of opinion $1$.
Let $(A^{(t)})_{t=0}^{\infty}$ be a sequence of random vertex subsets defined by $A^{(t+1)}\defeq (A^{(t)})'$ for each $t\geq 0$.
Let $({\boldsymbol \alpha}^{(t)})_{t=0}^{\infty}$, where ${\boldsymbol \alpha}^{(t)}=(\alpha^{(t)}_1,\alpha^{(t)}_2)$ and $\alpha^{(t)}_i=|A^{(t)}\cap V_i|/n$.
Define a mapping $H=(H_1,H_2)$ as
\begin{align*}
H_i(a_1,a_2)=a_i f_1\left(\frac{a_i+ra_{3-i}}{1+r}\right)
+
(1-a_i) f_2\left(\frac{a_i+ra_{3-i}}{1+r}\right) \,\,\,\text{for $i=1,2$}.
\end{align*}
%Let $(\mathbf{a}^{(t)})_{t=0}^\infty$ be a point sequence given by \cref{eqn:veca_def}.
Define $(\mathbf{a}^{(t)})_{t=0}^\infty$ as $\mathbf{a}^{(0)}={\boldsymbol \alpha}^{(0)}$ and $\mathbf{a}^{(t+1)}=H(\mathbf{a}^{(t)})$ for each $t\geq 0$. 
Then, %$G(2n,p,q)$ w.h.p.~satisfies the following:
there exists a constant $C>0$ such that
\begin{align*}
\forall 0\leq t\leq n^{o(1)},\forall A^{(0)}\subseteq V\,:\,\Pr\left[\|{\boldsymbol \alpha}^{(t)}-\mathbf{a}^{(t)}\|_\infty \leq C^t \left(\frac{1}{\sqrt{np}}+\sqrt{\frac{\log n}{n}}\right)\right] \geq 1-n^{-\Omega(1)}.
\end{align*}
\end{theorem}
Remark that the mapping $H$ of \cref{thm:approximation_vectorfield_general} is the induced dynamical system.
%{\subsection{Proof of the approximation theorem (\cref{thm:approximation_vectorfield_general})}\label{sec:approximation_vectorfield}}
\begin{proof}%[Proof of \cref{thm:approximation_vectorfield_general}]
For any polynomial voting process, the cardinality $|A_i'|$ can be written as the sum of independent random variables:
\begin{align*}
|A_i'|=\sum_{v\in V_i}X_v.
\end{align*}
Thus, if we fix $A\subseteq V$, the Hoeffding bound (\cref{lem:Hoeffding}) implies that \cref{eqn:A'concentration} holds w.h.p.
Since
\begin{align*}
\E[|A_i'|]
&=\sum_{v\in V_i}\E[X_v]
=\sum_{v\in A_i}f_1\left(\frac{\deg_A(v)}{\deg(v)}\right)+\sum_{v\in V\setminus A_i}f_2\left(\frac{\deg_A(v)}{\deg(v)}\right), 
\end{align*}
the property~\ref{pro:gomiN} and \cref{eqn:A'concentration} lead to
\begin{align*}
\| {\boldsymbol \alpha'}- H({\boldsymbol \alpha}) \|_\infty \leq C_1\left(\frac{1}{\sqrt{np}}+\sqrt{\frac{\log n}{n}}\right)
\end{align*}
for some constant $C_1>0$.

Note that the function $H$ satisfies the Lipschitz condition.
Hence, a positive constant $C_2$ exists such that
\begin{align*}
\|H(\mathbf{x})-H(\mathbf{y})\|_\infty \leq C_2 \|\mathbf{x}-\mathbf{y}\|_\infty
\end{align*}
holds for any $\mathbf{x},\mathbf{y}\in[0,1]^2$ (see \cref{sec:lipschitzcondition}).
Let ${\boldsymbol \alpha}^{(t)}=(\alpha^{(t)}_1,\alpha^{(t)}_2)$ be the vector-valued stochastic process and $\mathbf{a}^{(t)}=(\mathbf{a}^{(t)}_1,\mathbf{a}^{(t)}_2)$ be the vector sequence given in \cref{eqn:veca_def}.
Then, we have
\begin{align*}
\|{\boldsymbol \alpha}^{(t)}-\mathbf{a}^{(t)}\|_\infty&=
\|{\boldsymbol \alpha}^{(t)} - H({\boldsymbol \alpha}^{(t-1)}) + H({\boldsymbol \alpha}^{(t-1)}) - H(\mathbf{a}^{(t-1)}) \|_\infty \\
&\leq
\|{\boldsymbol \alpha}^{(t)} - H({\boldsymbol \alpha}^{(t-1)}) \|_\infty + C_2\|{\boldsymbol \alpha}^{(t-1)} - \mathbf{a}^{(t-1)} \|_\infty \\
&\leq
C_2\|{\boldsymbol \alpha}^{(t-1)} - \mathbf{a}^{(t-1)} \|_\infty + C_1\left(\frac{1}{\sqrt{np}}+\sqrt{\frac{\log n}{n}}\right) \\
&\leq C^t\left(\frac{1}{\sqrt{np}}+\sqrt{\frac{\log n}{n}}\right),
\end{align*}
where $C$ is sufficiently large constant.
\end{proof}

\section{Results of general induced dynamical systems with applications to the Best-of-three}
\label{sec:SJacobian}
Now let us focus on the orbit $({\boldsymbol \alpha}^{(t)})_{t=1}^\infty$ such that $H({\boldsymbol \alpha}^{(0)})={\boldsymbol \alpha}^{(0)}$ holds, where $H$ is the induced dynamical system.
In this case, \cref{thm:approximation_vectorfield_general} does not provide enough information about the dynamics.
In dynamical system theory, a natural approach for the local behavior around fixed points is to consider the \emph{Jacobian matrix}.
Recall that, the Jacobian matrix $J$ of a function
$H:\mathbf{x} \mapsto (H_1(\mathbf{x}),H_2(\mathbf{x}))$ at $\mathbf{a}\in \mathbb{R}^2$ is a $2 \times 2$ matrix given by
\begin{align*}
J= \left(\frac{\partial H_i}{\partial x_j}(\mathbf{a})\right)_{i,j\in [2]}.
\end{align*}
In the following subsections, we will investigate the local dynamics from the viewpoint of the maximum singular value and eigenvalue of the Jacobian matrix.

In contrast to the local dynamics, it is quite difficult to predicate the orbit of general dynamical systems since some of them exhibits so-called chaos phenomenon.
Therefore, the proof of the orbit convergence (e.g.~\cref{prop:3Mvectorfield}) is not trivial.
Fortunately, the induced dynamical system of the Best-of-three on $G(2n,p,q)$ is \emph{competitive}, a well-known nice property for predicting the future orbit~\cite{HS05}~(see \cref{sec:competitive_system} for definition).
In \cref{sec:apply_to_3M}, we use known results of competitive dynamical systems to show \cref{prop:3Mvectorfield}.
It should be noted that the same argument leads to the orbit convergence for the Best-of-two as we shall discuss in \cref{sec:2Choices}.
\subsection{Sink point}
We begin with defining the notion of sink points.
Recall that the singular value of a matrix $A$ is the positive square root of the eigenvalue of $A^{\top}A$ (see \cref{sec:linearalgebra} for formal definition and basic properties).

\begin{definition}[sink point]
Consider a dynamical system $H$.
A fixed point $\mathbf{a}^*\in\mathbb{R}^2$ is \emph{sink} if the Jacobian matrix $J$ at $\mathbf{a}^*$ satisfies $\sigma_{\max}<1$, where $\sigma_{\max}$ is the largest singular value of $J$.
\end{definition}

\begin{proposition} \label{prop:sinkpoint}
Consider an $(f_1,f_2)$-polynomial voting process on an $f_1$-good and $f_2$-good $G(2n,p,q)$ such that $r=\frac{q}{p}$ is a constant.
Let $H$ be the induced dynamical system.
Then, for any sink point $\mathbf{a}^*$ and any sufficiently small $\epsilon=\omega(\sqrt{1/np})$,
\begin{align*}
\Pr\left[{\boldsymbol \alpha}'\not\in B(\mathbf{a}^*,\epsilon)\,\middle|\,{\boldsymbol \alpha}\in B(\mathbf{a}^*,\epsilon)\right]\leq \exp(-\Omega(\epsilon^2 n))
\end{align*}
holds.

In particular, let 
\begin{align*}
\tau:=\inf\left\{t\in\mathbb{N}:{\boldsymbol \alpha}^{(t)}\not\in B(\mathbf{a}^*,\epsilon)\right\}
\end{align*}
be a stopping time.
Then, $\tau\geq\exp(\Omega(\epsilon^2 n))$ holds w.h.p.~conditioned on ${\boldsymbol \alpha}^{(0)}\in B(\mathbf{a}^*,\epsilon)$ for any $\epsilon$ satisfying $\epsilon=\omega(\max\{1/\sqrt{np},\sqrt{\log n/n}\})$.
\end{proposition}

\subsection{Fast consensus}
Suppose that the Jacobian matrix at the consensus point (i.e.~$\mathbf{\alpha}\in\{(0,0),(1,1)\}$ is the all-zero matrix.
Then, we claim that the polynomial voting process reaches consensus within a small number of iterations if the initial set $A^{(0)}$ has small size.

\begin{proposition} \label{prop:fastconsensus_Jacob}
Consider an $(f_1,f_2)$-polynomial voting process on an $f_1$-good and $f_2$-good $G(2n,p,q)$ such that $\frac{p}{q}$ is a constant.
Suppose that the Jacobian matrix at the point ${\boldsymbol \alpha}=(0,0)$ is the all-zero matrix.

Then, there exists a constant $C_1,C_2,\delta>0$ such that
\begin{align*}
\Pr\left[\Tcons(A)\leq C_1\left(\log\log n+\frac{\log n}{\log np}\right)\right]\geq 1-n^{-C_2}
\end{align*}
for all $A\subseteq V$ satisfying $|A|\leq \delta n$.
\end{proposition}

To show \cref{prop:fastconsensus_Jacob}, we prove the following result which might be an independent interest:
\begin{proposition} \label{prop:fastconsensus}
Consider a polynomial voting process on a graph $G$ of $n$ vertices.
Suppose that there exist absolute constants $C,\delta>0$ and a function $\epsilon=\epsilon(n)=o(1)$ such that
\begin{align*}
\E_{}[|A'|]\leq \frac{C|A|^{2}}{n}+\epsilon|A|
\end{align*}
holds for all $A\subseteq V$ satisfying $|A|\leq \delta n$.

Then, there exist positive constants $\delta',C',C''$ such that
\begin{align*}
\Pr\left[\Tcons(A)\leq C'\left(\log\log n+\frac{\log n}{\log\epsilon^{-1}}\right) \right]\geq 1-n^{-C''}
\end{align*}
holds for all $A\subseteq V$ satisfying $|A|\leq \delta'n$.
\end{proposition}

It should be noted that in \cref{prop:fastconsensus}, we do not restrict the underlying graph $G$ to be random graphs.

\subsection{Escape from a fixed point}
Consider an $(f_1,f_2)$-polynomial voting process on an $f_1$-good and $f_2$-good $G(2n,p,q)$ such that $p$ and $q$ are constants.
Let $\mathbf{a}^* \in \mathbb{R}^2$ be a fixed point of the induced dynamical system $H$.
Let $J$ be the Jacobian matrix of $H$ at $\mathbf{a}^*$ and $\lambda_1,\lambda_2$ be its eigenvalues.
Let $\mathbf{u}_i$ be the eigenvector of $J$ corresponding to $\lambda_i$.
Suppose that $\mathbf{u}_1,\mathbf{u}_2$ are linearly independent.
Then, we can rewrite $J$ as
\begin{align*}
J=U^{-1}\Lambda U,
\end{align*}
where $\Lambda=\mathrm{diag}(\lambda_1,\lambda_2)$ and $U=(\mathbf{u}_1\,\mathbf{u}_2)^{-1}$.
For a fixed point $\mathbf{a}^* \in \mathbb{R}^2$, let ${\boldsymbol \beta}=(\beta_1,\beta_2)$ be a vector-valued random variable defined as
\begin{align}
{\boldsymbol \beta}=U({\boldsymbol \alpha}-\mathbf{a}^*). \label{eqn:betadef}
\end{align}
Roughly speaking, from the Taylor expansion of $H$ at $\mathbf{a}^*$, we have
\begin{align*}
\E[{\boldsymbol \beta}']\approx \Lambda {\boldsymbol \beta}
\end{align*}
if $\|{\boldsymbol \beta}\|_\infty$ is sufficiently small.
Thus, $|\beta_i'|\approx |\lambda_i||\beta_i|$.

Recall that $B(\mathbf{x},R)$ is the open ball of radius $R$ centered at $\mathbf{x}$.
If $|\lambda_i|>1$ for some $i\in[2]$, one may expect that ${\boldsymbol \alpha}^{(\tau)}\not\in B(\mathbf{a}^*,\epsilon_0)$ holds for any $A^{(0)}\subseteq V$ and for some constant $\epsilon_0>0$.
We aim to prove this under some assumptions.

\begin{assumption}[Basic assumptions] \label{asm:escape_assumption0}
We consider an $(f_1,f_2)$-polynomial voting process on an $f_1$-good and $f_2$-good $G(2n,p,q)$ for constants $p\geq q\geq 0$.
Let $\mathbf{a}^*$ be a fixed point and $J$ be the corresponding Jacobian matrix satisfying
\begin{enumerate}[label={\rm (A\arabic*)}]
\item The eigenvectors $\mathbf{u}_1$ and $\mathbf{u}_2$ are linearly independent. \label{asm:asm1}
\item A positive constant $\epsilon_0$ exists such that $\Var[\alpha'_i\mid A]\geq \Omega(n^{-1})$ for all $i\in\{1,2\}$ and all $A\subseteq V$ of ${\boldsymbol \alpha}\in B\left(\mathbf{a}^*,\epsilon_0\right)$. \label{asm:asm2}
\item The matrix $J$ contains an eigenvalue $\lambda$ satisfying $|\lambda|>1$. \label{asm:asm3}
\end{enumerate}
\end{assumption}

Under \cref{asm:escape_assumption0}, we can define the random variable ${\boldsymbol \beta}$ of \cref{eqn:betadef}.
Further, we put the following.
\begin{assumption} \label{asm:escape_assumption1}
In addition to \cref{asm:escape_assumption0}, we assume that there exists a positive constant $\epsilon^*$ satisfying the followings:
\begin{enumerate}[label={\rm (A\arabic*)}]
\setcounter{enumi}{3}
\item There exist two positive constants $\epsilon_1,C$ such that
\begin{align*}
|\E[\beta'_i\mid A]|\geq (1+\epsilon_1)|\beta_i|-\frac{C}{\sqrt{n}}
\end{align*}
holds for any $A\subseteq V$ of $\|{\boldsymbol \beta}\|\leq \epsilon^*$ and any $i\in[2]$ of $|\lambda_i|>1$. \label{asm:asm4}
\item For any $i\in [2]$ of $|\lambda_i|\leq 1$,
\begin{align*}
\Pr[|\beta'_i|\leq \epsilon^*\mid |\beta_i|\leq \epsilon^*] \geq 1-n^{-\Omega(1)}.
\end{align*}
\label{asm:asm5}
\end{enumerate}
\end{assumption}

Sometimes, it might be not easy to check the conditions of \cref{asm:escape_assumption1}.
In this paper, we provide the following alternative condition which is easy to check:
\begin{assumption}\label{asm:escape_assumption2}
In addition to \cref{asm:escape_assumption0}, we assume the following:
\begin{enumerate}[label={\rm (A\arabic*)}]
\setcounter{enumi}{5}
\item The eigenvalues $\lambda_1,\lambda_2$ of $J$ satisfies $|\lambda_i|\neq 1$ for all $i\in[2]$. \label{asm:asm6}
\end{enumerate}
\end{assumption}

Based on the assumptions, we prove the following result:
\begin{proposition}[Escape from source and sink areas] \label{prop:escape}
Let $\mathbf{a}^*$ be a fixed point satisfying either \cref{asm:escape_assumption1}~or~\ref{asm:escape_assumption2}.
Then, there exist $\tau=O(\log n)$ and a constant $\epsilon'>0$ such that the followings hold w.h.p.:
\begin{enumerate}[label=(\roman*)]
\item $\|\beta^{(\tau)}\|_\infty > \epsilon'$, and
\item $|\beta^{(\tau)}_j| \leq \epsilon'$ for any $j\in [2]$ of $|\lambda_j|\leq 1$.
\end{enumerate}
\end{proposition}

\subsection{Application to the Best-of-three} \label{sec:apply_to_3M}
In this section, we will prove the results of \cref{sec:auxiliary_results}. 
Consider the Best-of-three.
The Jacobian matrix of the dynamical system of \cref{eqn:3Mvecfielddef} is
\begin{align} \label{eqn:general_Jacobian_3M}
J&=\frac{3}{2}\left( \begin{array}{cc}
u(1-(ud_1)^2-d_2^2) & -2ud_1d_2 \\
-2u^2d_1d_2 & 1-(ud_1)^2-d_2^2
\end{array} \right).
\end{align}
Let $J_i$ be the Jacobian matrix at $\mathbf{d}^*_i$, where $\mathbf{d}^*_i$ is the fixed points \cref{eqn:fixed_points_3M}.
A straightforward calculation yields
\begin{align}  \label{eqn:Jacobian_3M}
J_i&=\begin{cases}
\frac{3}{2}\left( \begin{array}{cc}
u & 0 \\
0 & 1
\end{array} \right) & \text{if $i=1$}, \\
3\left( \begin{array}{cc}
1-u & 0 \\
0 & \frac{1}{u}-1
\end{array} \right) & \text{if $i=2$ and $u\geq \frac{2}{3}$}, \\
\frac{1}{2}\left( \begin{array}{cc}
1 & -\frac{\sqrt{4u-3}}{u} \\
-\sqrt{4u-3} & \frac{1}{u}
\end{array} \right) & \text{if $i=3$ and $u\geq \frac{3}{4}$}, \\
\left( \begin{array}{cc}
0 & 0 \\
0 & 0
\end{array} \right) & \text{if $i=4$}.
\end{cases}
\end{align}

Depending on the eigenvalues $\lambda_1\geq \lambda_2$ of $J_i$, the property of $\mathbf{d}^*_i$ changes as shown in \cref{tbl:eigentable}.
\begin{table}[htbp]
\caption{
Each cell $(c_1,c_2)$ represents the property of the eigenvalues $\lambda_1\geq\lambda_2$ of the corresponding Jacobian matrix.
Precisely, the sign $c_i$ represents whether $\lambda_i$ is larger than $1$ or not.
For example, $(+,1)$ indicates that $\lambda_1>\lambda_2=1$.
If $(+,-)$ or $(+,+)$, we may apply \cref{prop:escape}.
Indeed, cells with $(-,-)$ correspond sink points in this model.
Note that $\mathbf{d}^*_2$ is saddle if $\frac{2}{3}<u<\frac{3}{4}$ but is sink if $\frac{3}{4}<u\leq 1$.
\label{tbl:eigentable}}
\vspace{1em}
\centering
\begin{tabular}{|c|c|c|c|c|c|}
\hline
points           & $0<u<\frac{2}{3}$ & $u=\frac{2}{3}$ & $\frac{2}{3}<u<\frac{3}{4}$ & $u=\frac{3}{4}$ & $\frac{3}{4}<u\leq 1$ \\ \hline
$\mathbf{d}^*_1$ & $(+,-)$           & $(+,1)$         & $(+,+)$                     & $(+,+)$         & $(+,+)$               \\ \hline
$\mathbf{d}^*_2$ &  undefined  & $(+,1)$         & $(+,-)$                     & $(1,-)$         & $(-,-)$               \\ \hline
$\mathbf{d}^*_3$ &  undefined       &  undefined      &    undefined               & $(1,-)$         & $(+,-)$               \\ \hline
$\mathbf{d}^*_4$ & $(-,-)$           & $(-,-)$         & $(-,-)$                     & $(-,-)$         & $(-,-)$               \\ \hline
\end{tabular}
\end{table}

\begin{proof}[Proof of \cref{lem:3Mvector_field_lem1}]
If $(d_1,d_2)\in S$, we have $3d_2^2+(ud_1)^2\leq\max\{3,u^2\}\leq 3$ and $d_2^2+3(ud_1)^2\leq\max\{1,3u^2\}\leq 3$.
%\begin{align*}
%&3d_2^2+(ud_1)^2\leq\max\{3,u^2\}\leq 3,\\
%&d_2^2+3(ud_1)^2\leq\max\{1,3u^2\}\leq 3.
%\end{align*}
Hence, we have $d'_1\geq 0$ and $d'_2\geq 0$.
Let $x=\frac{1+d_1+d_2}{2}$ and $y=\frac{1+d_2-d_1}{2}$.
Then, $(d_1,d_2)\in S$ implies $\frac{1}{2}\leq x \leq 1$ and $0\leq y\leq 1$.
In addition, a simple calculation yields
$
d_1'+d_2'=3\left(\frac{x+ry}{1+r}\right)^2-2\left(\frac{x+ry}{1+r}\right)^3 \leq 1,
$
where $r=\frac{1-u}{1+u}$.
Note that $0\leq\frac{x+ry}{1+r}\leq 1$ and the function $f:z\mapsto 3z^2-2z^3$ satisfies $f(z)\leq f(1)=1$ for all $0\leq z\leq 1$.
\end{proof}

\begin{proof}[Proof of \cref{prop:sink_3M}]
It is straightforward to check that the points $\mathbf{d}^*_2$ and $-\mathbf{d}^*_2$ are sink.
Therefore, \cref{prop:sink_3M} immediately follows from \cref{prop:sinkpoint}.
\end{proof}

\begin{proof}[Proof of \cref{prop:fastconsensus_3M}]
Note that $J_4$ is the all-zero matrix and the same holds at $-\mathbf{d}^*_4$.
Let $\epsilon>0$ be sufficiently small constant.
If $A$ satisfies $|A|\leq \epsilon n$, apply \cref{prop:fastconsensus_Jacob}.
If $A$ satisfies $|A|\geq (2-\epsilon)n$, apply \cref{prop:fastconsensus_Jacob} for $V\setminus A$.
\end{proof}

\begin{proof}[Proof of \cref{prop:escape_3M}]
Suppose that $u<\frac{3}{4}$ (or equivalently, $r>\frac{1}{7}$) and that $p\geq q\geq 0$ are constants.
We prove that $\Tcons(A)=O(\log n)$ for any $A=A^{(0)}\subseteq V$.
Consider the stochastic process $({\boldsymbol \delta}^{(t)})_{t=0}^\infty$.
From \cref{prop:3Mvectorfield,prop:fastconsensus_3M}, it suffices to show that $|\delta_2^{(\tau)}|>\epsilon_1$ holds w.h.p.~for some constant $\epsilon_1>0$ and some $\tau=O(\log n)$.
To this end, we use \cref{prop:escape}.
From \cref{prop:3Mvectorfield}, we may assume
\begin{align}
{\boldsymbol \delta}^{(0)}\in \bigcup_{i\in\{1,2\}}B(\mathbf{d}^*_i,\epsilon_2) \label{eqn:initial_assumption}
\end{align}
for any constant $\epsilon_2>0$.
We check the condition~\ref{asm:asm2} of Assumption~\ref{asm:escape_assumption0}.
To this end, we show the following result:
\begin{theorem}[Concentration of the variance for the Best-of-three]  \label{thm:mainthm_Var}
Consider the Best-of-three on $f^{\Bthree}$-good $G(2n,p,q)$.
Then, two constants $C_1,C_2>0$ exist such that %$G(2n,p,q)$ satisfies the following with probability $1-O(n^{-C_1})$:

\begin{align*}
\forall A\subseteq V, \, \forall i\in\{1,2\}\,:\,
\left|\Var\left[|A'_i|\,\middle|\, A \right]
-
n\cdot g\left(\frac{|A_i|p+|A_{3-i}|q}{n(p+q)}\right)
\right| \leq C_2\sqrt{\frac{n}{p}},
\end{align*}
where $g(x)\defeq f^{\Bthree}(x)(1-f^{\Bthree}(x))$.
\end{theorem}
\begin{proof}
That variance $\Var[|A'_i| \mid A]$ can be represented as
\begin{align*}
\Var[|A'_i| \mid A]
 = \sum_{v\in V_i} \Pr[v\in A'](1-\Pr[v\in A'])
 = \sum_{v\in V_i} f^{\Bthree}\left(\frac{\deg_A(v)}{\deg(v)}\right) \left(1-f^{\Bthree}\left(\frac{\deg_A(v)}{\deg(v)}\right)\right).
\end{align*}
Therefore, \cref{thm:mainthm_Var} immediately follows from property~\ref{pro:gomiN}.
\end{proof}
From \cref{thm:mainthm_Var}, a straightforward calculation leads to
\begin{align*}
\Var[\alpha'_i \mid A]
&= \frac{1}{n^2}\Var[|A'_i| \mid A] \\
&= \frac{(3z_i^2-2z_i^3)\bigl(1-(3z_i^2-2z_i^3)\bigr)}{n} \pm O\left(\frac{1}{\sqrt{n^3p}}\right) \\
&=\frac{z_i^2(3-2z_i)(1-z_i)^2(2z_i+1)}{n} \pm O\left(\frac{1}{\sqrt{n^3p}}\right),
\end{align*}
where $z_i=\frac{a_i+ra_{3-i}}{1+r}$.
If $|\delta_2|<1-\epsilon_3$ (or equivalently, $\epsilon_3<|a_1+a_2|<2-\epsilon_3$) for some constant $\epsilon_3>0$, we have $\Var[\alpha'_i]=\Omega(n^{-1})$.
Hence, the statement~\ref{asm:asm2} holds for every ${\boldsymbol \delta}$ satisfying ${\boldsymbol \delta} \in \bigcup_{i\in\{1,2\}}B(\mathbf{d}^*_i,\epsilon_2)$ with sufficiently small constant $\epsilon_2<1$ mentioned in \cref{eqn:initial_assumption}.

\paragraph*{The case of $u\neq\frac{2}{3}$.}
If $u\neq\frac{2}{3}$, then both $\mathbf{d}^*_1$ and $\mathbf{d}^*_2$ satisfy Assumption~\ref{asm:escape_assumption2}:
From \cref{eqn:Jacobian_3M}, both of these points satisfies the condition~\ref{asm:asm1} and~\ref{asm:asm3} of Assumption~\ref{asm:escape_assumption0}, and the condition~\ref{asm:asm6} of Assumption~\ref{asm:escape_assumption2} (here, we use $u\neq\frac{2}{3}$).

Suppose that $u<\frac{2}{3}$.
Then the fixed point $\mathbf{d}^*_2$ does not exist and thus we assume ${\boldsymbol \delta}^{(0)} \in B(\mathbf{d}^*_1,\epsilon_2)$.
From \cref{prop:escape}, we have $|\delta_2^{(\tau)}|>\epsilon_2$ for some $\tau=O(\log n)$ (note that ${\boldsymbol \beta}=(\beta_1,\beta_2)$ of \cref{eqn:betadef} is given by ${\boldsymbol \beta}={\boldsymbol \delta}$ and the eigenvalues satisfy $0\leq \lambda_1 < 1 < \lambda_2$).

Suppose that $u>\frac{2}{3}$.
Both eigenvalues of $J_1$ are strictly larger than $1$.
Hence, for $\mathbf{d}^*_1$, \cref{prop:escape} implies that either $|\delta_1^{(\tau)}|>\epsilon_2$ or $|\delta_2^{(\tau)}|>\epsilon_2$ holds for some $\tau=O(\log n)$ if ${\boldsymbol \delta}^{(0)} \in B(\mathbf{d}^*_1,\epsilon_2)$.
If the former holds with $|\delta_2^{(\tau)}|=o(1)$, then ${\boldsymbol \delta}^{(\tau+T)} \in B(\mathbf{d}^*_2,\epsilon_2)$ holds for some constant $T=T(\epsilon_2)$ since $d_1'>d_1$ holds whenever $0<d_1<\sqrt{\frac{3u-2}{u^3}}$ and $d_2=0$.
Note that, at the point $\mathbf{d}^*_2$, the Jacobian matrix $J_2$ has eigenvalues $\lambda_1,\lambda_2$ satisfying $0<\lambda_1<1<\lambda_2$.
Moreover, in look at \cref{eqn:betadef}, we have ${\boldsymbol \beta}={\boldsymbol \delta}-\mathbf{d}^*_2$.
Thus, \cref{prop:escape} yields that $|\delta_2^{(\tau')}|>\epsilon_2$ holds for some $\tau'=O(\log n)$ and for any ${\boldsymbol \delta}^{(0)} \in B(\mathbf{d}^*_2,\epsilon_2)$.

\paragraph*{The case of $u=\frac{2}{3}$.}
In this case, we have $\mathbf{d}^*_1=\mathbf{d}^*_2=(0,0)$.
We claim that this point satisfies \ref{asm:asm4} and \ref{asm:asm5} and then apply \cref{prop:escape}.

Let $\epsilon_2>0$ be sufficiently small constant mentioned in \cref{eqn:initial_assumption}.
The Jacobian matrix $J_1=J_2$ has eigenvalues $1$ and $\frac{3}{2}$.
Suppose that $\|{\boldsymbol \delta}^{(0)}\|_\infty \leq \epsilon_2$ for sufficiently small constant $\epsilon_2>0$.
Then, we have
\begin{align*}
|\E[\delta_2' \mid A]|&=\left|\frac{\delta_2}{2}(3-3(u\delta_1)^2-\delta_2^2)\right| \pm O(n^{-0.5})  \\
& \geq |\delta_2|\left(\frac{3}{2} - O_{\epsilon_2\to 0}(\epsilon_2)\right) - O(n^{-0.5})  \\
&\geq 1.49 |\delta_2| - O(n^{-0.5}).
\end{align*}
This verifies the assumption~\ref{asm:asm4}.
On the other hand, for any ${\boldsymbol \delta}$ of $|\delta_1|\leq\epsilon_2$, we have
\begin{align}
\left|\E[\delta_1' \mid A]\right|&=\left|\frac{u\delta_1}{2}\right|\left|3-(u\delta_1)^2-3\delta_2^2\right| \pm O(n^{-0.5}) \nonumber\\
&\leq |\delta_1|\left|1 - \frac{4}{27} \delta_1^2\right| + O(n^{-0.5}) \nonumber\\
&\leq |\delta_1|\left(1 - \frac{4}{27} \delta_1^2\right) + O(n^{-0.5}). \label{eqn:gamma1_exp_bound}
\end{align}
By applying the Hoeffding bound (\cref{lem:Hoeffding}) to the random variables $\delta_1=\alpha_1-\alpha_2$ and $\delta_2=\alpha_1+\alpha_2-\frac{1}{2}$, we obtain
\begin{align*}
\Pr\left[|\delta'_i-\E[\delta'_i]|\geq t\right] \leq \exp\left(-\Omega\left(nt^2\right)\right).
\end{align*}
In particular, letting $t=\sqrt{\frac{\log n}{n}}$, we have
\begin{align}
\delta'_i=\E[\delta'_i] \pm O\left(\sqrt{\frac{\log n}{n}}\right) \label{eqn:gamma_bound}
\end{align}
holds w.h.p.
We claim that 
\begin{align*}
\Pr\left[|\delta'_1|\leq\epsilon_2 \,\middle|\,|\delta_1|\leq\epsilon_2\right] \geq 1-n^{-\Omega(1)} 
\end{align*}
holds, which is equivalent to \ref{asm:asm5}.
From \cref{eqn:gamma1_exp_bound,eqn:gamma_bound}, if $|\delta_1|\leq\epsilon_2$, we have
\begin{align*}
|\delta'_1| \leq |\delta_1| - \frac{4}{27}|\delta_1|^3 + C\sqrt{\frac{\log n}{n}}
\end{align*}
for sufficiently large constant $C>0$ and large $n$.
If $|\delta_1|^3\geq \frac{27C}{4}\sqrt{\frac{\log n}{n}}$, we have $|\delta'_1|\leq |\delta_1|\leq \epsilon_2$ holds w.h.p.
If $|\delta_1|^3 < \frac{27C}{4}\sqrt{\frac{\log n}{n}}$, we have $|\delta'_1|=O\left(\sqrt{\frac{\log n}{n}}\right)\leq \epsilon_2$ holds w.h.p.

Thus, from \cref{prop:escape}, we have $|\delta^{(\tau)}_2|>\epsilon_2$ w.h.p.~for some $\tau=O(\log n)$.
This completes the proof of \cref{thm:worst_case_3M}. 
\end{proof}

\begin{proof}[Proof of \cref{prop:3Mvectorfield}]
If $u=1$, we have
\begin{align*}
d'_1+d'_2 &= f(d_1+d_2), \\
d'_1-d'_2 &= f(d_1-d_2),
\end{align*}
where $f:z\mapsto \frac{1}{2}z(3-z^2)$.
Since $f(z)>z$ for $z\in(0,1)$, we have $\lim_{t\to\infty}\mathbf{d}^{(t)}\in\{(0,0),(0,1),(1,0)\}$.
In addition, if $d_2^{(0)}>0$ then $d_2^{(t)}\to 1$ as $t\to\infty$.

Suppose that $0\leq u<1$.
We use basic results of \emph{competitive dynamics} (see \cref{sec:competitive_system}).
We first claim that the map $T:S\to S$ is competitive and it satisfies the conditions \ref{cond:map_condition1} to \ref{cond:map_condition4} described in \cref{sec:competitive_system}.
Then, we apply \cref{thm:orbit_convergence} and complete the proof of the first statement.
To this end, we consider the Jacobian matrix $J$ given in \cref{eqn:general_Jacobian_3M}.

The condition \ref{cond:map_condition1} follows from \cref{lem:criterion_C1}: it is straightforward to check that the Jacobian matrix \cref{eqn:general_Jacobian_3M} satisfies the condition of \cref{lem:criterion_C1}.
To verify the condition \ref{cond:map_condition2}, we use the Inverse Function Theorem (\cref{thm:IFT}).
We claim that $\det J>0$ for any $\mathbf{d}\in S\setminus\{(0,1)\}$.
Indeed, in look at \cref{eqn:general_Jacobian_3M}, we have
\begin{align}
\det J = \frac{9u}{4}(1-(ud_1-d_2)^2)(1-(ud_1+d_2)^2) 
\geq \frac{9u}{4}(1-ud_1-d_2)^4 
>0 \label{eqn:positive_det_3M}
\end{align}
for $(d_1,d_2)\in S\setminus \{(0,1)\}$.
Hence, for any $\mathbf{d}\in S\setminus \{(0,1)\}$, an inverse mapping $T^{-1}(\mathbf{d})$ is defined.
Moreover, it is straightforward to check that $T(\mathbf{x})=(0,1)$ if and only if $\mathbf{x}=(0,1)$.
Consequently, $T$ is injective.
The condition \ref{cond:map_condition3} immediately follows from the definition \cref{eqn:3Mvecfielddef}.
To condition \ref{cond:map_condition4} follows from \cref{lem:criterion_C4} and \cref{eqn:positive_det_3M}.
Now we apply \cref{thm:orbit_convergence} and complete the proof of the first claim of \cref{prop:3Mvectorfield}.

To obtain the second claim of \cref{prop:3Mvectorfield}, we show that $d'_2>0$ whenever $(d_1,d_2)\in S$ satisfies $d_2>0$ (here, $d_2>0$ means that $d_2>c$ for some constant $c>0$; we do not deal with the case of $d_2\geq o(1)$).
This follows from a simple calculation
\begin{align*}
d'_2 \geq \frac{d_2}{2}(3-3d_1^2-d_2^2) \geq \frac{d_2}{2}(3-3d_1-d_2) 
\geq d_2^2
>0.
\end{align*}
\end{proof}

%\paragraph*{Around the fixed point $\mathbf{d}^*_3$.}
%If $\frac{3}{4}<u\leq 1$ (or equivalently, $0\leq r<\frac{1}{7}$), the point $\mathbf{d}^*_3$ satisfies Assumption~\ref{asm:escape_assumption2}.
%Indeed, the Jacobian matrix $J_3$ can be eigendecomposed as
%\begin{align*}
%J_3=U^{-1}\Lambda U,
%\end{align*}
%where
%\begin{align*}
%U&=\frac{1}{\sqrt{17u^2-14u+1}}
%\left( \begin{array}{cc}
%u\sqrt{4u-3} & \frac{-1+u+\sqrt{17u^2-14u+1}}{2} \\
%-u\sqrt{4u-3} & -\frac{-1+u-\sqrt{17u^2-14u+1}}{2} \\
%\end{array} \right)
%\end{align*}
%and
%\begin{align*}
%\Lambda&=
%\left( \begin{array}{cc}
%1+u+\sqrt{17u^2-14u+1} & 0 \\
%0 & 1+u-\sqrt{17u^2-14u+1}
%\end{array} \right).
%\end{align*}

%\Cref{prop:escape} implies that ${\boldsymbol \delta}^{(\tau)} \not \in B(\mathbf{d}^*_3,\epsilon_2)$ for some $\tau=O(\log n)$ if ${\boldsymbol \delta}^{(0)}  \in B(\mathbf{d}^*_3,\epsilon_2)$.
%Therefore, we obtain the following result:
%\begin{proposition}
%Consider the Best-of-three on $G(2n,p,q)$ with constants $p\geq q\geq 0$.
%Suppose that $r=\frac{q}{p}<\frac{1}{7}$.
%For some constant $\epsilon>0$ and some $\tau=O(\log n)$, it holds w.h.p.~that either the Best-of-three reaches the consensus or ${\boldsymbol \delta}^{(\tau)}\in B(\mathbf{d}^*_2,\epsilon)$ for any $A^{(0)}\subseteq V$.
%\end{proposition}

%\newpage

\section[Proof of f-goodness]{Proof of the $f$-goodness of the stochastic block model (\cref{thm:fgood})}\label{sec:randomgraphs}
%In this section we complete a proof of \cref{thm:fgood}.
In this section we show \cref{thm:fgood}. %
%\Cref{sec:reduction,sec:W-EW,sec:EW-Wh,sec:technique} correspond to~\ref{pro:gomiN}
%and \cref{sec:gomiA} corresponds to~\ref{pro:gomiA}.
In \cref{sec:reduction}, we show that the property \ref{pro:gomiN} is obtained from \cref{lem:WIdiscrepancy}.
\begin{lemma}
\label{lem:WIdiscrepancy}
For a vertex set $V$ with $|V|=\Vsize$, let $(\Ind_e)_{e\in \binom{V}{2}}$ be $\binom{N}{2}$ independent binary random variables. 
Let $\pmax \defeq \max_{e\in \binom{V}{2}}\E[\Ind_e]$.
Suppose that $\Vsize \pmax \geq 1$.
For $\ell+1$ vertex subsets $S_0, S_1, \ldots, S_\ell$, let
\begin{align*}
\W(S_0;S_1,\ldots,S_\ell)&\defeq\sum_{s\in S_0}\prod_{i\in[\ell]}\deg_{S_i}(s), 
\\
\Wh(S_0;S_1,\ldots,S_\ell)&\defeq \sum_{s\in S_0}\prod_{i\in[\ell]}\E[\deg_{S_i}(s)]
\end{align*}
where $\deg_S(v)=\sum_{s\in S\setminus \{v\}}\Ind_{\{v,s\}}$ for $S\subseteq V$ and $v\in V$.

Then two positive constants $C_1, C_2$ exist such that the following holds with probability $1-\Vsize^{-C_1}$: 
\begin{align*}
\forall S_0, S_1, \ldots, S_\ell: \bigl|\W(S_0;S_1,\ldots,S_\ell)-\Wh(S_0;S_1,\ldots,S_\ell)\bigr| 
\leq C_2\Vsize (\Vsize \pmax)^{\ell-1/2}.
\end{align*}
\end{lemma}

Our proof of \cref{lem:WIdiscrepancy} consists of three parts.
First, we give a concentration of $W$ (\cref{lem:WEWdiscrepancy}). % in \cref{sec:W-EW}.
Next, we give an upper bound on the discrepancy between $\E[W]$ and $\Wh$ (\cref{lem:EWIdiscrepancy}). % in \cref{sec:EW-Wh}.
At the end, we show \cref{lem:technique} which plays a key role in showing \cref{lem:WEWdiscrepancy,lem:EWIdiscrepancy}. %, which completes the proof of \cref{lem:WIdiscrepancy}.

\subsection[Reduction to W]{Reduction to $W$}\label{sec:reduction}
\begin{proof}[Proof of \ref{pro:gomiN} of \cref{thm:fgood} via \cref{lem:WIdiscrepancy}]
Let $f(x)=\sum_{j=0}^\ell c_jx^j$.
For notational convenience, 
let $x_v=\frac{\deg_A(v)}{\deg(v)}$, $\bar{x}_v=\frac{\E[\deg_A(v)]}{\E[\deg(v)]}$ and $\hat{x}=\frac{|A_{i}|p+|A_{3-i}|q}{n(p+q)}$.
Then from the triangle inequality, it holds that
\begin{align}
\left|\sum_{v\in S\cap V_i}\bigl(f(x_v)-f(\hat{x})\bigr)\right|
%\left|\sum_{v\in S\cap V_i}\left(f\left(\frac{\deg_A(v)}{\deg(v)}\right) -f\left( \frac{|A_i|p+|A_{3-i}|q}{n(p+q)} \right)\right)\right|
&\leq \left|\sum_{v\in S\cap V_i}\bigl(f(x_v)-f(\bar{x}_v)\bigr)\right|+\left|\sum_{v\in S\cap V_i}\bigl(f(\bar{x}_v)-f(\hat{x})\bigr)\right|.
\label{eq:trianglegomi}
\end{align}
For the second term of the right hand of \cref{eq:trianglegomi}, two positive constant $C_1, C_2$ exist such that
\begin{align*}
\left|\sum_{v\in S\cap V_i}\bigl(f(\bar{x}_v)-f(\hat{x})\bigr)\right|
\leq \sum_{v\in S\cap V_i}|f(\bar{x}_v)-f(\hat{x})|
\leq C_1\sum_{v\in S\cap V_i}|\bar{x}_v-\hat{x}|
\leq C_2\frac{|S\cap V_i|}{n} 
\leq C_2.
\end{align*}
The second inequality follows from the Lipschitz condition of $f$ (c.f.~\cref{sec:lipschitzcondition}).
The third inequality holds since $\E[\deg(v)]=(n-1)p+nq$ and $(|A_{i}|-1)p+|A_{3-i}|q \leq \E[\deg_A(v)]\leq |A_{i}|p+|A_{3-i}|q$ for any $v\in V_i$.

For the first term of the right hand of \cref{eq:trianglegomi}, since
\begin{align*}
\left(\frac{\deg_A(v)}{\deg(v)}\right)^j-\left(\frac{\E[\deg_A(v)]}{\E[\deg(v)]}\right)^j
&=\frac{\left(\E[\deg(v)]^j-\deg(v)^j\right)\left(\frac{\deg_A(v)}{\deg(v)}\right)^j+\left(\deg_A(v)^j-\E[\deg_A(v)]^j\right)}{\E[\deg(v)]^j}
\end{align*}
for any $j$ and $v\in V$, we have
\begin{align*}
\lefteqn{\left|\sum_{v\in S\cap V_i}\bigl(f(x_v)-f(\bar{x}_v)\bigr)\right|
=\left|\sum_{j=1}^{\ell}c_j \sum_{v\in S\cap V_i}\left((x_v)^j-(\bar{x}_v)^j\right)\right|}\\
&\leq \sum_{j=1}^\ell \frac{|c_j|}{((n-1)p)^j} \left(\left|\sum_{v\in S\cap V_i}\left(\E[\deg(v)]^{j}-\deg(v)^j\right)\left(x_v\right)^j \right|+\left| \sum_{v\in S\cap V_i}\left(\deg_A(v)^j-\E[\deg_A(v)]^j\right)\right|\right).
%\left| \sum_{v\in S\cap V_i}\left(\bigl(\E[\deg(v)]x_v\bigr)^j -\E[\deg_A(v)]^j\right)\right|.
\end{align*}
Note that $\E[\deg(v)]=(n-1)p+nq\geq (n-1)p$ for any $v\in V$.
Since
\begin{align*}
\left|\sum_{v\in S\cap V_i}\left(\E[\deg(v)]^{j}-\deg(v)^j\right)\left(x_v\right)^j \right|
&\leq \max_{U\subseteq V}\left|\sum_{u\in U}\left(\E[\deg(u)]^{j}-\deg(u)^j\right)\right|\\
&=\max_{U\subseteq V}\left|W(U;\overbrace{V,\ldots,V}^{j})-\Wh(U;\overbrace{V,\ldots,V}^{j})\right|
\end{align*}
and $\sum_{v\in S\cap V_i}\left(\deg_A(v)^j-\E[\deg_A(v)]^j\right)=W(S\cap V_i;\overbrace{A,\ldots,A}^{j})-\Wh(S\cap V_i;\overbrace{A,\ldots,A}^{j})$, 
we obtain the claim from \cref{lem:WIdiscrepancy}.
Note that, for any $S\subseteq V$, $a\in \mathbb{R}^V$ and $x\in [0,1]^V$, $|\sum_{s\in S}a_sx_s|\leq \max_{U\subseteq V}|\sum_{u\in U}a_u|$ since 
$\sum_{s\in S:a_s\leq 0}a_s \leq \sum_{s\in S}a_sx_s\leq \sum_{s\in S:a_s\geq 0}a_s$.
\end{proof}
Now we introduce the following lemma, which we will use in \cref{sec:W-EW,sec:EW-Wh}.
\begin{lemma}
\label{lem:technique}
Let $V$ be a set of $|V|=\Vsize$ vertices
and fix $l+1$ subsets $S_0, S_1, \ldots, S_l\subseteq V$. % be $l+1$ subsets of vertices.
For any $\mathbf{s}=(s_0,s_1,\ldots, s_l)\in S_0\times S_1\times \cdots \times S_l$, 
define 
\begin{align*}
U(\mathbf{s})&\defeq \{s_i: i\in \{0\}\cup [l]\}.
\end{align*}
Consider $\sum_{\substack{\mathbf{s}\in \mathcal{S}}}p^{|\mathcal{F}(\mathbf{s})|}$, where 
$p\in [1/\Vsize,1]$, $\mathcal{S}\subseteq S_0\times S_1\times \cdots \times S_l$ and  
\begin{align*}
\mathcal{F}: S_0\times S_1\times \cdots \times S_l \to 2^{\binom{V}{2}}.
\end{align*}
Suppose that the following three conditions hold for any $\mathbf{s}\in \mathcal{S}$:  
(1) $|\mathcal{F}(\mathbf{s})|\leq k$, 
(2) $\mathcal{\Es}(\mathbf{s})\subseteq \binom{U(\mathbf{s})}{2}$ 
and 
(3) the graph $G(\mathbf{s})=\bigl(\Vs(\mathbf{s}), \mathcal{F}(\mathbf{s})\bigr)$ is connected.
Let $L\subseteq \{0\}\cup [l]$ be the set of indices such that $S_i\cap S_j=\emptyset$ for any $i,j\in L$ $(i\neq j)$.
Then
\begin{align*}
\sum_{\substack{\mathbf{s}\in \mathcal{S}}}p^{|\mathcal{F}(\mathbf{s})|}\leq \mathcal{B}_{l+1}\Vsize (\Vsize p)^k\frac{\prod_{i\in L}|S_i|}{\Vsize ^{|L|}}
\end{align*}
where $\mathcal{B}_l$ denotes the $l$-th Bell number.
\end{lemma} 
The $l$-th {\em Bell number} $\mathcal{B}_l$ is the number of possible partitions of a set with $l$ labeled elements. 
It is known that $\mathcal{B}_{l}<\left(\frac{0.792 l}{\ln(l+1)}\right)^l$ for all positive integer $l$ \cite{DT10}.
\subsection[Concentration of W]{Concentration of $W$} \label{sec:W-EW}
\begin{lemma}
\label{lem:WEWdiscrepancy}
Suppose the same setting of \cref{lem:WIdiscrepancy}.
Then two positive constants $C_1, C_2$ exist such that the following holds with probability $1-\Vsize^{-C_1}$: 
\begin{align*}
\forall S_0, S_1, \ldots, S_\ell: 
\bigl|\W(S_0;S_1,\ldots,S_\ell)-\E[\W(S_0;S_1,\ldots,S_\ell)]\bigr|\leq  C_2 \Vsize(\Vsize \pmax)^{\ell-1/2}.
%=O\left(n(np)^{\ell-1/2}\right)
\end{align*}
\end{lemma}
\begin{proof}
For $\ell+1$ vertex subsets $S_0, S_1, \ldots, S_\ell$, let
\begin{align*}
\Slc&\defeq\left\{(s_0,s_1,\ldots,s_\ell): s_0\in S_0, s_i \in S_i\setminus \{s_0\}\ \textrm{for every $i\in [\ell]$}\right\}
\end{align*}
and for any $\mathbf{s}=(s_0, s_1,\ldots, s_\ell)\in \prod_{i=0}^\ell S_i=S_0 \times S_1\times \cdots \times S_\ell$, let 
\begin{align*}
\Es(\mathbf{s})&\defeq \bigl\{\{s_0,s_i\}: i\in [\ell] \bigr\}\setminus \{s_0\}. %\label{def:Es}
\end{align*}
For example, 
%$\Vs(\mathbf{s})=\{a,b,c,d,f\}$ and 
$\Es\bigl((a,b,a,c,d,b,f,a)\bigr)=\bigl\{\{a,b\},\{a,c\},\{a,d\},\{a,f\}\bigr\}$.
% for a given vector $\mathbf{s}=(a,b,a,c,d,b,f,a)$.
Then, %from the definition of $W(S_0;S_1,\ldots, S_\ell)$, 
\begin{align}
W(S_0;S_1,\ldots, S_\ell)
&=\sum_{s_0\in S_0}\prod_{i\in [\ell]}\left(\sum_{s_i\in S_i\setminus \{s_0\}}\Ind_{\{s_0,s_i\}}\right)
=\sum_{\mathbf{s}\in \Slc}\prod_{i\in [\ell]} \Ind_{\{s_0,s_i\}}
=\sum_{\mathbf{s}\in \Slc}\prod_{e\in \Es(\mathbf{s})} \Ind_{e}.
%=\sum_{\mathbf{s}\in \Slc}\mathbbm{1}_{\Es(\mathbf{s})\subseteq E}.
\label{eq:WtoJ}
\end{align}
%We give a following concentration results of $W(S_0;S_1,\ldots, S_\ell)$ by applying the Janson's inequality (Lemma~\ref{lem:Janson}) and the Kim-Vu inequality (Lemma~\ref{lem:Kim-Vu}). % to \eqref{eq:WtoJ}.
%
%\label{sec:WEWdiscrepancy}
%

\

\noindent \textbf{Lower bound on $W$. } 
First, we claim the following:
Two positive constants $C_3, C_4$ exist such that
\begin{align}
\Pr\left[\forall S_0, S_1\ldots, S_\ell: W(S_0;S_1,\ldots, S_\ell)\geq \E[W(S_0;S_1,\ldots, S_\ell)]-C_4 \Vsize (\Vsize \pmax)^{\ell-1/2}  \right]\geq 1-\Vsize^{-C_3}.
\label{claim:WEWlower}
\end{align}
To obtain \cref{claim:WEWlower}, we apply Janson's inequality (\cref{lem:Janson}) to \cref{eq:WtoJ}.
Then we have
\begin{align}
\lefteqn{\Pr\left[\exists S_0,S_1,\ldots, S_\ell: W(S_0;S_1,\ldots,S_\ell)\leq \E[W(S_0;S_1,\ldots,S_\ell)]-t\right]} \nonumber \\
&\leq \left(2^\Vsize\right)^{\ell+1}\exp\left(-\frac{t^2}{2\nabla(S_0;S_1,\ldots,S_\ell)}\right)
\leq \exp\left((\ell+1)\Vsize-\frac{t^2}{2\nabla(S_0;S_1,\ldots,S_\ell)}\right)
\label{eq:WJansonwithNabla}
%\leq \exp\left((\ell+1)\log \Vsize-\frac{t^2}{2\mathcal{B}_{2(\ell+1)}\Vsize (\Vsize \pmax)^{2\ell-1}}\right).
\end{align}
where
\begin{align*}
\nabla(S_0;S_1,\ldots,S_\ell)
&= \sum_{\substack{\mathbf{s}\in \Slc,\mathbf{s}'\in \Slc:\\\Es(\mathbf{s})\cap \Es(\mathbf{s}')\neq \emptyset}}
\E\left[\prod_{e\in \Es(\mathbf{s})} \Ind_{e}\prod_{e'\in \Es(\mathbf{s}')} \Ind_{e'}\right].
\end{align*}
Thus it suffices to show that
$
\nabla(S_0;S_1,\ldots,S_\ell)
= O(\Vsize (\Vsize \pmax)^{2\ell-1})
$.
Since $\max_{e\in \binom{V}{2}}\E[J_e]=\pmax$, it holds that
\begin{align}
\nabla(S_0;S_1,\ldots,S_\ell)
&=\sum_{\substack{\mathbf{s}\in \Slc,\mathbf{s}'\in \Slc:\\\Es(\mathbf{s})\cap \Es(\mathbf{s}')\neq \emptyset}}
\E\left[\prod_{e\in \Es(\mathbf{s})} \Ind_{e}\prod_{e'\in \Es(\mathbf{s}')} \Ind_{e'}\right]
\leq \sum_{\substack{\mathbf{s}\in \Slc,\mathbf{s}'\in \Slc:\\\Es(\mathbf{s})\cap \Es(\mathbf{s}')\neq \emptyset}}
p^{|\Es(\mathbf{s})\cup \Es(\mathbf{s}')|}. 
\label{eq:nablaupper1}
\end{align}
To bound \cref{eq:nablaupper1}, we apply \cref{lem:technique} which we will prove in \cref{sec:technique}.
Consider $2\ell+2$ vertex subsets $S'_0, S'_1, \ldots, S'_{2\ell+1}$ where $S'_i\defeq S_{i\bmod (\ell+1)}$.
For any $i\in \{0\}\cup [2\ell+1]$, let 
\begin{align*}
\mathcal{S}&\defeq \left\{ (s_0, s_1, \ldots, s_{2\ell+1})\in \Slc\times \Slc : \Es\bigl((s_0, \ldots, s_{\ell})\bigr)\cap \Es\bigl((s_{\ell+1},\ldots, s_{2\ell+1})\bigr)\neq \emptyset \right\}\subseteq \prod_{i=0}^{2\ell+1}S'_i, \\
\mathcal{F}(\mathbf{s})&\defeq \Es\bigl((s_0, \ldots, s_{\ell})\bigr)\cup \Es\bigl((s_{\ell+1},\ldots, s_{2\ell+1})\bigr)\ \textrm{for any}\ \mathbf{s}=(s_0, s_1, \ldots, s_{2\ell+1})\in \prod_{i=0}^{2\ell+1}S'_i.
\end{align*}
Then for any $\mathbf{s}\in \mathcal{S}$, 
$G({\mathbf{s}})=\bigl(\Vs({\mathbf{s}}), \mathcal{F}(\mathbf{s})\bigr)$ is a connected graph 
and $|\mathcal{F}(\mathbf{s})|\leq 2\ell-1$.
Thus, for any $i_* \in \{0\} \cup [\ell]$, \cref{lem:technique} with letting $l=2\ell+1,k=2\ell-1$ and $L=\{i_*\}$ yields
\begin{align}
\sum_{\substack{\mathbf{s}\in \Slc,\mathbf{s}'\in \Slc:\\\Es(\mathbf{s})\cap \Es(\mathbf{s}')\neq \emptyset}}
p^{|\Es(\mathbf{s})\cup \Es(\mathbf{s}')|}
&=\sum_{\mathbf{s}\in \mathcal{S}}p^{|\mathcal{F}(\mathbf{s})|}
\leq \mathcal{B}_{2(\ell+1)}\Vsize(\Vsize\pmax)^{2\ell-1}\frac{|S_{i_*}'|}{\Vsize}
=\mathcal{B}_{2(\ell+1)}|S_{i_*}|(\Vsize \pmax)^{2\ell-1}.
\label{eq:nablaupper2}
\end{align}
%Note that $\ell$ is a constant, 
Equations~\cref{eq:nablaupper1,eq:nablaupper2} imply the following statement:
For any $\ell+1$ vertex subsets $S_0, S_1, \ldots S_\ell$ and for any $i_*\in \{0\}\cup [\ell]$, %there exists a positive constant $C_5$ such that
\begin{align}
\nabla(S_0;S_1,\ldots,S_\ell)
&\leq \mathcal{B}_{2(\ell+1)}|S_{i_*}|(\Vsize \pmax)^{2\ell-1}
\leq \mathcal{B}_{2(\ell+1)}\Vsize (\Vsize \pmax)^{2\ell-1}.
\label{claim:nablaupper}
\end{align}
Thus by substituting $t=C_4 \Vsize(\Vsize \pmax)^{\ell-1/2}$ with $C_4=\sqrt{2(\ell+1+C_3)\mathcal{B}_{2(\ell+1)}}$ to \cref{eq:WJansonwithNabla}, we obtain the claim \cref{claim:WEWlower}.

\

\noindent \textbf{Upper bound on $W$. }
To complete the proof of \cref{lem:WEWdiscrepancy}, we combine the claim \cref{claim:WEWlower} and  the following claim:
Two positive constants $C_5, C_6$ exist such that
\begin{align}
\Pr\left[\forall S_0, S_1\ldots, S_\ell: W(S_0;S_1,\ldots, S_\ell)\leq \E[W(S_0;S_1,\ldots, S_\ell)]+C_6 \Vsize (\Vsize \pmax)^{\ell-1/2}  \right]\geq 1-\Vsize^{-C_5}.
\label{claim:WEWupper}
\end{align}
To show the claim, we consider the following expression of $W$.
For any $S_0, S_1, \ldots, S_\ell$, let 
$W_0\defeq W(S_0;S_1,\ldots,S_\ell)$ and
let 
$W_i\defeq W(\overbrace{V; V\ldots, V}^{i}, S_{i}, S_{i+1}, \ldots, S_{\ell})$ for each $i\in[\ell+1]$.
Since $W_{i+1}-W_{i}=W(\overbrace{V;V,\ldots,V}^{i},V\setminus S_{i}, S_{i+1}, \ldots, S_{\ell})$ for any $i\in \{0\}\cup [\ell+1]$
and $\sum_{i=0}^{\ell}(W_{i+1}-W_{i})=W_{\ell+1}-W_0$, we have
\begin{align}
\label{eq:WdlV}
W(S_0;S_1,\ldots,S_{\ell})&=W(\overbrace{V;V,\ldots,V}^{\ell+1})-\sum_{i=0}^{\ell}W(\overbrace{V;V,\ldots,V}^{i},V\setminus S_{i}, S_{i+1}, \ldots, S_{\ell}).
\end{align}
We can apply \cref{claim:WEWlower} for the second term of the right hand of \cref{eq:WdlV}. 
%For the notational convenience, let $Y=W(\overbrace{V;V,\ldots,V}^{\ell+1})$.
Now we try to get an upper bound on $W(\overbrace{V;V,\ldots,V}^{\ell+1})$. % by applying Kim-Vu inequality (\cref{lem:Kim-Vu}).
For the notational convenience, let $Y=W(\overbrace{V;V,\ldots,V}^{\ell+1})$.
Let $S_i=V$ for every $i\in \{0\}\cup [\ell]$ and let 
\begin{align*}
\mathcal{E}&\defeq \left\{\Es(\mathbf{s}): \mathbf{s}\in \Slc\right\}.
\end{align*}
From \eqref{eq:WtoJ}, we have
\begin{align*}
Y
=\sum_{\mathbf{s}\in \Slc}\prod_{e\in \Es(\mathbf{s})}\Ind_e
=\sum_{\Es\in \mathcal{E}} |\{\mathbf{s}\in \Slc : \Es(\mathbf{s})=F\}|\prod_{e\in \Es}\Ind_e.
\end{align*}
Thus applying Kim-Vu inequality (\cref{lem:Kim-Vu}) to $Y$ yields
\begin{align}
\Pr\left[|Y-\E[Y]|\geq \sqrt{\ell! \max_{A\subseteq \binom{V}{2}}\E[Y_{A}]\max_{A\subseteq \binom{V}{2}:A\neq \emptyset}\E[Y_{A}]}(8\lambda)^\ell\right]\leq 2\exp(2+2(\ell-1)\log \Vsize-\lambda)
\label{eq:KimVutoY}
\end{align}
where
\begin{align*}
Y_{A}
&=\sum_{\substack{\Es\in \mathcal{E}:\\F\supseteq A}} |\{\mathbf{s}\in \Slc:  \Es(\mathbf{s})=F\}|\prod_{e\in \Es\setminus A}\Ind_e
=\sum_{\substack{\mathbf{s}\in \Slc:\\ \Es(\mathbf{s}) \supseteq A}}\prod_{e\in \Es(\mathbf{s})\setminus A}\Ind_e.
\end{align*}
%for any $A\subseteq {V\choose 2}$.
%\begin{align*}
%d_\ell(V)&%=\sum_{\mathbf{s}\in \Slc}\mathbbm{1}_{D_{\mathbf{s}}\subseteq E}
%%=\sum_{\mathbf{s}\in \Slc}\prod_{e\in \Es(\mathbf{s})}J_e
%=\sum_{\Es\in \mathcal{E}} |\{\mathbf{s}\in \Slc\mid \Es(\mathbf{s})=F\}|\prod_{e\in \Es}J_e.
%\end{align*}
Now, we give an upper bound on $\E[Y_{A}]$. % to apply the Kim-Vu inequality (Lemma~\ref{lem:Kim-Vu}) to $d_\ell(V)$. 
%. , we have
Since $\max_{e\in \binom{V}{2}}\E[J_e]=p$,
it holds that
\begin{align*}
\E[Y_{A}]=\sum_{\substack{\mathbf{s}\in \Slc:\\ \Es(\mathbf{s}) \supseteq A}}\E\left[\prod_{e\in \Es(\mathbf{s})\setminus A}\Ind_e\right]
\leq \sum_{\substack{\mathbf{s}\in \Slc:\\ \Es(\mathbf{s}) \supseteq A}}p^{|\Es(\mathbf{s})\setminus A|}
=\sum_{\substack{\mathbf{s}\in \Slc:\\ \Es(\mathbf{s}) \supseteq A}}p^{|\Es(\mathbf{s})|-|A|}.
\end{align*}

If $A=\emptyset$, 
a direct application of \cref{lem:technique} with letting $l=k=\ell$ and $L=\emptyset$ yields
\begin{align}
\E[Y_{A}]=\E[Y]
&\leq \sum_{\substack{\mathbf{s}\in \Slc}}p^{|\Es(\mathbf{s})|}
%=\sum_{\substack{\mathbf{s}\in \Slc}}p^{|\Vs(\mathbf{s})|-1}
%\leq p^{-1} \mathcal{B}_{\ell+1}(np)^{\ell+1}
\leq \mathcal{B}_{\ell+1}\Vsize (\Vsize \pmax)^{\ell}.
\label{eq:YAzero}
\end{align}
Note that $|\Es(\mathbf{s})|\leq \ell$ and $G({\mathbf{s}})=(\Vs(\mathbf{s}),\Es(\mathbf{s}))$ is a connected graph for any $\mathbf{s}\in \Slc\subseteq \prod_{i=0}^\ell S_i$.

Now we consider the case $|A|=\kappa\geq 1$.
Let $V(A)$ be the set of vertices induced by the edge set $A\subseteq \binom{V}{2}$.
If $\Es(\mathbf{s})\supseteq A$ for some $\mathbf{s}\in \prod_{i=0}^\ell V$, the graph 
$G'=(V(A),A)$ is a star graph and hence $|V(A)|=|A|+1=\kappa+1$.
Let $V(A)=\{a_0, a_1, \ldots, a_{\kappa}\}$.
Now consider $(\ell+1)+(\kappa+1)$ vertex subsets $S'_0, S'_1, \ldots S'_{\ell+\kappa+1}$ where $S'_i=S_i$ for any $0\leq i\leq \ell$ and $S'_i=\{a_{i-(\ell+1)}\}$ for any $\ell+1\leq i\leq \ell+\kappa+1$.
Let
\begin{align*}
\mathcal{S}&\defeq \left\{ (s_0, s_1, \ldots, s_{\ell+\kappa+1})\in \Slc\times \prod_{i=0}^{\kappa}\{a_i\}: F\bigl( (s_0,\ldots,s_\ell)\bigr)\supseteq A \right\}\subseteq \prod_{i=0}^{\ell+\kappa+1}S'_i,\\
\mathcal{F}(\mathbf{s})&\defeq F\bigl((s_0,\ldots,s_\ell)\bigr) \ \textrm{for any}\ \mathbf{s}=(s_0, s_1, \ldots, s_{\ell+\kappa+1})\in \prod_{i=0}^{\ell+\kappa+1}S'_i.
\end{align*}
Note that, for any $\mathbf{s}\in \mathcal{S}$, 
the graph $G({\mathbf{s}})=\bigl(\Vs({\mathbf{s}}), \mathcal{F}(\mathbf{s})\bigr)$ is connected and $|\mathcal{F}(\mathbf{s})|\leq \ell$.
Thus \cref{lem:technique} 
with letting $l=\ell+\kappa+1$, $k=\ell$ and $L=\{\ell+1, \ell+2, \ldots, \ell+\kappa+1\}$ (note that $a_i\neq a_j$ for any $i\neq j$ and $\prod_{i=\ell+1}^{\ell+\kappa+1}|S'_i|=\prod_{i=\ell+1}^{\ell+\kappa+1}|\{a_{i-(\ell+1)}\}|=1$)
yields
\begin{align}
\E[Y_{A}]
&\leq 
\sum_{\substack{\mathbf{s}\in \Slc:\\ \Es(\mathbf{s}) \supseteq A}}p^{|\Es(\mathbf{s})|-|A|}
=\frac{1}{p^{\kappa}}\sum_{\substack{\mathbf{s}\in \mathcal{S}}}p^{|\mathcal{F}(\mathbf{s})|}
\leq \frac{1}{p^{\kappa}} \mathcal{B}_{\ell+\kappa+2} \Vsize(\Vsize\pmax)^{\ell}\frac{\prod_{i=\ell+1}^{\ell+\kappa+1}|S'_i|}{\Vsize^{\kappa+1}}
%\leq \frac{1}{p^k} \mathcal{B}_{2(\ell+1)} n(np)^{\ell}\frac{1}{n^{k+1}}
\leq \mathcal{B}_{2(\ell+1)}(\Vsize \pmax)^{\ell-\kappa}.
\label{eq:YApositive}
\end{align}
%Note that $\prod_{i=\ell+1}^{\ell+k+1}|S'_i|=\prod_{i=\ell+1}^{\ell+k+1}|\{a_{i-(\ell+1)}\}|=1$. 
%
Combining \cref{eq:YAzero,eq:YApositive}, we have
\begin{align*}
\max_{\substack{A\subseteq \binom{V}{ 2}:|A|\geq 1}}\E[Y_{A}]
&\leq \max_{\substack{A\subseteq \binom{V}{2}:|A|\geq 1}}\mathcal{B}_{2(\ell+1)}(\Vsize\pmax)^{\ell-|A|}=\mathcal{B}_{2(\ell+1)}(\Vsize \pmax)^{\ell-1}, \\
\max_{\substack{A\subseteq \binom{V}{2}}}\E[Y_{A}]
&=\max\left\{ \max_{\substack{A\subseteq \binom{V}{2}:|A|\geq 1}}\E[Y_{A}], \E[d_\ell(V)] \right\}
\leq \mathcal{B}_{2(\ell+1)}\Vsize(\Vsize\pmax)^{\ell}.
\end{align*}
Thus from \cref{eq:KimVutoY}  
with $\lambda=(2(\ell-1)+C_7/2)\log \Vsize$ and $C_8=\sqrt{\ell!} \mathcal{B}_{2(\ell+1)} (16(\ell-1+C_7/2))^\ell$, 
we obtain
\begin{align}
\Pr\left[|Y-\E[Y]|\geq C_8\sqrt{\Vsize }(\log \Vsize)^\ell (\Vsize \pmax)^{\ell-1/2}\right]\leq 2\mathrm{e}^2/\Vsize^{C_7}.
\label{claim:dlV}
\end{align}
Combining \cref{eq:WdlV}, \cref{claim:WEWlower} and \cref{claim:dlV}, the following holds with probability at least $1-2e^2/\Vsize^{C_7}-1/\Vsize^{C_3}$:
\begin{align*}
%\label{eq:WdlV}
&\forall S_0, S_1, \ldots, S_\ell: \\
&W(S_0;S_1,\ldots,S_{\ell})
\leq
\E[W(S_0;S_1,\ldots,S_{\ell})]+C_9\sqrt{\Vsize }(\log \Vsize)^\ell (\Vsize \pmax)^{\ell-1/2}+(\ell+1)C_4 \Vsize (\Vsize\pmax)^{\ell-1/2}.
\end{align*}
Thus we obtain the claim \cref{claim:WEWupper} and combining the claims \cref{claim:WEWlower,claim:WEWupper} complete the proof of \cref{lem:WEWdiscrepancy}.
\end{proof}
%%%%%%%%%%%%%%%%%%
\subsection{Discrepancy between the expected value and the ideal value}
\label{sec:EW-Wh}

\begin{lemma} \label{lem:EWIdiscrepancy}
Suppose the same setting of \cref{lem:WIdiscrepancy}.
Then for any vertex subsets $S_0, S_1, \ldots, S_\ell$ 
and for any $i_*\in \{0\}\cup [\ell]$, 
a positive constant $C$ exists such that
\begin{align*}
\left|\E[\W(S_0;S_1,\ldots,S_\ell)]-\Wh(S_0;S_1,\ldots,S_\ell)\right|
\leq 
%\mathcal{B}_{\ell+1}
C|S_{i_*}|(\Vsize \pmax)^{\ell-1}.
\end{align*}
%where $C(\ell)$ is a function only depends on $\ell$.
\end{lemma}
\begin{proof}[Proof of \cref{lem:EWIdiscrepancy}]
We show
\begin{align*}
\sum_{s\in S_0}\prod_{i\in [\ell]}\E[\deg_{S_i}(s)]
\leq \E[W(S_0;S_1,\ldots,S_\ell)]
\leq \sum_{s\in S_0}\prod_{i\in [\ell]}\E[\deg_{S_i}(s)]+ \mathcal{B}_{\ell+1}|S_{i_*}|(\Vsize\pmax)^{\ell-1}
\end{align*}
for any $i_*\in \{0\}\cup [\ell]$.
The first inequality follows directly from the FKG inequality (\cref{lem:FKG}) since $\deg_{S_i}(s)$ is a monotone increase function on $(\Ind_e)_{e\in \binom{V}{2}}$ for every $i$.
Now we show the second inequality.
We write each element $\mathbf{s}\in \Slc$ as $\mathbf{s}=(s_0,s_1,\ldots,s_\ell)$.
Since $\E[W(S_0;S_1,\ldots,S_\ell)]=\sum_{\substack{\mathbf{s}\in \Slc}}\E\left[\prod_{i\in [\ell]}\Ind_{\{s_0,s_i\}}\right]$, we have
\begin{align*}
\E[W(S_0;S_1,\ldots,S_\ell)]
&=\sum_{\substack{\mathbf{s}\in \Slc:\\
|\Es(\mathbf{s})|=\ell
}}\E\left[\prod_{i\in [\ell]}\Ind_{\{s_0,s_i\}}\right]+
\sum_{\substack{\mathbf{s}\in \Slc:\\
|\Es(\mathbf{s})|\leq\ell-1
}}\E\left[\prod_{i\in [\ell]}\Ind_{\{s_0,s_i\}}\right].
\end{align*}
For the first term, since $s_i\neq s_j$ for any $i,j\in[\ell]$ ($i\neq j$) if $|\Es(\mathbf{s})|=\ell$, we obtain
\begin{align*}
\sum_{\substack{\mathbf{s}\in \Slc:\\
|\Es(\mathbf{s})|=\ell
}}\E\left[\prod_{i\in [\ell]}\Ind_{\{s_0,s_i\}}\right]
&=\sum_{\substack{\mathbf{s}\in \Slc:\\
|\Es(\mathbf{s})|=\ell
}}\prod_{i\in [\ell]}\E\left[\Ind_{\{s_0,s_i\}}\right]
\leq \sum_{\substack{\mathbf{s}\in \Slc}}\prod_{i\in [\ell]}\E\left[\Ind_{\{s_0,s_i\}}\right]
= \sum_{s\in S_0}\prod_{i\in [\ell]}\E[\deg_{S_i}(s)].
\end{align*}
For the second term, from \cref{lem:technique}, 
\begin{align*}
%\sum_{\substack{\mathbf{s}\in \Slc:\\ \Vs(\mathbf{s})\leq \ell}}\E\left[\prod_{i\in [\ell]}J_{\{s_0,s_i\}}\right]
%&=
\sum_{\substack{\mathbf{s}\in \Slc:\\
|\Es(\mathbf{s})|\leq\ell-1
}}\E\left[\prod_{e\in \Es(\mathbf{s})}\Ind_{e}\right]
\leq \sum_{\substack{\mathbf{s}\in \Slc:\\
|\Es(\mathbf{s})|\leq\ell-1
}}\pmax^{|\Es(\mathbf{s})|}
\leq \mathcal{B}_{\ell+1}|S_{i_*}|(\Vsize\pmax)^{\ell-1}.
\end{align*}
Note that $G(\mathbf{s})=(\Vs(\mathbf{s}),\Es(\mathbf{s}))$ is a connected graph for any $\mathbf{s}\in \mathbf{S}$. %star graph.
\if0
We complete the proof of \cref{lem:EWIdiscrepancy} by showing
\begin{align}
 I(S_0;S_1,\ldots, S_\ell)- |\mathcal{R}_{\ell+1}||S_{i_*}|(np)^{\ell-1}
\leq \sum_{s\in S_0}\prod_{i\in [\ell]}\E[\deg(s)]
\leq I(S_0;S_1,\ldots, S_\ell).
\end{align}
The second inequality is trivial since
\begin{align*}
\Wh(S_0;S_1,\ldots,S_\ell)
&=\sum_{\mathbf{s}\in \Slc}\prod_{i\in [\ell]}\E\left[J_{\{s_0,s_i\}}\right]
+\sum_{\mathbf{s}\in (\prod_{i=0}^\ell S_i) \setminus \Slc}\prod_{i\in [\ell]}P(s_0,s_i)
\geq \sum_{\mathbf{s}\in \Slc}\prod_{i\in [\ell]}\E\left[J_{\{s_0,s_i\}}\right].
\end{align*}
For the first inequality, we have
\begin{align*}
\sum_{\mathbf{s}\in (\prod_{i=0}^\ell S_i) \setminus \Slc}\prod_{i\in [\ell]}P(s_0,s_i)
&= \sum_{\substack{\mathbf{s}\in (\prod_{i=0}^\ell S_i) \setminus \Slc:\\ |\Vs(\mathbf{s})|\leq \ell}}\prod_{i\in [\ell]}P(s_0,s_i)
\leq \sum_{\substack{\mathbf{s}\in \prod_{i=0}^\ell S_i:\\ |\Es(\mathbf{s})|\leq \ell-1}} p^{|\Es(\mathbf{s})|+1}
%&\leq  \sum_{\substack{R\in \mathcal{R}_{\ell+1}:\\|R|\leq \ell}}p^{\ell}
%|S_{i_*}|n^{|R|-1}
\leq \mathcal{B}_{\ell+1}|S_{i_*}|(np)^{\ell-1}p.
\end{align*}

Thus we obtain the claim.\fi
\end{proof}
\subsection[Proof of key lemma]{Proof of key lemma (\cref{lem:technique})}\label{sec:technique}
To complete the proof of \cref{lem:WIdiscrepancy}, we show \cref{lem:technique} in this section.
\begin{proof}[Proof of \cref{lem:technique}]
It is easy to see that 
\begin{align*}
\forall \mathbf{s}\in \mathcal{S}: |\Vs(\mathbf{s})|-1\leq |\mathcal{F}(\mathbf{s})|\leq k
\end{align*}
since $G(\mathbf{s})=\bigl(\Vs(\mathbf{s}),\mathcal{\Es}(\mathbf{s})\bigr)$ is a connected graph from the assumption.
Hence we have
\begin{align}
\sum_{\mathbf{s}\in \mathcal{S}}p^{|\mathcal{\Es}(\mathbf{s})|}
&\leq \sum_{\mathbf{s}\in \mathcal{S}}p^{|\Vs(\mathbf{s})|-1}
= \sum_{\mathbf{s}\in \mathcal{S}:|\Vs(\mathbf{s})|\leq k+1}p^{|\Vs(\mathbf{s})|-1}
\leq \sum_{\mathbf{s}\in \prod_{i=0}^{l}S_i:|\Vs(\mathbf{s})|\leq k+1}p^{|\Vs(\mathbf{s})|-1}. \label{eqn:C2_1}
\end{align}
To estimate above, we introduce the following notations.
For any $(l+1)$-dimensional vector $\mathbf{s}=(s_0, s_1, \ldots, s_{l})\in S_0\times S_1\times \cdots \times S_l$, 
let 
\begin{align*}
R(\mathbf{s})&\defeq \bigl\{ \{j\in \{0\}\cup [l]: s_j=s_i\} : i\in \{0\}\cup [l]\bigr\}.
\end{align*}
For example, $R\bigl((a,b,a,c,d,b,f,a)\bigr)=\bigl\{\{0,2,7\}, \{1,5\}, \{3\}, \{4\}, \{6\}\bigr\}$.
Note that $R(\mathbf{s})$ is a partition of $\{0\}\cup [l]$.
From the definition, we have $|R(\mathbf{s})|=|\Vs(\mathbf{s})|$. 
For example, $|U\bigl((a,b,a,c,d,b,f,a)\bigr)|=|\{a,b,c,d,f\}|=5=|R\bigl((a,b,a,c,d,b,f,a)\bigr)|$.
Let $\mathcal{R}_l$ be the family of all partitions of $\{0\}\cup [l]$. 
For example, 
\begin{align*}
\mathcal{R}_2
&=
\Bigl\{
\bigl\{ \{0\},\{1\},\{2\} \bigr\},
\bigl\{ \{0\},\{1,2\} \bigr\},
\bigl\{ \{1\},\{0,2\} \bigr\},
\bigl\{ \{2\},\{0,1\} \bigr\},
\bigl\{ \{0,1,2\} \bigr\}
\Bigr\}.
\end{align*}
Note that $|\mathcal{R}_l|=\mathcal{B}_{l+1}$.
Then we have
\begin{align}
\sum_{\substack{\mathbf{s}\in \prod_{i=0}^l S_i:\\ |\Vs(\mathbf{s})| \leq k+1}}p^{|\Vs(\mathbf{s})|}
&=\sum_{\substack{R\in \mathcal{R}_{l}:\\|R|\leq k+1}}\sum_{\substack{\mathbf{s}\in \prod_{i=0}^l S_i:  \\ R(\mathbf{s})=R}}p^{|\Vs(\mathbf{s})|}
=\sum_{\substack{R\in \mathcal{R}_{l}:\\|R|\leq k+1}}p^{|R|}\left|\left\{\mathbf{s}\in \prod_{i=0}^l S_i: R(\mathbf{s})=R\right\}\right|. \label{eqn:C2_2}
\end{align}
From the definition of $R(\mathbf{s})$, for any $r \in R(\mathbf{s})$, $s_i=s_j$ for any $i,j \in r$.
Thus
\begin{align}
\left|\left\{\mathbf{s}\in \prod_{i=0}^l S_i: R(\mathbf{s})=R\right\}\right|
&= \sum_{\substack{\mathbf{s}\in \prod_{i=0}^l S_i:\\ R(\mathbf{s})=R}}1
\leq \sum_{s_0\in S_0}\sum_{s_1\in S_1}\cdots \sum_{s_l\in S_l}\prod_{r\in R} \prod_{i,j\in r}\mathbbm{1}_{s_i=s_j}
\leq \prod_{r\in R}\left|\bigcap_{i\in r}S_i\right|.
\label{eq:equivalenceexample1}
\end{align}
For example, 
consider four vertex subsets $S_0, S_1, S_2, S_3$, let $R=\bigl\{ \{0,1\},\{2\},\{3\} \bigr\}\in \mathcal{R}_3$ and let $l=\{i_*\}\subseteq \{0\}\cup [3]$ where $i_*\in \{0\}\cup [3]$.
Then \cref{eq:equivalenceexample1} means that
\begin{align*}
\lefteqn{
\left|\left\{\mathbf{s}\in \prod_{i=0}^3 S_i: R(\mathbf{s})=R\right\}\right|}\\
&=\left|\left\{(s_0,s_1,s_2,s_3)\in S_0\times S_1\times S_2\times S_3: s_0=s_1, s_0\neq s_2, s_0\neq s_3, s_2\neq s_3\right\}\right|\\
&\leq \sum_{s_0\in S_0}\sum_{s_1\in S_1}\sum_{s_2\in S_2}\sum_{s_3\in S_3}\mathbbm{1}_{s_0=s_1}\leq |S_0\cap S_1||S_2||S_3|
=\prod_{r\in \{ \{0,1\},\{2\},\{3\} \}}\left|\bigcap_{i\in r}S_i\right|.
%&\leq \Vsize^3\frac{|S_{i_*}|}{\Vsize}=|S_{i_*}|\Vsize^2.
\end{align*}
For an index $i\in \{0\}\cup [l]$, let $r_i$ be the element of $R$ such that $r_i\ni i$.
Now let us consider the set $L$ described in the statement (of \cref{lem:technique}).
First we assume that there exist $i,j\in L$ with $i\neq j$ such that both $i$ and $j$ in the same $r_*=r_i=r_j\in R$.
In this case, since $S_i\cap S_j=\emptyset$ from the definition of $L$, we have
\begin{align}
\prod_{r\in R}\left|\bigcap_{i\in r}S_i\right|=\left|\bigcap_{i\in r_*}S_i\right|\prod_{r\in R\setminus r_*}\left|\bigcap_{i\in r}S_i\right|=0. \label{eqn:C2_4}
\end{align}
Now we assume that $r_i\neq r_j$ for any $i, j\in L$.
Then since $|\{r_i: i\in L\}|=|L|$ and $R=\{r_i: i\in L\} \cup R\setminus \{r_i: i\in L\}$, we have
\begin{align}
\prod_{r\in R}\left|\bigcap_{i\in r}S_i\right|
=\prod_{i\in L}\left|\bigcap_{j\in r_i}S_j\right|\prod_{r\in R\setminus \{r_i: i\in L\}}\left|\bigcap_{j\in r}S_j\right|
\leq \left(\prod_{i\in L}|S_i|\right) \Vsize^{|R|-|L|}. \label{eqn:C2_5}
\end{align}
Finally, by combining \cref{eqn:C2_1,eqn:C2_2,eqn:C2_4,eq:equivalenceexample1,eqn:C2_5}, we obtain
\begin{align*}
\sum_{\mathbf{s}\in \mathcal{S}}p^{|\mathcal{\Es}(\mathbf{s})|}
&\leq \frac{1}{p}\sum_{\substack{R\in \mathcal{R}_{l}:\\|R|\leq k+1}}p^{|R|}\left|\left\{\mathbf{s}\in \prod_{i=0}^l S_i: R(\mathbf{s})=R\right\}\right| \\
&\leq  \frac{1}{p}\sum_{\substack{R\in \mathcal{R}_{l}:\\|R|\leq k+1}} p^{|R|} \Vsize^{|R|}\frac{\prod_{i\in L}|S_{i}|}{\Vsize^{|L|}}
\leq \frac{1}{p}\left(\frac{\prod_{i\in L}|S_{i}|}{\Vsize^{|L|}}\right) (\Vsize \pmax)^{k+1} \sum_{\substack{R\in \mathcal{R}_{l}:\\|R|\leq k+1}} 1\\
&\leq |\mathcal{R}_{l}| \left(\frac{\prod_{i\in l}|S_{i}|}{\Vsize^{|L|}}\right) \Vsize (\Vsize \pmax)^k
= \mathcal{B}_{l+1}\left(\frac{\prod_{i\in L}|S_{i}|}{\Vsize^{|L|}}\right) \Vsize (\Vsize \pmax)^k.
\end{align*}
Note that the third inequality follows since $\Vsize p\geq 1$ from the assumption.
\end{proof}

\begin{proof}[Proof of \cref{lem:WIdiscrepancy}]
Combining \cref{lem:WEWdiscrepancy,lem:EWIdiscrepancy}, we obtain the proof.
\end{proof}

\subsection[Proof of (P3) of Theorem 3.2]{Proof of \ref{pro:gomiA} of \cref{thm:fgood}}\label{sec:gomiA}
The property \ref{pro:gomiA} is obtained by the following two lemmas.
%To show \cref{lem:WIdiscrepancy2}, we use the following concentration result of the maximum and minimum degree of the stochastic block model.
\begin{lemma}\label{lem:maxmindegree}
Suppose that $0\leq q\leq p=\omega(\log n/n)$.
Then two positive constants $C_1, C_2$ exist such that $G(2n,p,q)$ satisfies the following with probability $1-O(n^{-C_1})$:
\begin{align*}
\forall v\in V: |\deg(v)-n(p+q)|\leq C_2\sqrt{np\log n}.
\end{align*}
\end{lemma}
\begin{proof}
Applying the Chernoff bound (Lemma~\ref{lem:chernoff-additive}), 
% with $t=C\sqrt{np\log n}$ for some sufficiently large $C$ depends on $c_1$ and $c$, we obtain 
\begin{align*}
\lefteqn{
\Pr\left[\exists v\in V: |\deg(v)-n(p+q)|> t\right]}\\
&\leq \sum_{v\in V}\left(\exp\left(-\frac{t^2}{3\E[\deg(v)]}\right)+\exp\left(-\frac{t}{3}\right)+\exp\left(-\frac{t^2}{2\E[\deg(v)]}\right)\right)\\
&\leq n\left(2\exp\left(-\frac{t^2}{6np}\right)+ \exp\left(-\frac{t}{3}\right)\right)\\
&\leq 2\exp\left(\log n-\frac{t^2}{6np}\right)+ \exp\left(\log n-\frac{t}{3}\right).
%2\exp\left(\log n- \frac{t^2}{3\E[X]}\right)+\exp\left(\log n- \frac{t}{3}\right)\\
%&\leq 2\exp\left(\log n- \frac{C^2\log n}{6}\right)+\exp\left(\log n- \frac{C\sqrt{c_1}\log n}{3}\right)\\
%&\leq n^{-c}
\end{align*}
Note that $\E[\deg(v)]=(n-1)p+nq$ for any $v\in V$ and
$\E[\deg(v)]\leq n(p+q)\leq 2np$.
Thus we obtain the claim letting $t=C_2\sqrt{np\log n}$ since
%\begin{align*}
%\frac{t^2}{\E[\deg(v)]}
%=\frac{C^2np\log n}{\E[\deg(v)]}
%\geq \frac{C^2\log n}{2}
%\end{align*}
%and
$
t
=C_2\sqrt{np\log n}
\geq C\log n
$ for some constant $C$.

\end{proof}
\begin{lemma}
\label{lem:WIdiscrepancy2}
Suppose that $0\leq q\leq p=\omega(\log n/n)$.
Let $\mathcal{S}(A)=\bigl\{ S\cap U: S\in \{A,V\setminus A, V\}, U\in \{V_1, V_2, V\} \bigr\}$ for $A\subseteq V$.
For any constant $\ell$, two positive constants $C_1, C_2$ exist such that $G(2n,p,q)$ satisfies the following with probability $1-O(n^{-C_1})$:
\begin{align*}
\lefteqn{\forall A\subseteq V, \forall S_0, \ldots, S_{\ell-1}\in \mathcal{S}(A):}\\
%\ \textrm{s.t. $S_i\in \mathcal{S}(A)$ for all $i\in \{0\}\cup [\ell-1]$}:} \\
&\bigl|\W(S_0;S_1,\ldots,S_{\ell-1}, A)-\Wh(S_0;S_1,\ldots,S_{\ell-1}, A)\bigr|\leq  C_2 |A|\sqrt{\log n}(np)^{\ell-1/2}.
\end{align*}
\end{lemma}
\begin{proof}%[Proof of \cref{WIdiscrepancy2}]

\ 

\noindent \textbf{Lower bound.}
First we claim the following:
Two positive constants $C_3, C_4$ exist such that the following holds with probability $1-n^{-C_3}$: 
\begin{align}
&\forall A\subseteq V, \forall S_0, \ldots, S_{\ell-1}\in \mathcal{S}(A):\nonumber \\
&\W(S_0;S_1,\ldots,S_{\ell-1}, A)\geq \Wh(S_0;S_1,\ldots,S_{\ell-1}, A)-C_4|A|\sqrt{\log n}(np)^{\ell-1/2}.
\label{claim:gomiAlower}
\end{align}
From Janson's inequality (\cref{lem:Janson}) and \cref{claim:nablaupper} with a constant $C_5$ and $C_6=\sqrt{2(C_5+1)\mathcal{B}_{2(\ell+1)}}$, we have
\begin{align}
&\Pr\left[ \substack{\exists A\subseteq V, \\ \exists S_0, \ldots, S_{\ell-1}\in \mathcal{S}(A)}: \W(S_0;S_1,\ldots,S_{\ell-1}, A)\leq \E[\W(S_0;S_1,\ldots,S_{\ell-1}, A)] -C_6|A|\sqrt{\log \Vsize}(\Vsize\pmax)^{\ell-1/2} \right] \nonumber \\
&\leq \binom{\Vsize}{|A|} |\mathcal{S}(A)|^\ell \exp\left(-\frac{2(C_5+1)\mathcal{B}_{2(\ell+1)}|A| (\log \Vsize) (\Vsize\pmax)^{2\ell-1}}{2\mathcal{B}_{2(\ell+1)}|A|(\Vsize\pmax)^{2\ell-1}}\right) \nonumber  \\
& \leq 9^\ell \exp\left(|A|\log \Vsize-(C_5+1)|A|\log \Vsize\right)
\leq 9^\ell/\Vsize^{C_5}.
\label{eq:gomijan1}
\end{align}
Thus combining \cref{eq:gomijan1} and \cref{lem:EWIdiscrepancy} yields the claim \cref{claim:gomiAlower}. 
% it holds w.h.p.~that
%\begin{align*}
%\W(S_0;S_1,\ldots,S_{\ell-1}, A)\leq \Wh(S_0;S_1,\ldots,S_{\ell-1}, A) -O(|A|\sqrt{\log n}(np)^{\ell-1/2}).
%\end{align*}

\noindent \textbf{Upper bound.}
Now we show the following claim:
Two positive constants $C_7, C_8$ exist such that the following holds with probability $1-n^{-C_7}$: 
\begin{align}
&\forall A\subseteq V, \forall S_0, \ldots, S_{\ell-1}\in \mathcal{S}(A):\nonumber \\
&\W(S_0;S_1,\ldots,S_{\ell-1}, A)\leq \Wh(S_0;S_1,\ldots,S_{\ell-1}, A)+C_8|A|\sqrt{\log n}(np)^{\ell-1/2}.
\label{claim:gomiAupper}
\end{align}
From the same discussion of \eqref{eq:WdlV}, 
\begin{align}
\W(S_0;S_1,\ldots,S_{\ell-1}, A)
&=\W(\overbrace{V;V,\ldots, V}^{\ell},A)-\sum_{i=0}^{\ell-1}W(\overbrace{V;V,\ldots,V}^{i},V\setminus S_{i}, S_{i+1}, \ldots, S_{\ell-1},A)
\label{eq:gimiAsum}
\end{align}
since $W_{\ell}-W_0=\sum_{i=0}^{\ell-1}(W_{i+1}-W_{i})$.
Thus we consider an upper bound on $\W(S_0;S_1,\ldots,S_{\ell-1}, A)$. 
Let $d_{\max}\defeq \max_{v\in V}\deg(v)$.
Since $\sum_{v\in V}\deg_A(v)=\sum_{a\in A}\deg(v)$, we have
\begin{align*}
\W(\overbrace{V;V,\ldots, V}^{\ell},A)
&=\sum_{v\in V}\deg(v)^{\ell-1}\deg_A(v)
\leq d_{\max}^{\ell-1}\sum_{v\in V}\deg_A(v)
\leq d_{\max}^{\ell-1}\sum_{a\in A}\deg(a)
\leq d_{\max}^{\ell}|A|.
\end{align*}
From \cref{lem:maxmindegree}, it holds with high probability that
\begin{align*}
d_{\max}^\ell
&=\left(n(p+q)+O(\sqrt{np\log n})\right)^{\ell}
=\bigl(n(p+q)\bigr)^\ell \left(1+O\left(\sqrt{\frac{\log n}{np}}\right)\right).
\end{align*}
The second equality holds since $(\log n)/(np)=o(1)$ and $\ell$ is a constant.
Hence we have
%Since $d_{\max}\leq n(p+q)+C_9\sqrt{np\log n}$ with probability $1-n^{C_{10}}$ from Lemma~\ref{lem:maxmindegree}, there exists a constant $C_{11}$ such that
\begin{align}
\W(\overbrace{V;V,\ldots, V}^{\ell},A)
&\leq 
d_{\max}^{\ell}|A|
=|A|\bigl(n(p+q)\bigr)^\ell \left(1+O\left(\sqrt{\frac{\log n}{np}}\right)\right) \nonumber \\
&\leq \Wh(\overbrace{V;V,\ldots, V}^{\ell},A)+O(|A|\sqrt{\log n} (np)^{\ell-1/2}).
\label{eq:gomiAlast}
\end{align}
Note that $\Wh(\overbrace{V;V,\ldots,V}^{\ell},A)=|A|\bigl((n-1)p+nq\bigr)^{\ell}$. % and
%$\left(1+C_9\sqrt{\frac{\log n}{np}}\right)^{\ell}=1+O\left(\sqrt{\frac{\log n}{np}}\right)$ since $(\log n)/(np)=o(1)$.
Thus we obtain the claim \cref{claim:gomiAupper} by applying \cref{eq:gomiAlast,claim:gomiAlower} to \cref{eq:gimiAsum}.
Combining the claims \cref{claim:gomiAupper,claim:gomiAlower} complete the proof of \cref{lem:WIdiscrepancy2}.
\end{proof}
\begin{proof}[Proof of  \ref{pro:gomiA} of \cref{thm:fgood}]
Let $d_{\min}\defeq \min_{v\in V}\deg(v)$.
Then for any $j\in [\ell]$, 
\begin{align*}
\sum_{s\in S\cap V_i}\left(\frac{\deg_A(s)}{\deg(s)}\right)^j
&\leq d_{\min}^{-j}\sum_{s\in S\cap V_i}\deg_A(s)^j
=d_{\min}^{-j}W(S\cap V_i;\overbrace{A,\ldots,A}^{j}).
\end{align*}
From \cref{lem:maxmindegree}, it holds with high probability that
\begin{align*}
d_{\min}^{-j}
&=\left(n(p+q)-O(\sqrt{np\log n})\right)^{-j}
=\left(1+O\left(\sqrt{\frac{\log n}{np}}\right)\right)\bigl(n(p+q)\bigr)^{-j}.
\end{align*}
The second equality holds since $(\log n)/(np)=o(1)$ and $j\in [\ell]$ is a constant.
Thus from \cref{lem:WIdiscrepancy2}, we have
\begin{align*}
\sum_{s\in S\cap V_i}\left(\frac{\deg_A(s)}{\deg(s)}\right)^j
&=\left(1+O\left(\sqrt{\frac{\log n}{np}}\right)\right)\left(\frac{\Wh(S\cap V_i;\overbrace{A,\ldots,A}^{j})}{(n(p+q))^j}+O\left(|A|\sqrt{\frac{\log n}{np}}\right)\right)\\
&\leq |S\cap V_i|\left(\frac{|A_i|p+|A_{3-i}|q}{n(p+q)}\right)^j+O\left(|A|\sqrt{\frac{\log n}{np}}\right).
\end{align*}
Note that $\frac{|S\cap V_i|(|A_i|p+|A_{3-i}|q)^{j}}{(n(p+q))^j}= \frac{|S\cap V_i|(|A_i|p+|A_{3-i}|q)}{n(p+q)}\left(\frac{|A_i|p+|A_{3-i}|q}{n(p+q)}\right)^{j-1}\leq\frac{|A|p}{(p+q)}\leq  |A|$.
Thus we obtain the claim.
\end{proof}

\section{Proofs of general results of dynamical systems}\label{sec:generaljacobi}
In this section, we consider a polynomial voter process according to $f_1$ and $f_2$ on $G(2n,p,q)$ that is both $f_1$-good and $f_2$-good.
Moreover, we assume that $q/p$ is a constant.
Let $H=(H_1,H_2):[0,1]^2\to[0,1]^2$ be the induced dynamical system.
Throughout this section, probability and expectation are taken over the voter process unless otherwise noted.

\subsection[dynamics at sink points]{Proof of stationary dynamics around a sink point (\cref{prop:sinkpoint})} \label{sec:sinkpoint_proof}
We begin with establishing two auxiliary results.
\begin{lemma} \label{lem:constantfar}
For any $\epsilon = \omega(\sqrt{1/np})$, there exists an absolute constant $C>0$ satisfying
\begin{align*}
\Pr\left[\|{\boldsymbol \alpha'}-H({\boldsymbol \alpha})\|_\infty \geq\epsilon\right]\leq O(\exp(-C \epsilon^2 n)).
\end{align*}
\end{lemma}
\begin{proof}
From \cref{thm:mainthm_E} and $\epsilon = \omega(\sqrt{1/np})$, we have
\begin{align*}
\Pr\left[\|{\boldsymbol \alpha'}-H({\boldsymbol \alpha})\|_\infty\geq\epsilon\right]
&\leq \sum_{i\in[2]}\Pr\left[||A'_i|-\E[|A_i'|]|\geq \epsilon n-O(\sqrt{n/p})\right]\\
&\leq \sum_{i\in[2]}\Pr\left[||A'_i|-\E[|A_i'|]|\geq \frac{\epsilon n}{2}\right]
\end{align*}
for sufficiently large $n$.
Hence, from the Hoeffding inequality (\cref{lem:Hoeffding}), we have
\begin{align*}
\Pr\left[||A'_i|-\E[|A_i'|]|\geq \frac{\epsilon n}{2}\right]
&\leq O\left(\exp\left(-C\epsilon^2 n\right)\right).
\end{align*}
\end{proof}

\begin{lemma} \label{lem:sinkdistance}
Let $\mathbf{x}^*$ be a sink point of $H$.
Then, for sufficiently small $\epsilon>0$, there exists a constant $K=K(H)>0$ such that
\begin{align*}
\inf\{\|H(\mathbf{x})-\mathbf{y}\|_\infty:\,\mathbf{x}\in B(\mathbf{x}^*,\epsilon),\mathbf{y}\not\in B(\mathbf{x}^*,\epsilon)\}\geq K\epsilon.
\end{align*}
\end{lemma}
\begin{proof}
Let $\mathbf{x}^*$ be a sink point and $J$ be the Jacobian matrix of $H$ at $\mathbf{x}^*$.
From the Taylor expansion (see, e.g.~Theorem 12.15 of \cite{Krantz16}), we have
\begin{align}
H(\mathbf{x})&=H(\mathbf{x}^*)+J(\mathbf{x}-\mathbf{x}^*)+O_{\mathbf{x}\to\mathbf{x}^*}(\|\mathbf{x}-\mathbf{x}^*\|_2^2) \nonumber\\
&=\mathbf{x}^*+J(\mathbf{x}-\mathbf{x}^*)+O_{\mathbf{x}\to\mathbf{x}^*}(\|\mathbf{x}-\mathbf{x}^*\|_2^2). \label{eqn:Taylorexpansion}
\end{align}
By the property of singular value (\cref{prop:singularvalue}), there exist constants $\epsilon,K>0$ such that, for any $\mathbf{x}\in B(\mathbf{x}^*,\epsilon)$, it holds that
\begin{align*}
\|H(\mathbf{x})-\mathbf{x}^*\|_2&=\|H(\mathbf{x})-H(\mathbf{x}^*)\|_2
\leq \sigma_{\max}\|\mathbf{x}-\mathbf{x}^*\|_2+O_{\mathbf{x}\to\mathbf{x}^*}(\|\mathbf{x}-\mathbf{x}^*\|_2^2)\\
&\leq(1-K)\|\mathbf{x}-\mathbf{x}^*\|_\infty
<(1-K)\epsilon.
\end{align*}
Consequently, for any $\mathbf{y}\not\in B(\mathbf{x}^*,\epsilon)$, we have
\begin{align*}
\epsilon&\leq \|\mathbf{y}-\mathbf{x}^*\|_\infty
\leq \|H(\mathbf{x})-\mathbf{y}\|_\infty+\|H(\mathbf{x})-\mathbf{x}^*\|_\infty
< \|H(\mathbf{x})-\mathbf{y}\|_\infty+(1-K)\epsilon.
\end{align*}
Equivalently, we obtain $\|H(\mathbf{x})-\mathbf{y}\|_\infty>K\epsilon$.
\end{proof}

\begin{proof}[Proof of \cref{prop:sinkpoint}]
Let $\mathbf{a}^*$ be a sink point of $H$.
From \cref{lem:sinkdistance}, we can take a constant $K>0$ satisfying
$K\epsilon \leq \inf\{\|H(\mathbf{x})-\mathbf{y}\|_\infty:\,\mathbf{x}\in B(\mathbf{a}^*,\epsilon),\mathbf{y}\not\in B(\mathbf{a}^*,\epsilon)\}$.
Then, \cref{lem:constantfar} implies
\begin{align*}
\Pr\left[{\boldsymbol \alpha}'\not\in B(\mathbf{a}^*,\epsilon) \mid {\boldsymbol \alpha}\in B(\mathbf{a}^*,\epsilon)\right]
&\leq \Pr\left[\|{\boldsymbol \alpha}'-H({\boldsymbol \alpha})\|_\infty>K \epsilon\right]\\
&\leq \exp(-\Omega(\epsilon^2 n)).
\end{align*}
For any $T$, the union bound over the time $t=1,\ldots,T$ leads to
\begin{align*}
\Pr\left[\exists t\in\{1,\ldots,T\}:{\boldsymbol \alpha}^{(t)}\not\in B(\mathbf{a}^*,\epsilon) \,\middle|\, {\boldsymbol \alpha}^{(0)}\in B(\mathbf{a}^*,\epsilon)\right] \leq T\exp(-\Omega(\epsilon^2 n))
\end{align*}
and thus, the stopping time $\tau$ satisfies $\tau\geq\exp(\Omega(\epsilon^2 n))$ w.h.p.~if $\epsilon = \omega(\max\{\sqrt{\log n/n}, \sqrt{1/np}\})$.
\end{proof}

\subsection[Proof of fast consensus]{Proof of fast consensus results (\cref{prop:fastconsensus_Jacob,prop:fastconsensus})} \label{sec:fastconsensus_proof}
\paragraph*{Derive \cref{prop:fastconsensus_Jacob} from \cref{prop:fastconsensus}}
It suffices to check the condition of \cref{prop:fastconsensus} for $\epsilon=\epsilon(n)=\Theta\left(\sqrt{\frac{\log n}{np}}\right)$.
Using \ref{pro:gomiA} and the Taylor expansion \cref{eqn:Taylorexpansion}, there exists $C=C(H)$ such that
\begin{align*}
\E[|A_i'|]
&=n H_i(\alpha_1, \alpha_2)\pm O\left(|A|\sqrt{\frac{\log n}{np}}\right)
\leq n C\left((\alpha_1+\alpha_2)^2+ |A|\sqrt{\frac{\log n}{np}}\right) \\
&=C\frac{|A|^2}{n}+ C|A|\sqrt{\frac{\log n}{np}}
\end{align*}
holds if $\|{\boldsymbol \alpha}\|\leq \delta$ for sufficiently small constant $\delta$.
\QED

\paragraph*{Proof of \cref{prop:fastconsensus}}

We prove \cref{prop:fastconsensus}.
Let $n=|V(G)|$.
Suppose that there exist absolute constants $C,\delta>0$ and a function $\epsilon=\epsilon(n)$ such that $\epsilon(n)=o(1)$ and $\E[|A'|]\leq \frac{C|A|^2}{n}+\epsilon|A|$ holds for all $A\subseteq V$ of $|A|\leq \delta n$.
Note that we may assume $\epsilon(n)=\Omega(\sqrt{\log n/n})$: If $\epsilon=o(\sqrt{\log n/n})$, we have $\log n/\log \epsilon^{-1}=O(1)$ and we will obtain the claim by applying \cref{prop:fastconsensus} with letting $\epsilon=\sqrt{\log n/n}$.

Take a positive constant $\delta'$ such that
\begin{align}
&10\left(\frac{CM^2}{n}+\epsilon M\right) \leq M, \label{eqn:delta'condition1}\\
&\delta'\leq \min\left\{\delta,\frac{1}{16C}\right\} \label{eqn:delta'condition2}
\end{align}
hold for any $0\leq M\leq \delta'n$.
We can take such constant $\delta'>0$ since $\epsilon=o(1)$ and thus the inequality \cref{eqn:delta'condition1} holds if the ratio $\frac{M}{n}$ is sufficiently small.

Consider $A^{(0)},A^{(1)},\ldots$ given by the voting model such that $|A^{(0)}|\leq \delta n$.
To exploit the assumption of the expectation, we first claim that $|A^{(t)}|\leq \delta' n \leq \delta n$ holds w.h.p.~for all $t=0,\ldots,n^{o(1)}$.
Let $\mathcal{B}^{(t)}$ be the event that $|A^{(i)}|\leq \delta n$ for all $i=0,\ldots,t$.
Note that $\mathcal{B}^{(0)}$ holds.
Consider $\Pr[\mathcal{B}^{(t+1)}|\mathcal{B}^{(t)}]$.
If $\E[|A'|]\geq\log n$, the Chernoff bound (\ref{eqn:cher1} of \cref{lem:chernoff}) implies
\begin{align*}
\Pr\left[|A'|\geq 10\left(\frac{C|A|^2}{n}+\epsilon|A|\right) \,\middle|\, |A|\leq \delta'n\right]
\leq \Pr\left[|A'|\geq 10\E[|A'|]\right]
\leq \exp\left(-\frac{10}{3}\log n\right)
\leq n^{-3}.
\end{align*}
Then, for $|A|\leq \delta'n$, it holds with probability $1-O(n^{-3})$ that
\begin{align}
|A'|\leq 10\left(\frac{C|A|^2}{n}+\epsilon|A|\right)\leq |A|\leq\delta' n. \label{eqn:A'bound1}
\end{align}
Here, we used \cref{eqn:delta'condition1} with letting $M=|A|$.
If $\E[|A'|]\leq\log n$, from the Chernoff bound (\ref{eqn:cher3} of \cref{lem:chernoff}), we have %lem:chernoff-k
\begin{align}
|A'|\leq 6\log n =o(\delta' n) \label{eqn:A'bound2}
\end{align}
with probability at least $1-O(n^{-3})$.
From \cref{eqn:A'bound1,eqn:A'bound2}, we obtain $\Pr\left[\mathcal{B}^{(t+1)}\,\middle|\, \mathcal{B}^{(t)}\right]\geq 1-O(n^{-3})$ for each $t$ and thus $\mathcal{B}^{(t)}$ holds for $t=n^{0.01}$ with probability $1-O(n^{-2.99})$.

Now we look at $|A^{(t)}|$.
Note that
\begin{align*}
\E[|A'|\mid |A|\leq \delta' n] \leq 
\begin{cases}
\frac{2C|A|^2}{n} & \text{if $ \frac{\epsilon}{C}n \leq |A| \leq \delta'n$},\\
2\epsilon |A| & \text{if $0\leq |A| \leq \frac{\epsilon}{C}$}.
\end{cases}
\end{align*}
Thus let us consider the following two cases.
\paragraph*{Case I: $\frac{\epsilon}{C}n \leq |A|^{(t)} \leq \delta'n$.}
From the Chernoff bound (\ref{eqn:cher3} of \cref{lem:chernoff}), we have
\begin{align*}
\Pr\left[|A^{(t+1)}|\geq \frac{12C|A^{(t)}|^2}{n} \,\middle|\, \mathcal{B}^{(t)} \right]
&\leq 2^{-\Omega(\log n)}.
\end{align*}
In the last inequality, we used $|A^{(t)}|\geq \frac{\epsilon}{C} n = \Omega(\sqrt{n\log n})$.
Hence, conditioned on $\mathcal{B}^{(t)}$ and $|A|^{(i)}\geq \frac{\epsilon}{C}n$ ($i=0,\ldots,t$), it holds w.h.p.~that
\begin{align*}
|A^{(t)}| \leq \frac{12C (|A|^{(t-1)})^2}{n} \leq \frac{n}{12C}\left(\frac{12C |A^{(0)}|}{n}\right)^{2^t} \leq \frac{0.75^{2^t}n}{12C}.
\end{align*}
Here, we used \cref{eqn:delta'condition2}.
Therefore, for some $\tau_1=O(\log\log n)$, $|A^{(\tau_1)}| \leq \frac{\epsilon}{C}$ holds w.h.p.

\paragraph*{Case I\hspace{-.1em}I: $0\leq |A^{(t)}| \leq \frac{\epsilon}{C}n$.}
Conditioned on $|A^{(0)}| \leq \frac{\epsilon}{C}n$, we claim that $\E[|A^{(\tau_2)}|]\leq n^{-\Omega(1)}$ for some $\tau_2=O(\log n/\log \epsilon^{-1})$.
Note that this completes the proof of \cref{prop:fastconsensus} since $\Pr[A^{(\tau_2)}\neq \emptyset]\leq \E[|A^{(\tau_2)}|]=n^{-\Omega(1)}$ from the Markov inequality.

To show the claim, we exploit the property that $\E[|A'|\mid A]\leq 2\epsilon|A|$ if $|A|\leq \frac{\epsilon}{C}n$.
Before using this, we show that $|A^{(t)}|\leq \frac{\epsilon}{C}n$ holds for all $t=1,\ldots,n^{o(1)}$.
Conditioned on $|A|\leq\frac{\epsilon}{C}n$, we have $\E[|A'|]\leq 2\epsilon |A| \leq O(\epsilon^2 n)$ and thus the Chernoff bound (\ref{eqn:cher3} of \cref{lem:chernoff}) yields
\begin{align*}
\Pr\left[|A'|\geq \frac{\epsilon}{C}n \middle| |A|\leq \frac{\epsilon}{C}n\right] \leq 2^{-\Omega(\epsilon n)} = n^{-\Omega(1)}.
\end{align*}
Therefore, $|A^{(t)}|\leq \frac{\epsilon}{C}n$ holds for all $t=0,\ldots,n^{o(1)}$.
Let $\mathcal{C}^{(t)}$ be the event that $|A^{(i)}|\leq\frac{\epsilon}{C}n$ holds for all $i=0,\ldots,t$.
Then, from the tower property of the conditional expectation, we have
\begin{align*}
\E[|A^{(\tau_2)}|\mid \mathcal{C}^{(\tau_2)}] &\leq \E[\E[|A^{(\tau_2)}|\mid A^{(\tau_2-1)},\mathcal{C}^{(\tau_2)}]\mid \mathcal{C}^{(\tau_2)}] \\
&\leq \E[2\epsilon |A^{(\tau_2-1)}| \mid \mathcal{C}^{(\tau_2)}] \\
&\leq (2\epsilon)^{\tau_2} \cdot \frac{\epsilon n}{C} \\
&\leq n^{-\Omega(1)}
\end{align*}
for some $\tau_2=O(\log n/\log \epsilon^{-1})$.
This shows the aforementioned claim as well as completes the proof of \cref{prop:fastconsensus}.
\QED

\subsection[Proof of the escape result]{Proof of the escape result (\cref{prop:escape})} \label{sec:escape_proof}
In this section, we prove \cref{prop:escape}.
Let $\mathbf{a}^*$ be a fixed point satisfying \cref{asm:escape_assumption0}.
The proof of \cref{prop:escape} consists of two parts: We derive \cref{prop:escape} from \cref{asm:escape_assumption1,asm:escape_assumption2}.

Recall the random variable ${\boldsymbol \beta}$ defined in \cref{eqn:betadef}.
From the definition \cref{eqn:betadef}, each element $\beta_i$ of ${\boldsymbol \beta}$ can be rewritten as
\begin{align}
\beta_i&=\sum_{j=1}^2 u_{ij} \alpha_j-(U\mathbf{a}^*)_i, \label{eqn:beta_sum_of_irv}
\end{align}
where we let $U=(u_{ij})$.
Each element $u_{ij}$ of the matrix $U$ does not depend on $n$.
Hence, the Hoeffding bound (\cref{lem:Hoeffding}) implies
\begin{align}
\Pr[|\beta'_i-\E[\beta'_i]|\geq t]\leq \exp\left(-\Omega(t^2n)\right). \label{eqn:beta_concentration}
\end{align}
From \cref{thm:approximation_vectorfield} and the Taylor expansion \cref{eqn:Taylorexpansion}, we have
\begin{align*}
\E[{\boldsymbol \beta}' \mid A] &= \E[U({\boldsymbol \alpha}'-\mathbf{a}^*) \mid A]\\
&=U(H({\boldsymbol \alpha})-\mathbf{a}^*)+O\left(\frac{1}{\sqrt{np}}\right)\cdot \mathbf{1}\\
&=UJ({\boldsymbol \alpha}-\mathbf{a}^*)+\left(O_{{\boldsymbol \alpha}\to \mathbf{a}^*}(\|{\boldsymbol \alpha}-\mathbf{a}^*\|_\infty^2)+O\left(\frac{1}{\sqrt{np}}\right)\right)\cdot \mathbf{1}\\
&=\Lambda {\boldsymbol \beta}+\left(O_{\|{\boldsymbol \beta}\|\to 0}(\|{\boldsymbol \beta}\|_\infty^2)+O\left(\frac{1}{\sqrt{np}}\right)\right)\cdot \mathbf{1}.
\end{align*}
Hence, the $i$-th element $\beta_i$ of ${\boldsymbol \beta}=(\beta_1\,\beta_2)^{\top}$ satisfies
\begin{align}
|\E[\beta'_i]| = |\lambda_i\|\beta_i|+O_{\|{\boldsymbol \beta}\|\to 0}(\|{\boldsymbol \beta}\|_\infty^2)+O\left(\frac{1}{\sqrt{np}}\right). \label{eqn:exp_beta'}
\end{align}
It is convenient to consider the behavior of ${\boldsymbol \beta}$ instead of ${\boldsymbol \alpha}$.
Note that ${\boldsymbol \alpha}\to \mathbf{a}^*$ implies ${\boldsymbol \beta}\to \mathbf{0}$ and vice versa since the matrix $U$ is nonsingular.
By substituting $t=\Theta\left(\sqrt{\frac{\log n}{n}}\right)$ to \cref{eqn:beta_concentration}, for sufficiently large constant $C>0$, it holds w.h.p.~that
\begin{align}
\left| |\beta'_i| - |\lambda_i||\beta_i| \right| \leq C\|{\boldmath \beta}\|_{\infty}^2+ C\sqrt{\frac{\log n}{n}}. \label{eqn:beta'bounds}
\end{align}

\paragraph*{Derive \cref{prop:escape} from \cref{asm:escape_assumption1}} \label{sec:derivethm2.11fromasm2.9}
Suppose that the fixed point $\mathbf{a}^*$ satisfies \cref{asm:escape_assumption1}.
Let $I_{>1} \defeq \{i\in[2]:|\lambda_i|>1\}$ and $I_{\leq 1} \defeq [2]\setminus I_{>1}$.
Fix a sufficiently large constant $K>0$ and let $\epsilon^*$ be the constant mentioned in \cref{asm:escape_assumption1}.
Define
\begin{align*}
\mathcal{A}_1&=\left\{A\subseteq V\,:\, \|{\boldsymbol \beta}\|_\infty \leq \epsilon^* \text{ and }|\beta_j|< K\sqrt{\frac{\log n}{n}} \text{ for all }j\in I_{>1} \right\},\\
\mathcal{A}_2&=\left\{A\subseteq V\,:\, \|{\boldsymbol \beta}\|_\infty \leq \epsilon^* \text{ and }|\beta_j|\geq K\sqrt{\frac{\log n}{n}} \text{ for some }j\in I_{>1} \right\},\\
\mathcal{A}_3&=\left\{A\subseteq V\,:\, \|{\boldsymbol \beta}\|_\infty >\epsilon^* \text{ and }|\beta_j|\leq \epsilon^* \text{ for all }j\in I_{\leq 1} \right\}.
\end{align*}

We claim that, for each $i=1,2$ and any $A\in \mathcal{A}_i$, there exists $\tau=O(\log n)$ satisfying $\Pr\left[A^{(\tau)}\in \mathcal{A}_{i+1}\,\middle| A^{(0)}=A \right] \geq 1-n^{-\Omega(1)}$.
This completes the proof of \cref{prop:escape} under \cref{asm:escape_assumption1}.

\paragraph*{Case I: $A^{(0)}\in \mathcal{A}_1$.}
Let $f(A)\defeq \lfloor n\max\{|\beta_i|:i\in I_{>1}\}\rfloor$ and $m=K\sqrt{n\log n}$.
We use \cref{cor:nazocor} to show $A^{(\tau)}\in \mathcal{A}_2$ for some $\tau=O(\log n)$.
Here, we use $\mathcal{A}_1$ as $\mathcal{B}$ of \cref{cor:nazocor}.
Note that $A\in\mathcal{A}_1$ implies $f(A)<m$.

From \cref{eqn:beta_sum_of_irv} and \ref{asm:asm2}, we have $\Var[|\beta_i|  \mid  A]=\sum_{j\in[2]} u_{ij}^2 \Var[\alpha_j \mid A]=\Omega(n^{-1})$.
Here, note that, for every $i\in[2]$, there exists $j\in[2]$ such that $u_{ij}\neq 0$, since otherwise, it contradicts to the fact that the matrix $U$ is nonsingular.
Thus, from \cref{cor:BEbound_cor}, it holds that, for any constant $h>0$, there exists a positive constant $C_1<1$ such that $\Pr[f(A')< h\sqrt{n} \mid f(A)\leq m] < C_1.$
This verifies the condition~\ref{cond:nazo_sqrt} of \cref{cor:nazocor}.

Now we check the condition~\ref{cond:chernoff}.
Let $z\in[2]$ be the least index satisfying $|\beta_z|=\max\{|\beta_i|:i\in[2]\}$.
Suppose that $A\in \mathcal{A}_1$ satisfies $f(A)=\lfloor n|\beta_z|\rfloor \geq h\sqrt{n}$ for sufficiently large constant $h>\frac{100 C}{\epsilon_1}$ (recall that the constant $\epsilon_1$ is mentioned in \ref{asm:asm4}).
Then, from \ref{asm:asm4}, we have
\begin{align*}
|\E[\beta'_z \mid A]| &\geq (1+0.99\epsilon_1)|\beta_z| + 0.01\epsilon_1|\beta_z|-\frac{C}{\sqrt{n}}
\geq (1+0.99\epsilon_1)|\beta_z|.
\end{align*}
Thus, from the Hoeffding inequality (\cref{lem:Hoeffding}), we obtain
\begin{align*}
\Pr[f(A')<(1+0.98\epsilon_1)f(A) \mid A]
\leq
\Pr\left[f(A')<\frac{1+0.98\epsilon_1}{1+0.99\epsilon_1}\E[f(A')]\,\middle|\,A\right]
\leq
\exp\left(-\Omega\left(\frac{f(A)^2}{n}\right)\right)
\end{align*}
holds for every $A\in \mathcal{A}_1$ satisfying $f(A)\geq h\sqrt{n}$.
This verifies the condition~\ref{cond:chernoff}.

Finally, we check the condition~\ref{cond:whp} of \cref{cor:nazocor}.
From \ref{asm:asm5}, it holds that
\begin{align*}
\Pr[A'\not\in \mathcal{A}_1\text{ and }f(A')<m \mid A\in \mathcal{A}_1]
\leq
\Pr[\exists j\in I_{\leq 1},|\beta'_j|>\epsilon^* \mid A\in\mathcal{A}_1]
\leq
n^{-\Omega(1)}.
\end{align*}

Therefore, from \cref{cor:nazocor}, we have $f(A^{(\tau)})\geq m=K\sqrt{n\log n}$ (i.e.~$A^{(\tau)}\in \mathcal{A}_2$) holds w.h.p.~for some $\tau=O(\log n)$.

\paragraph*{Case I\hspace{-.1em}I: $A^{(0)}\in \mathcal{A}_2$.}
Suppose that $A^{(0)}\in\mathcal{A}_2$ and let $j\in I_{>1}$ be the index satisfying $|\beta_j|>K\sqrt{\frac{\log n}{n}}$.
We remark that $K$ is sufficiently large.
From \ref{asm:asm4} and \cref{eqn:beta'bounds}, we have $|\beta'_j|\geq (1+0.99\epsilon_1)|\beta_j|$.
Thus, for some $\tau=O(\log n)$, we have $|\beta^{(\tau)}_j|>\epsilon^*$.
Moreover, from \ref{asm:asm5}, we have $|\beta^{(\tau)}_i|\leq\epsilon^*$ for all $i\in I_{\leq 1}$.
Therefore, $A^{(\tau)}\in \mathcal{A}_3$ holds w.h.p.

\paragraph*{Derive \cref{prop:escape} from \cref{asm:escape_assumption2}}

Suppose that the fixed point $\mathbf{a}^*$ satisfies \cref{asm:escape_assumption2}.
Let
\begin{align*}
I_{<1}&\defeq\{i\in[2]:|\lambda_i|<1\}, \\
I_{>1}&\defeq\{i\in[2]:|\lambda_i|>1\}.
\end{align*}
From \ref{asm:asm6}, we have $I_{<1}\cup I_{>1}=[2]$.
Moreover, there exists some constant $\epsilon>0$ such that
\begin{align}
||\lambda_i|-1| > 3\epsilon \label{eqn:epsdef}
\end{align}
holds for every $i\in [2]$.
For $A\subseteq V$, let $z=z(A)\in[2]$ be the least index satisfying $|\beta_z|=\|{\boldsymbol \beta}\|_\infty$.
We use four constants: In \cref{eqn:beta'bounds,eqn:epsdef}, we defined $C$ and $\epsilon$.
Let $K\defeq\frac{C}{\epsilon}$ and $\epsilon'\defeq \frac{\epsilon}{C}$.
Consider four events
\begin{align*}
\mathcal{B}_1&=\left\{A\subseteq V\,:\, K\sqrt{\frac{\log n}{n}}<\|{\boldsymbol \beta}\|_\infty \leq \epsilon' \text{ and }z(A)\in I_{<1} \right\},\\
\mathcal{B}_2&=\left\{A\subseteq V\,:\, \|{\boldsymbol \beta}\|_\infty \leq K\sqrt{\frac{\log n}{n}} \right\}, \\
\mathcal{B}_3&=\left\{A\subseteq V\,:\, K\sqrt{\frac{\log n}{n}}<\|{\boldsymbol \beta}\|_\infty \leq \epsilon' \text{ and }z(A)\in I_{>1} \right\}, \\
\mathcal{B}_4&=\left\{A\subseteq V\,:\, \|{\boldsymbol \beta}\|_\infty > \epsilon' \text{ and } |\beta_j|\leq \epsilon'\text{ for all }j\in I_{<1} \right\}.
\end{align*}

We claim that, if $A^{(0)}\in \mathcal{B}_i$, then $A^{(\tau)}\not \in  \cup_{1\leq j\leq i}\mathcal{B}_{j}$ holds w.h.p.~for any $i\in\{1,2,3\}$ and some $\tau=O(\log n)$.
This completes the proof of \cref{prop:escape} under \cref{asm:escape_assumption2}.

\paragraph*{Case I: $A^{(0)}\in\mathcal{B}_1$.}
Suppose $A^{(0)}\in \mathcal{B}_1$.
We argue that, if $A^{(t)}\in \mathcal{B}_1$, then either $A^{(t+1)}\in \mathcal{B}_3$ or $\|{\boldsymbol \beta}^{(t+1)}\|_\infty \leq (1-\epsilon)\|{\boldsymbol \beta}^{(t)}\|_\infty$ holds w.h.p.
For any $j\in I_{<1}$, the bound \cref{eqn:beta'bounds} yields that
\begin{align*}
|\beta'_j| &\leq (1-3\epsilon)\|{\boldsymbol \beta}\|_\infty +C\|{\boldsymbol \beta}\|^2_\infty + C\sqrt{\frac{\log n}{n}} \\
&\leq (1-\epsilon)\|{\boldsymbol \beta}\|_\infty - 2\epsilon \|\beta\|_\infty + C\epsilon'\|\beta\|_\infty + C\sqrt{\frac{\log n}{n}} \\
&= (1-\epsilon)\|\beta\|_\infty - \epsilon\|\beta\|_\infty + C\sqrt{\frac{\log n}{n}} \\
&\leq (1-\epsilon)\|\beta\|_\infty-(\epsilon K-C)\sqrt{\frac{\log n}{n}} \\
&\leq (1-\epsilon)\|\beta\|_\infty
\end{align*}
holds w.h.p.
If $A^{(t+1)}\not\in\mathcal{B}_3$, then $\|{\boldsymbol \beta}'\|_\infty = |\beta'_j|$ for some $j\in I_{<1}$; thus, we have $\|{\boldsymbol \beta}'\|_\infty \leq (1-\epsilon)\|{\boldsymbol \beta}\|_\infty$ w.h.p.
Therefore, for some $\tau=O(\log n)$, either $A^{(\tau)}\in \mathcal{B}_2$ or $A^{(\tau)}\in \mathcal{B}_{3}$ holds w.h.p.

\paragraph*{Case I\hspace{-.1em}I: $A^{(0)}\in\mathcal{B}_2$.}
Suppose $A^{(0)}\in \mathcal{B}_2$.
Our strategy is to apply \cref{cor:nazocor}.
We will prove the following result in the last part of this subsection.
\begin{lemma} \label{lem:beta'bounds}
Conditioned on $A\in \mathcal{B}_2$, the followings hold w.h.p.:

\begin{enumerate}[label=(\roman*)]

\item\label{state:small} For every $i\in I_{<1}$, it holds that $|\beta'_i|\leq K\sqrt{\frac{\log n}{n}}$.

\item\label{state:large} there exists a constant $h>0$ such that, for every $i\in I_{>1}$, 
\begin{align*}
\left|\E\left[\beta'_i\,\middle|\,A,|\beta_i|\geq \frac{h}{\sqrt{n}}\right]\right| \geq (1+\epsilon) |\beta_i|.
\end{align*}
\end{enumerate}
\end{lemma}
Let $m=K\sqrt{n\log n}$ and define $f(A)\defeq\left\lfloor n\cdot \max_{i\in I_{>1}}|\beta_i| \right\rfloor$.
Suppose that $f(A^{(\tau)}) \geq K\sqrt{n\log n} $ holds w.h.p.~for some $\tau=O(\log n)$.
Then, we have $A^{(\tau)}\not\in\mathcal{B}_1\cup\mathcal{B}_2$ w.h.p.~since $|\beta_i^{(\tau)}|\leq K\sqrt{\frac{\log n}{n}}$ holds w.h.p.~for any $i\in I_{< 1}$.
Here, we used \ref{state:small} of \cref{lem:beta'bounds}.
To show $f(A^{(\tau)}) \geq K\sqrt{n\log n}$, we check the condition \ref{cond:nazo_sqrt} to \ref{cond:whp} of \cref{cor:nazocor} and then apply it.

First, we check the condition~\ref{cond:nazo_sqrt} of \cref{cor:nazocor}.
We use the same argument described in the Case I in \cref{sec:derivethm2.11fromasm2.9}.
From \cref{eqn:beta_sum_of_irv}, we have $\Var[\beta'_i \mid A]\geq \sum_{j=1}^2 u_{ij}^2 \Var[\alpha'_j \mid A]$.
Moreover, for every $i\in[2]$ there exists $j\in[2]$ such that $u_{ij}\neq 0$, since otherwise, it contradicts to the fact that $U$ is nonsingular.
From \ref{asm:asm2}, we have $\Var[\beta'_i \mid A]=\Omega(n^{-1})$; thus, from \cref{cor:BEbound_cor}, it holds that, for any constant $h>0$, there exists a positive constant $C_1<1$ such that $\Pr[f(A')\geq h\sqrt{n} \mid f(A)\leq m] < C_1$.

We check the condition~\ref{cond:chernoff} of \cref{cor:nazocor}.
For every $i\in I_{>1}$, \cref{lem:beta'bounds} yields
\begin{align}
|\E[\beta^{(t+1)}_i \mid A^{(t)}\in\mathcal{B}_2]| \geq (1+\epsilon) |\beta^{(t)}_i|. \label{eqn:beta_expect}
\end{align}
In look at \cref{eqn:beta_sum_of_irv}, from the Hoeffding inequality (\cref{lem:Hoeffding}), it holds for any set $A^{(t)}\in\mathcal{B}_2$, any index $i\in I_{>1}$ and any constant $\epsilon'>0$ that
\begin{align}
\Pr[|\beta'_i| \leq (1-\epsilon')|\E[\beta_i']|]&\leq \exp\left(-\Omega\left(\epsilon'^2\E[\beta'_i]^2n\right)\right)
\leq \exp\left(-\Omega\left(\epsilon'^2\frac{f(A)^2}{n}\right)\right). \label{eqn:beta_hoeffding}
\end{align}
From \cref{eqn:beta_expect,eqn:beta_hoeffding}, by letting $\epsilon'=\frac{\epsilon}{2(1+\epsilon)}$, we obtain
\begin{align*}
\Pr\left[|\beta^{(t+1)}_i| \leq \left(1+\frac{\epsilon}{2}\right)\cdot |\beta^{(t)}_i| \,\middle|\,A^{(t)}\in \mathcal{B}_2\right] 
&\leq \Pr\left[|\beta'_i| \leq (1-\epsilon') \cdot |\E[\beta_i'  \mid A^{(t)}\in\mathcal{B}_2] |\right]\\
&\leq \exp\left(-\Omega\left(\frac{f(A)^2}{n}\right)\right). 
\end{align*}
In other words, for any $A\in\mathcal{B}_2$ satisfying $f(A)\geq h\sqrt{n}$ for some constant $h>0$, we have
\begin{align*}
\Pr\left[f(A^{(t+1)})<\left(1+\frac{\epsilon}{2}\right)f(A^{(t)})\,\middle|\,A^{(t)}=A \right]\leq \exp\left(-\Omega\left(\frac{f(A)^2}{n}\right)\right).
\end{align*}

Finally, we check the condition~\ref{cond:whp} of \cref{cor:nazocor}.
From \cref{lem:beta'bounds}, we have
\begin{align*}
\Pr[A^{(t+1)}\not\in \mathcal{B}_2 \land f(A^{(t+1)})\leq m \mid A^{(t)}\in\mathcal{B}_2]
&\leq \Pr\left[\exists j\in I_{<1}:\,|\beta'_j|> K\sqrt{\frac{\log n}{n}}\,\middle|\,A\in\mathcal{B}_2\right] \\
&\leq n^{-\Omega(1)}.
\end{align*}

Now we are ready to apply \cref{cor:nazocor}.
Thus, there exists $\tau=O(\log n)$ such that $f(A^{(\tau)})\geq K\sqrt{\frac{\log n}{n}}$ and $|\beta^{(\tau)}_j| \leq K\sqrt{\frac{\log n}{n}}$ hold w.h.p.~for every $j\in I_{<1}$.
Consequently, $A^{(\tau)}\in \mathcal{B}_3 \cup \mathcal{B}_4$ holds w.h.p.

\paragraph*{Case I\hspace{-.1em}I\hspace{-.1em}I: $A^{(0)}\in\mathcal{B}_3$.}
Suppose that $A^{(0)}\in \mathcal{B}_3$.
From \cref{eqn:beta'bounds}, it holds w.h.p.~that
\begin{align*}
|\beta'_z|&\geq |\lambda_z||\beta_z|-C\|{\boldsymbol \beta}\|_\infty - C\sqrt{\frac{\log n}{n}} \\
&\geq (1+\epsilon)|\beta_z| + (\epsilon|\beta_z|-C|\beta_z|^2)+\left(\epsilon|\beta_z|-C\sqrt{\frac{\log n}{n}}\right) \\
&\geq (1+\epsilon)|\beta_z|.
\end{align*}
Moreover, for any $j\in I_{<1}$, it holds w.h.p.~that
\begin{align*}
|\beta'_j| \leq (1-3\epsilon)\|{\boldsymbol \beta}\|_\infty + C\|{\boldsymbol \beta}\|^2_{\infty} + C\sqrt{\frac{\log n}{n}}
\leq (1-\epsilon)|\beta_z|
\end{align*}
These imply that $A^{(t+1)}\not\in\mathcal{B}_1 \cup \mathcal{B}_2$ holds w.h.p.~whenever $A^{(t)}\in \mathcal{B}_3$.
Let $\tau$ be the stopping time given by $\tau \defeq \min\{t\,:\,A^{(t)}\not\in \mathcal{B}_3\}$.
Then, $\|{\boldsymbol \beta}^{(t+1)}\|_\infty \geq (1+\epsilon)\|{\boldsymbol \beta}^{(t)}\|_\infty $ holds w.h.p.~for all $t<\tau$.
Therefore, we have $A^{(\tau)}\in\mathcal{B}_4$ with $\tau=O(\log n)$, and $|\beta^{(\tau)}_j|\leq \epsilon'$ for all $j\in I_{<1}$.

\paragraph*{Proof of \cref{lem:beta'bounds}.}
Suppose $A\in\mathcal{B}_2$ and recall the definition $K=\frac{C}{\epsilon}$.
For any $i\in I_{<1}$, the bound \cref{eqn:beta'bounds} yields that
\begin{align*}
|\beta'_i|&\leq (1-3\varepsilon)K\sqrt{\frac{\log n}{n}}+K^2\frac{\log n}{n}+ C\sqrt{\frac{\log n}{n}} \\
&\leq \frac{C}{\varepsilon}\sqrt{\frac{\log n}{n}} - 2C\sqrt{\frac{\log n}{n}} + K^2\frac{\log n}{n} \\
&\leq K\sqrt{\frac{\log n}{n}}
\end{align*}
holds w.h.p.
This completes the proof of the statement \ref{state:small}.

Now we consider the statement \ref{state:large}.
Suppose that $A\in\mathcal{B}_2$ and $|\beta_i|\geq \frac{C}{\epsilon}\cdot\frac{1}{\sqrt{n}}$ (we expect $h=\frac{C}{\epsilon}$).
For $i\in I_{>1}$, the bound \cref{eqn:exp_beta'} implies $|\E[\beta'_i]|\geq |\lambda_i||\beta_i| - C\|{\boldsymbol \beta}\|^2_{\infty} - \frac{C}{\sqrt{n}}$.
Since $\|{\boldsymbol \beta}\|_\infty \leq K\sqrt{\frac{\log n}{n}}$ and $|\beta_i|\geq \frac{h}{\sqrt{n}}= \frac{C}{\epsilon}\cdot \frac{1}{\sqrt{n}}$, we have
\begin{align*}
&C\|{\boldsymbol \beta}\|^2_{\infty} \leq \frac{CK\log n}{n}\leq \epsilon|\beta_i| \,\,\,\,\text{(for sufficiently large $n$)},\\
&\frac{C}{\sqrt{n}}\leq \epsilon|\beta_i|.
\end{align*}
This leads to $|\E[\beta'_i]|\geq (1+\epsilon)|\beta_i|$, which completes the proof of the statement \ref{state:large}. \QED

\section{Best-of-two voting process} \label{sec:2Choices}

This section is devoted to show \cref{thm:phasetransition_2C,thm:worst_case_2C}.
To this end, we investigate the induced dynamical system of the Best-of-two.
Let $\alpha\defeq \frac{|A|}{n}$, $\alpha_i \defeq \frac{|A_i|}{n}$, $\alpha'_i\defeq \frac{|A'_i|}{n}$ and $r \defeq \frac{q}{p}$.
Suppose that $r$ is a constant.
Let $f^\Btwo_1(x)\defeq 2x(1-x)$ and $f^\Btwo_2(x)\defeq x^2$.
Observe that the Best-of-two is the polynomial voter process according to $f^\Btwo_1$ and $f^\Btwo_2$.
Then, on an $f^{\Btwo}_1$-good and $f^{\Btwo}_2$-good $G(2n,p,q)$, it holds w.h.p.~that
\begin{align}
\left|\alpha_i'-\alpha_if^{\Btwo}_1\left(\frac{\alpha_i+r\alpha_{3-i}}{1+r}\right)-(1-\alpha_i)f^{\Btwo}_2\left(\frac{\alpha_i+r\alpha_{3-i}}{1+r}\right) \right|= O\left(\sqrt{\frac{1}{np}}+\sqrt{\frac{\log n}{n}}\right)
\label{thm:mainthm_E2}
\end{align}
for all $A\subseteq V$ and $i=1,2$.

For convenience of calculation, we change the coordinate of $H$ by
\begin{align*}
{\boldsymbol \delta}=(\delta_1, \delta_2)\defeq (\alpha_1-\alpha_2, \alpha_1+\alpha_2-1).
%&\delta_1\defeq \alpha_1-\alpha_2,\\
%&\delta_2\defeq \alpha_1+\alpha_2-1.
\end{align*}
Note that $\delta_1,\delta_2$ are random variables.
Let $u\defeq\frac{1-r}{1+r}$. Then we have
\begin{align*}
\E[\delta'_i \mid A]=T_i(\delta_1,\delta_2)+O\left(\frac{1}{\sqrt{np}}\right),
\end{align*}
where
\begin{align}
T_1(d_1,d_2)\defeq \frac{d_1}{2}\left((2u+1)-(ud_1)^2-(2u+1) d_2^2\right), \, 
T_2(d_1,d_2)\defeq \frac{ d_2}{2}\left(3-u(2+u) d_1^2- d_2^2\right). \label{eqn:2Cvecfielddef}
\end{align}
This suggests another dynamical system $T(\mathbf{d})=(T_1(\mathbf{d}),T_2(\mathbf{d}))$.
We use $\mathbf{d}=(d_1,d_2)$ as a specific point and ${\boldsymbol \delta}=(\delta_1,\delta_2)$ as a vector-valued random variable.
Consider ${\boldsymbol \delta^{(t)}}=(\delta^{(t)}_1,\delta^{(t)}_2)$ and $(\mathbf{d}^{(t)})_{t=0}^\infty$, where $\mathbf{d}^{(0)}={\boldsymbol \delta}^{(0)}$ and $\mathbf{d}^{(t+1)}=T(\mathbf{d}^{(t)})$ for each $t\geq 0$.
%\begin{align*}
%\begin{cases}
%\mathbf{d}^{(0)}={\boldsymbol \delta}^{(0)},\\
%\mathbf{d}^{(t+1)}=T(\mathbf{d}^{(t)}).
%\end{cases}
%\end{align*}
For notational convenience, we use ${\boldsymbol \delta}'\defeq {\boldsymbol \delta}^{(t+1)}$ for ${\boldsymbol \delta}={\boldsymbol \delta}^{(t)}$.
Similarly, we refer $\mathbf{d}'$ to $T(\mathbf{d})$.
Note that ${\boldsymbol \delta}$ satisfies $|\delta_1|+|\delta_2|\leq 1$ since ${\boldsymbol \alpha}\in[0,1]^2$.
Moreover, the dynamical system $T$ has symmetry on $d_1=0$ and $d_2=0$.
Hence, we focus on the set
\begin{align*}
S\defeq \{(d_1,d_2)\in[0,1]^2:d_1+d_2\leq 1\}.
\end{align*}
%\Cref{fig:vectorfield_delta} illustrates the dynamical system $T$ on $S$ for $u\in\{0.6,\,0.7,\,0.8\}$.
The sequence $(\mathbf{d}^{(t)})_{t=0}^\infty $ is closed in $S$ as follows:
\begin{lemma} \label{lem:2Cvector_field_lem1}
For any $\mathbf{d}\in S$, it holds that $\mathbf{d}'\in S$.
\end{lemma}
The proof is the same as that of \cref{lem:3Mvector_field_lem1} and we omit in this paper.

We will consider the behavior of ${\boldsymbol \delta}$ around fixed points.
Note that $\mathbf{d}'=\mathbf{d}\in S$ holds if and only if $\mathbf{d}\in\{\mathbf{d}^*_i\}_{i=1}^4$,
where
\begin{align} \label{eqn:fixed_points_2C}
\mathbf{d}^*_i \defeq \begin{cases}
(0,0) & \text{if $i=1$},\\
\left(\sqrt{\frac{2u-1}{u^2}},0\right)& \text{if $i=2$ and $u\geq\frac{1}{2}$},\\
\left(\sqrt{\frac{u^2+u-1}{(u+1)^2}},\sqrt{\frac{1}{u(u+1)^2}}\right). & \text{if $i=3$ and $u\geq\frac{\sqrt{5}-1}{2}$},\\
(0,1) & \text{if $i=4$}.
\end{cases}
\end{align}
The Jacobian matrix $J$ at $(d_1,d_2)$ is given by
\begin{align} \label{eqn:general_Jacobian_2C}
J&=\frac{1}{2}\left( \begin{array}{cc}
2u+1-3(u d_1)^2-(2u+1) d_2^2 & -2(2u+1) d_1 d_2 \\
-2u(u+2) d_1 d_2 & 3-u(2+u) d_1^2-3d_2^2
\end{array} \right).
\end{align}
Let $J_i$ be the Jacobian matrix at $\mathbf{d}^*_i$.
From \cref{eqn:general_Jacobian_2C}, a straightforward calculation yields
\begin{align}  \label{eqn:Jacobian_2C}
J_i&=\begin{cases}
\frac{1}{2}\left( \begin{array}{cc}
2u+1 & 0 \\
0 & 3
\end{array} \right) & \text{if $i=1$}, \\
\left( \begin{array}{cc}
2-2u & 0 \\
0 & \frac{1-u^2}{u}
\end{array} \right) & \text{if $i=2$ and $u\geq \frac{1}{2}$}, \\
-\frac{1}{(1+u)^2}\left( \begin{array}{cc}
\frac{1+u-4u^2-8u^3+u^4+3u^5}{2u} & \frac{(2u+1)\sqrt{u^2+u-1}}{\sqrt{u}} \\
u(u+2)\sqrt{u^3+u^2-u} & \frac{3-3u-8u^2-2u^3+3u^4+u^5}{2u}
\end{array} \right) & \text{if $i=3$ and $u\geq \frac{\sqrt{5}-1}{2}$}, \\
\left( \begin{array}{cc}
0 & 0 \\
0 & 0
\end{array} \right) & \text{if $i=4$}.
\end{cases}
\end{align}

Depending on the eigenvalues $\lambda_1\geq\lambda_2$ of $J_i$, the property of $\mathbf{d}^*_i$ changes as shown in \cref{tbl:eigentable_2C}.
\begin{table}[htbp]
\caption{
Each cell $(c_1,c_2)$ represents the property of the eigenvalues $\lambda_1\geq\lambda_2$ of the corresponding Jacobian matrix.
Precisely, the sign $c_i$ represents whether $\lambda_i$ is larger than $1$ or not.
For example, $(+,1)$ indicates that $\lambda_1>\lambda_2=1$.
If $(+,-)$ or $(+,+)$, we may apply \cref{prop:escape}.
Indeed, cells with $(-,-)$ correspond sink points in this model.
Note that $\mathbf{d}^*_2$ is saddle if $\frac{2}{3}<u<\frac{3}{4}$ but is sink if $\frac{3}{4}<u\leq 1$.
\label{tbl:eigentable_2C}}
\vspace{1em}
\centering
\begin{tabular}{|c|c|c|c|c|c|}
\hline
points           & $0<u<\frac{1}{2}$ & $u=\frac{1}{2}$ & $\frac{1}{2}<u<\frac{\sqrt{5}-1}{2}$ & $u=\frac{\sqrt{5}-1}{2}$ & $\frac{\sqrt{5}-1}{2}<u\leq 1$ \\ \hline
$\mathbf{d}^*_1$ & $(+,-)$           & $(+,1)$         & $(+,+)$                     & $(+,+)$         & $(+,+)$               \\ \hline
$\mathbf{d}^*_2$ &  undefined  & $(+,1)$         & $(+,-)$                     & $(1,-)$         & $(-,-)$               \\ \hline
$\mathbf{d}^*_3$ &  undefined       &  undefined      &    undefined               & $(1,-)$         & $(+,-)$               \\ \hline
$\mathbf{d}^*_4$ & $(-,-)$           & $(-,-)$         & $(-,-)$                     & $(-,-)$         & $(-,-)$               \\ \hline
\end{tabular}
\end{table}
In this section, we prove the following analogous results as \cref{prop:3Mvectorfield,prop:sink_3M,prop:fastconsensus_3M,prop:escape_3M}.
Recall that, for $\mathbf{d}=(d_1,d_2)\in\mathbb{R}^2$, $\absp{\mathbf{d}}\defeq (|d_1|,|d_2|)\in\mathbb{R}^2$.
\begin{proposition}[Convergence of the orbit]\label{prop:2Cvectorfield}
For any sequence $(\mathbf{d}^{(t)})_{t=0}^\infty$, $\lim_{t\to\infty} \absp{\mathbf{d}^{(t)}}=\mathbf{d}^*_i$ for some $i\in\{1,2,3,4\}$.
Furthermore, if $u<\frac{\sqrt{5}-1}{2}$ and a positive constant $\kappa>0$ exists such that the initial point $\mathbf{d}^{(0)}=(d^{(0)}_1,d^{(0)}_2)\in S$ satisfies $|d^{(0)}_2|>\kappa$, then $\lim_{t\to\infty}\absp{\mathbf{d}^{(t)}}=\mathbf{d}^*_4$.
\end{proposition}
\begin{proof}
If $u=1$, the dynamical system given in \cref{eqn:2Cvecfielddef} is the same as that of the Best-of-three of \cref{eqn:3Mvecfielddef} and thus we are done.

Suppose that $0\leq u<1$.
The proof of this case depends on several results presented in \cref{sec:competitive_system}.
We first claim that the map $T:S\to S$ is competitive (see \cref{sec:competitive_system} for definition).
Then, we apply \cref{thm:orbit_convergence} and complete the proof of the first statement.
To this end, we verify that the map $T$ satisfies the conditions \ref{cond:map_condition1} to \ref{cond:map_condition4} described in \cref{sec:competitive_system}.
The condition~\ref{cond:map_condition1} follows from \cref{lem:criterion_C1}: it is straightforward to check that the Jacobian matrix \cref{eqn:general_Jacobian_2C} satisfies the condition of \cref{lem:criterion_C1}.
We use the Inverse Function Theorem (\cref{thm:IFT}) to verify the condition \ref{cond:map_condition2}.
To this end, we claim that $\det J>0$ for any $\mathbf{d}\in S\setminus\{(0,1)\}$.
In view of \cref{eqn:general_Jacobian_2C}, let
\begin{align*}
f(d_1,d_2) &\defeq \det J \\
&= \frac{1}{4}((2u+1-3u^2d_1^2-(2u+1)d_2^2)(3-u(2+u)d_1^2-3d_2^2)-4u(u+2)(2u+1)d_1^2d_2^2) \\
&\geq f(d_1,1-d_1) = \frac{1}{4}d_1^2(-1+u)\cdot h(d_1),
\end{align*}
where $h(\theta)\defeq 3(-1+u)(1+u)^2\theta^2 + 12(1+u)^2\theta - 4(3+7u+2u^2)$.
A simple calculation yields $h'(\theta) = 6(2-(1-u)\theta)(1+u)^2>0$ for $\theta\in[0,1]$ and thus 
\begin{align*}
h(\theta)\leq h(1)=-3(1-u^3)-7u(1-u) < 0.
\end{align*}
Here, recall that $u<1$.
Therefore, if $d_1>0$, we have
\begin{align*}
\det J &= f(d_1,d_2)
\geq f(d_1,1-d_1)
= \frac{1}{4}d_1^2(-1+u)h(d_1) >0.
\end{align*}
It is straightforward to see that $f(0,d_2)=0$ if and only if $d_2=1$.
In other words, we have
\begin{align}
\det J =f(d_1,d_2)>0 \label{eqn:positive_det_2C}.
\end{align}
for any $(d_1,d_2)\in S\setminus \{(0,1)\}$.
Consequently, for any $\mathbf{d}\in S\setminus \{(0,1)\}$, an inverse mapping $T^{-1}(\mathbf{d})$ is defined.
Moreover, it is straightforward to check that $T(\mathbf{x})=(0,1)$ if and only if $\mathbf{x}=(0,1)$.
Therefore, $T$ is injective and we verified the condition~\ref{cond:map_condition2}.
The condition~\ref{cond:map_condition3} follows immediately since we already have the Jacobian matrix \cref{eqn:general_Jacobian_2C}.
To condition~\ref{cond:map_condition4} follows from \cref{lem:criterion_C4} and \cref{eqn:positive_det_2C}.
Now we apply \cref{thm:orbit_convergence} and complete the proof of the first claim of \cref{prop:2Cvectorfield}.

To obtain the second claim, we show that $d'_2>0$ whenever $(d_1,d_2)\in S$ satisfies $d_2>0$.
This follows from a simple calculation 
\begin{align*}
d'_2 &\geq \frac{d_2}{2}(3-3d_1^2-d_2^2) 
\geq \frac{d_2}{2}(3-3d_1-d_2) 
\geq d_2^2
>0.
\end{align*}
\end{proof}
\begin{proposition}[Dynamics around sink points] \label{prop:sink_2C}
Consider the Best-of-two on $G(2n,p,q)$ that is both $f^\Btwo_1$-good and $f^\Btwo_2$-good.
Suppose that $r=q/p<\sqrt{5}-2$ is a constant.
Then, there exists a positive constant $\epsilon=\epsilon(r)$ satisfying
\begin{align*}
\Pr\left[{\boldsymbol \delta'}\not\in B(\mathbf{d}_2^*,\epsilon) \,\middle|\, {\boldsymbol \delta}\in B(\mathbf{d}_2^*,\epsilon)\right] \leq \exp(-\Omega(n)).
\end{align*}
In particular, 
%let $\tau=\inf\{t:{\boldsymbol \delta}^{(t)}\not\in B(\mathbf{d}^*_2,\epsilon),{\boldsymbol \delta}^{(0)}\in B(\mathbf{d}^*_2,\epsilon)\}$ be the stopping time.
%Then, 
$\Tcons(A) = \exp(\Omega(n))$ holds w.h.p.~for any $A\subseteq V$ satisfying $\absp{{\boldsymbol \delta}} \in B(\mathbf{d}_2^*,\epsilon)$.
%
%In particular, let $\tau=\inf\{t:{\boldsymbol \delta}^{(t)}\not\in B(\mathbf{d}^*_2,\epsilon),{\boldsymbol \delta}^{(0)}\in B(\mathbf{d}^*_2,\epsilon)\}$ be the stopping time.
%Then, $\tau\geq \exp(\Omega(n))$ w.h.p.
\end{proposition}
\begin{proof}
From \cref{eqn:Jacobian_2C}, it is easy to check that both $\mathbf{d}^*_2$ and $-\mathbf{d}^*_2$ are sink.
Then, \cref{prop:sink_2C} follows from \cref{prop:sinkpoint}.
\end{proof}
\begin{proposition}[Towards consensus] \label{prop:fastconsensus_2C}
Consider the Best-of-two on $G(2n,p,q)$ that is both $f^\Btwo_1$-good and $f^\Btwo_2$-good.
Suppose that $r=q/p$ is a constant.
Then, there exists a universal constant $\epsilon=\epsilon(r)>0$ satisfying the following:
$\Tcons(A)\leq O(\log\log n+\log n/\log (np))$ holds w.h.p.~for all $A\subseteq V$ with $\min\{|A|,2n-|A|\}\leq \epsilon n$.
\end{proposition}
\begin{proof}
Since $J_4$ is the all-zero matrix, \cref{prop:fastconsensus_2C} immediately follows from \cref{prop:fastconsensus_Jacob}.
\end{proof}
\begin{proposition}[Escape from fixed points] \label{prop:escape_2C}
Let $p,q$ are constants and consider the Best-of-two on $G(2n,p,q)$ that is both $f^{\Btwo}_1$-good and $f^{\Btwo}_2$-good.
If $q/p>\sqrt{5}-2$ and $|\delta_2^{(0)}|=o(1)$, then it holds w.h.p.~that $|\delta_2^{(\tau)}|>\epsilon$ for some $\tau=O(\log n)$ and some constant $\epsilon>0$.
\end{proposition}
\begin{proof}[Proof sketch]
The proof is the same as that of \cref{prop:escape_3M} and we just present the sketch.
We first show the following claim:
\begin{claim}[Concentration of the variance for the Best-of-two]  \label{claim:mainthm_Var_2C}
Consider the Best-of-two on $G(2n,p,q)$ that is both $f^{\Btwo}_1$-good and $f^{\Btwo}_2$-good.
Then, two constants $C_1,C_2>0$ exist such that
\begin{align*}
&\forall A\subseteq V, \, \forall i\in\{1,2\}\,: \\
&\hspace{3em}
\left|\Var\left[|A'_i|\,\middle|\, A \right]
-
|A_i| g_1\left(\frac{|A_i|p+|A_{3-i}|q}{n(p+q)}\right) - (n-|A_i|) g_2\left(\frac{|A_i|p+|A_{3-i}|q}{n(p+q)}\right)
\right| \leq C_2\sqrt{\frac{n}{p}},
\end{align*}
where $g_i(x)\defeq f^{\Btwo}_i(x)(1-f^{\Btwo}_i(x))$ for $i=1,2$.
\end{claim}
\begin{proof}
That variance $\Var[|A'_i| \mid A]$ can be represented as
\begin{align*}
\Var[|A'_i| \mid A]
= \sum_{v\in V_i} \Pr[v\in A'](1-\Pr[v\in A'])
= \sum_{v\in A_i} g_1\left(\frac{\deg_A(v)}{\deg(v)}\right) 
+ \sum_{v\in V_i\setminus A_i} g_2\left(\frac{\deg_A(v)}{\deg(v)}\right).
\end{align*}
Therefore, \cref{claim:mainthm_Var_2C} immediately follows from the property~\ref{pro:gomiN} and \cref{thm:fgood}.
\end{proof}

From \cref{claim:mainthm_Var_2C}, we have $\Var[\alpha'_i\mid A]=\Omega(n^{-1})$ if $\min\{\alpha_1,1-\alpha_1,\alpha_2,1-\alpha_2\}=\Omega(1)$.
This verifies the condition~\ref{asm:asm2} of Assumption~\ref{asm:escape_assumption0}.
We consider two cases: $u\neq \frac{\sqrt{5}-1}{2}$ and $u=\frac{\sqrt{5}-1}{2}$.
Suppose that $u\neq\frac{\sqrt{5}-1}{2}$.
Then both $\mathbf{d}^*_1$ and $\mathbf{d}^*_2$ satisfy Assumption~\ref{asm:escape_assumption2} and we can apply \cref{prop:escape}.
Suppose that  $u=\frac{\sqrt{5}-1}{2}$.
In this case, we have $\mathbf{d}^*_1=\mathbf{d}^*_2=(0,0)$.
This point satisfies conditions~\ref{asm:asm4} and \ref{asm:asm5} of Assumption~\ref{asm:escape_assumption1} and apply \cref{prop:escape}.
\end{proof}

\paragraph*{Proof sketches of \cref{thm:phasetransition_2C,thm:worst_case_2C}.}
Combining \cref{prop:2Cvectorfield,prop:sink_2C,prop:fastconsensus_2C,prop:escape_2C}, the proofs of \cref{thm:phasetransition_2C,thm:worst_case_2C} follow from the same argument as that of \cref{thm:phasetransition_3M,thm:worst_case_3M} (see \cref{sec:apply_to_3M}).

\section{Concluding remark} \label{sec:concluding}
In this paper we studied the Best-of-two and the Best-of-three voting processes on the stochastic block model $G(2n,p,q)$.
Here, we first generate $G(2n,p,q)$, then set an initial opinion configuration and observe the voting process.
We presented phase transition results on $r=q/p$ for both processes. 
In addition, if $p\geq q>0$ are constants, we proved that the consensus time is $O(\log n)$ for arbitrary initial opinion configurations.
In the proof, we combined the theory of dynamical systems and our technical result~\cref{thm:approximation_vectorfield_general} which approximates the stochastic processes by the corresponding appropriate deterministic processes.
To estimate the probability which the process reaches sink areas from the source area is future work to consider an application of these processes to distributed community detection algorithms.
For an application to distributed community detection algorithms, it is significant to estimate the probability that the voting process reaches the sink areas (in particular, starting from the source area).
This is a possible future direction of this paper.

Note that \cref{thm:approximation_vectorfield_general} is allowable for any polynomial function with constant degree.
For example, consider the Best-of-$(2k+1)$ voting process for a positive constant $k$.
This process is defined by $f_1(x)=f_2(x)\defeq \Pr[{\rm Bin}(2k+1,x)\geq k+1]=\sum_{i=k+1}^{2k+1}\binom{2k+1}{i}x^i(1-x)^{2k+1-i}$.
%Since $f$ is a polynomial function with degree $2k+1$,
\Cref{thm:approximation_vectorfield_general} guarantees $\|{\boldsymbol \alpha}^{(t)}-\mathbf{a}^{(t)}\|_\infty\leq C^t(\sqrt{1/np}+\sqrt{\log n/n})$ for this process.
Moreover, using \cref{lem:WIdiscrepancy}, it is not difficult to extend \cref{thm:approximation_vectorfield_general} to voting processes on general stochastic block models that has $c_1$ communities each of size $\Omega(n)$ and initially involving $c_2$ opinions, where $c_1, c_2$ denote arbitrary positive constants.
This setting yields induced dynamical systems of dimension more than two.
The Jacobian matrix would be helpful to investigate several properties including the exponential time lower bound (\cref{prop:sinkpoint}), the fast consensus (\cref{prop:fastconsensus_Jacob}) and escape result (\cref{prop:escape}) since the proofs of \cref{sec:sinkpoint_proof,sec:fastconsensus_proof,sec:escape_proof} work for induced dynamical systems with higher dimension.
Unfortunately, it may not be easy to specify other properties (e.g.~zero areas, convergence properties, \ldots) of $(\mathbf{a}^{(t)})_{t\in \mathbbm{N}}$ corresponding to such processes.
This problem is left for future work.
Also, the worst-case analysis of the consensus time for sparse random graphs remains open in this paper.

%謝辞, full paperには載せる
\section*{Acknowledgements}
This work is supported by JST CREST Grant Number JPMJCR14D2, and JSPS KAKENHI Grant Number 17H07116, 19J12876 and 19K20214, Japan.
We would like to thank Colin Cooper, Nan Kang, and Tomasz Radzik for helpful discussions.

%\newpage

\bibliographystyle{abbrv}
\bibliography{ref}

%\newpage
\appendix
\section{Tools}
\subsection{Linear algebra tools} \label{sec:linearalgebra}
\begin{definition}[singular value]
For a real matrix $A\in \mathbb{R}^{m\times n}$, \emph{singular values} $\sigma_1,\ldots,\sigma_m$ of $A$ are nonnegative square roots of eigenvalues of $AA^\top$.
We write $\sigma_i(A)$ when we specify $A$.
In particular, \emph{the maximum singular value}, denoted by $\sigma_{\max}$, is the largest value among all singular values.
\end{definition}

\begin{proposition} \label{prop:singularvalue}
For a real matrix $A\in\mathbb{R}^{m\times n}$, it holds that
\begin{align*}
\sigma_{\max}=\max_{\mathbf{v}\in \mathbb{R}^n:\|\mathbf{v}\|_2=1}\|A\mathbf{v}\|_2,
\end{align*}
where the norm $\|\cdot\|_2$ is the $\ell^2$ norm.

In particular, it holds that
\begin{align*}
\|A\mathbf{v}\|_2\leq \sigma_{\max}\|\mathbf{v}\|_2
\end{align*}
for any vector $\mathbf{v}\in\mathbb{R}^n$.
\end{proposition}
%
%\newpage

\subsection{Real analysis tools}\label{sec:lipschitzcondition}
\paragraph*{The Inverse Function Theorem.}
The Inverse Function Theorem is a fundamental result in real analysis and can be seen in many textbooks~\cite{DR14,Krantz16}.
  
\begin{theorem}[The Inverse Function Theorem, (See, e.g.~Theorem 12.17 of \cite{Krantz16} and Theorem 1A.1 of \cite{DR14})] \label{thm:IFT}
Let $f$ be a continuously differentiable function from an open set $U\subseteq \mathbb{R}^k$ into $\mathbb{R}^k$.
Suppose that the Jacobian matrix $J$ at $\mathbf{p}\in U$ is invertible.
Then there is a neighborhood $V$ of $\mathbf{p}$ such that the restriction of $f$ to $V$ is invertible.
Moreover, the Jacobian matrix of $f^{-1}$ at $p$ is given by $J^{-1}$.
\end{theorem}

\paragraph*{Lipschitz condition} 
\begin{definition} \label{def:lipschitzdef}
Consider a function $H:S\to T$, where $S\subseteq\mathbb{R}^m$ and $T\subseteq\mathbb{R}^n$ are closed sets.
The function $H$ satisfies the \emph{Lipschitz condition} if there exists a universal constant $C>0$ such that
\begin{align*}
\|H(\mathbf{x})-H(\mathbf{y})\|_\infty \leq C\|\mathbf{x}-\mathbf{y}\|_\infty
\end{align*}
holds for any $\mathbf{x},\mathbf{y}\in S$.
\end{definition}
It should be noted that the definition of the Lipschitz condition does not depends on the norm $\|\cdot\|$ on $\mathbb{R}^n$.
The following is a well-known result in real analysis.

\begin{proposition}[Exercise 1D.3 of \cite{DR14}]  \label{prop:lipschitzC1prop}
Let $O\subseteq \mathbb{R}^k$ be an open set and $S\subseteq O$ be a compact convex subset of $O$.
Suppose that $H:O\to \mathbb{R}^k$ is continuously differentiable on an open set $O$.
Then $H$ is Lipschitz continuous on $C$ and
\begin{align*}
\|H(\mathbf{x})-H(\mathbf{y})\|_\infty \leq \max_{\mathbf{p}\in S}\sigma_{\max}(J_{\mathbf{p}})\|\mathbf{x}-\mathbf{y}\|_\infty,
\end{align*}
where $J_{\mathbf{p}}$ is the Jacobian matrix at $\mathbf{p}$.

\end{proposition}

\begin{corollary} \label{cor:lipschitzprop}
Let $H:\mathbb{R}^m\to\mathbb{R}^n$ be a function given by
\begin{align*}
H(\mathbf{x})=(H_1(\mathbf{x}),\ldots,H_n(\mathbf{x})),
\end{align*}
where $H_i(\mathbf{x})=H_i(x_1,\ldots,x_m)$ is a polynomial on $x_1,\ldots,x_m$ for all $i\in[n]$.
Then, $H$ satisfies the Lipschitz condition on $[0,1]^m$.
\end{corollary}

\subsection{Probabilistic tools}
\if0
\begin{lemma} [Chernoff; Theorem 21.6 of \cite{FK16}] \label{lem:chernoffbound1}
Let $X_1, X_2, \ldots, X_n$ be independent random variables satisfying $0\leq X_i\leq 1$ and $\E[X_i]=p_i$.
Let $X=\sum_{i=1}^nX_i$.
Then, for any $t>0$, it holds that
\begin{align*}
\Pr[X\geq \E[X]+t]\leq \exp\left(-\frac{t^2}{2(\E[X]+t/3)}\right)
\end{align*}
and
\begin{align*}
\Pr[X\leq \E[X]-t]\leq \exp\left(-\frac{t^2}{2(\E[X]-t/3)}\right).
\end{align*}
\end{lemma}
\fi

\begin{lemma}[Chernoff bound (See, e.g.~Theorem 10.1, Corollary 10.4 and Theorem 10.5 of \cite{Doerr18})]
\label{lem:chernoff}
Let $X_1, X_2, \ldots, X_n$ be independent random variables taking values in $[0,1]$.
Let $X=\sum_{i=1}^nX_i$.
Then the following hold:
\begin{enumerate}[label=(\roman*)]
\item \label{eqn:cher1} for any $\delta\geq 0$, 
\begin{align*}
\Pr[X\geq (1+\delta)\E[X]] \leq \exp\left(-\frac{\min\{\delta,\delta^2\}\E[X]}{3}\right).
\end{align*}
\item \label{eqn:cher2} for any $\delta\in [0,1]$, 
\begin{align*}
\Pr[X\leq (1-\delta)\E[X]] \leq \exp\left(-\frac{\delta^2\E[X]}{2}\right).
\end{align*}
\item \label{eqn:cher3} for any $k\geq 2\mathrm{e}\E[X]$, 
\begin{align*}
\Pr\left[X\geq k \right]\leq 2^{-k}.
\end{align*}
Here, $\mathrm{e}$ denotes the Napier constant.
\end{enumerate}
\end{lemma}

\begin{lemma}[Additive Chernoff bound (See, e.g.~Theorems 10.10 and 10.11 of \cite{Doerr18})]
\label{lem:chernoff-additive}
\ 

\noindent 
Let $X_1, X_2, \ldots, X_n$ be independent random variables taking values in $[0,1]$.
Let $X=\sum_{i=1}^nX_i$.
Then for any $\delta \geq 0$, 
\begin{align*}
&\Pr[X\geq \E[X]+\delta]\leq \exp\left(-\frac{1}{3}\min\left\{\frac{\delta^2}{\E[X]},\delta\right\}\right), \\
&\Pr[X\leq \E[X]-\delta]\leq \exp\left(-\frac{\delta^2}{2\E[X]}\right).
\end{align*}
\end{lemma}
%

%\begin{lemma}[Corollary 10.4 of \cite{Doerr18}] \label{lem:chernoff-k}
%Let $X_1, X_2, \ldots, X_n$ be independent random variables taking values in $[0,1]$. %such that $\Pr[X_i=1]=p_i$.
%Let $X=\sum_{i=1}^nX_i$. 
%Then for all $k\geq 2e\E[X]$, 
%\begin{eqnarray*}
%\Pr\left[X\geq k \right]&\leq &2^{-k}.
%\end{eqnarray*}
%\end{lemma}

\begin{lemma}[Hoeffding bound (See, e.g.~Theorem 10.9 of \cite{Doerr18})]
\label{lem:Hoeffding}
Let $X_1, X_2, \ldots, X_n$ be independent random variables.
Assume that each $X_i$ takes values in a real interval $[a_i, b_i]$ of length $c_i\defeq b_i-a_i$.
Let $X=\sum_{i=1}^nX_i$.  
Then for any $\delta>0$,
\begin{align*}
\Pr\left[X\geq \E[X]+\delta \right]&\leq \exp\left(-\frac{2\delta^2 }{\sum_{i=1}^nc_i^2}\right), \\
\Pr\left[X\leq \E[X]-\delta \right]&\leq \exp\left(-\frac{2\delta^2 }{\sum_{i=1}^nc_i^2}\right).
\end{align*}
\end{lemma}

%
%Let $\Phi(z)=\frac{1}{\sqrt{2\pi}}\int_{-\infty}^z \mathrm{e}^{-x^2/2} dx$ (standard normal distribution).
%
%\begin{lemma}
%It holds for any $z$ that
%\begin{eqnarray}
%\frac{1}{\sqrt{2\pi}(1+z)}\cdot \mathrm{e}^{-z^2/2}\leq 1-\Phi(z)\leq \frac{1}{\sqrt{\pi}(1+z)}\cdot \mathrm{e}^{-z^2/2}.
%\end{eqnarray}
%\end{lemma}
%
\begin{lemma}[Berry-Esseen (See, e.g.~\cite{Shevtsova10})]
\label{lem:Berry-Esseen}
Let $X_1, X_2, \ldots, X_n$ be independent random variables such that
$\E[X_i]=0$, $\E[X_i^2]>0$, $\E[|X_i|^3]<\infty$ for all $i\in [n]$ and $\sum_{i=1}^n\E[X_i^2]=1$.
Let $X=\sum_{i=1}^nX_i$ and 
let $\Phi(x)=\frac{1}{\sqrt{2\pi}}\int_{-\infty}^x\mathrm{e}^{-y^2/2}\mathrm{d}y$ (the cumulative distribution function of the standard normal distribution).
Then
\begin{align*}
\sup_{x\in \mathbbm{R}}\bigl|\Pr\left[X\leq x\right]-\Phi(x)\bigr|
&\leq 5.6 \sum_{i=1}^n\E[|X_i|^3].
\end{align*}
\end{lemma}

%\textcolor{red}{$\E[Z_i^2]=0\iff \sum_{z\in \mathbbm{R}}z^2\Pr[Z_i=z]=0$だから, $\E[Z_i^2]=0$ならば$\Pr[Z_i=0]=1$が成り立つ.}
\begin{corollary}\label{cor:BEbound_cor}
Let $X_1, X_2,\ldots, X_n$ be $n$ independent random variables such that $\Var[X]\neq 0$ and $|X_i-\E[X_i]|\leq C<\infty$ for all $i\in [n]$ where $X=\sum_{i=1}^nX_i$.
Then for any $x\in \mathbb{R}$,
\begin{align*}
\left|\Pr\left[\frac{X-\E[X]}{\sqrt{\Var[X]}}\leq x\right]-\Phi(x)\right|
&\leq \frac{5.6C}{\sqrt{\Var[X]}}.
\end{align*}
%where $c=(c_1,\ldots,c_n)\in\mathbb{R}^n$.
\end{corollary}
\begin{proof}
For all $i\in [n]$, let
\begin{align*}
Z_i&\defeq \frac{X_i-\E[X_i]}{\sqrt{\Var[X]}}, \\
Z&\defeq \sum_{\substack{i\in [n]: \E[Z_i^2]>0}}Z_i
=\sum_{i\in [n]}Z_i.
\end{align*}
Note that 
$\E[Z_i^2]=0
\iff \sum_{z}z^2\Pr[Z_i=z]=0
\iff \Pr[Z_i=0]=1$.
Then, for all $i\in \{j\in [n]: \E[Z_j^2]> 0\}$, 
it is easy to check that
$\E[Z_i]=0$, 
$\E[Z_i^2]>0$, and 
$\E[|Z_i|^3]\leq \frac{C^3}{\Var[X]^{3/2}}<\infty$.
Furthermore,  
\begin{align*}
\sum_{\substack{i\in [n]:\E[Z_i^2]>0}}\E[Z_i^2]
&=\sum_{i\in [n]}\E[Z_i^2]
=\frac{\sum_{i\in [n]}\E[(X_i-\E[X_i])^2]}{\Var[X]}=1.
\end{align*}
Thus we can apply \cref{lem:Berry-Esseen} to $Z$ and it holds that
\begin{align*}
\left|\Pr\left[\frac{X-\E[X]}{\sqrt{\Var[X]}} \leq x\right]-\Phi(x)\right|
&=\left|\Pr\left[\sum_{i=1}^n Z_i \leq x\right]-\Phi(x)\right|\\
&=\left|\Pr\left[Z \leq x\right]-\Phi(x)\right|\\
&\leq 5.6 \sum_{i\in[n]:\E[Z_i^2]>0}\E[|Z_i|^3]\\
&\leq \frac{5.6C}{\sqrt{\Var[X]}} \sum_{i=1}^n\E[Z_i^2]
=\frac{5.6C}{\sqrt{\Var[X]}}.
\end{align*}
\end{proof}

\begin{lemma}[Lemma 4.5 of \cite{AMLEFG18}] \label{lem:nazolemma}
Consider a Markov chain $(X_t)_{t=1}^\infty$ with finite state space $\Omega$ and a function $f:\Omega \to \{0,\ldots,n\}$.
Let $C_3$ be arbitrary constant and $m=C_3 \sqrt{n\log n}$.
Suppose that $\Omega,f$ and $m$ satisfies the following conditions:

\begin{enumerate}[label=$(\arabic*)$]
\item\label{state:nazo_sqrt} For any positive constant $h$, there exists a positive constant $C_1<1$ such that
\begin{align*}
\Pr\left[f(X_{t+1})<h\sqrt{n}\,\middle|\,f(X_t)\leq m\right] < C_1.
\end{align*}

\item\label{state:nazo_cher} Three positive constants $\epsilon, C_2$ and $h$ exist such that, for any $x\in \Omega$ satisfying $h\sqrt{n}\leq f(x)<m$,
\begin{align*}
\Pr\left[f(X_{t+1})< (1+\epsilon)f(X_t) \,\middle|\, X_t=x\right] < \exp\left(-C_2\frac{f(x)^2}{n}\right).
\end{align*}
\end{enumerate}

Then $f(X_\tau) \geq m$ holds for some $\tau=O(\log n)$.
\end{lemma}

\begin{corollary} \label{cor:nazocor}
Consider a Markov chain $(X_t)_{t=1}^\infty$ with finite state space $\Omega$ and a function $f:\Omega \to \{0,\ldots,n\}$.
Let $C_3$ be arbitrary constant and $m=C_3 \sqrt{n\log n}$.
Consider a set $\mathcal{B}\subseteq \Omega$ such that
\begin{align*}
\mathcal{B}\subseteq \{x\in \Omega\,:\,f(x)< m\}.
\end{align*}
Suppose that $\Omega,f,m$ and $\mathcal{B}$ satisfy the following conditions:

\begin{enumerate}[label=$(\arabic*')$]
\item\label{cond:nazo_sqrt} For any positive constant $h$, there exists a positive constant $C_1<1$ such that
\begin{align*}
\Pr\left[f(X_{t+1})<h\sqrt{n}\,\middle|\,f(X_t)\leq m,X_t\in\mathcal{B}\right] < C_1.
\end{align*}

\item\label{cond:chernoff} Three positive constants $\epsilon, C_2,h$ exist such that, for any $x\in \mathcal{B}$ satisfying $h\sqrt{n}\leq f(x)<m$,
\begin{align*}
\Pr\left[f(X_{t+1})< (1+\epsilon)f(X_t) \,\middle|\, X_t=x\right] < \exp\left(-C_2\frac{f(x)^2}{n}\right).
\end{align*}

\item\label{cond:whp} For some constant $C_4>0$,
\begin{align*}
\Pr\left[X_{t+1}\not\in\mathcal{B}\text{ and } f(X_{t+1})< m\,\middle|\,X_t\in\mathcal{B}\right] \leq O(n^{-C_2}).
\end{align*}
\end{enumerate}

Then, 
\begin{align*}
\Pr\left[f(X_\tau) \geq m \mid X_0\in \mathcal{B}\right]\geq 1-n^{-\Omega(1)}
\end{align*}
holds for some $\tau=O(\log n)$.
\end{corollary}
\begin{proof}
Let $\Omega'=\mathcal{B}\cup \{a,b\}$ be the state space with two special states $a$ and $b$.
We consider a Markov chain $(X'_t)_{t+1}^\infty$ on $\Omega'$ by
\begin{align*}
\Pr[X'_{t+1}=x \mid X'_t=y] =\begin{cases}
\Pr[X_{t+1}=x \mid X_t=y] & \text{if $x,y\in \mathcal{B}$},\\
\Pr[X_{t+1}\not\in\mathcal{B}\land f(X_{t+1})< m \mid X_t=y] & \text{if $x=a$ and $y\in\mathcal{B}$}, \\
\Pr[f(X_{t+1})\geq m \mid X_t=y] &\text{if $x=b$ and $y\in\mathcal{B}$}, \\
1 & \text{if $x=y\in\{a,b\}$}.
\end{cases}
\end{align*}
In other words, the special state $a$ corresponds to the event ``$f(x)<m$ and $x\not\in\mathcal{B}$~", and $b$ does ``$f(x)\geq m$".

Suppose that $X'_0\in \mathcal{B}$ and let $\tau'=\min\{t:X'_t\not\in\mathcal{B}\}>0$ be the stopping time.
Then, the above definition of $X'_t$ naturally yields a coupling $(X_t,X'_t)_{t<\tau'}$ satisfying $X_t=X'_t$ for $t<\tau'$.

Let $f':\Omega'\to\{0,\ldots,n\}$ be a function given by
\begin{align*}
f'(x) = \begin{cases}
f(x) & \text{if $x\in\mathcal{B}$},\\
n & \text{if $x\in \{a,b\}$}.
\end{cases}
\end{align*}
Then, the Markov chain $(X'_t)$ on $\Omega'$ and the function $f'$ satisfies the conditions \ref{state:nazo_sqrt} and \ref{state:nazo_cher} of \cref{lem:nazolemma}.
Hence, for some $\tau=O(\log n)$, it holds that $X'_\tau \in \{a,b\}$.
We insist that $X'_\tau=b$, that is, $f(X_\tau)\geq m$.
Indeed, from the condition~\ref{cond:whp}, we have
\begin{align*}
\Pr[X'_\tau=a \mid X'_0\in\mathcal{B}]&\leq \tau\cdot O(n^{-c_2}) \\
&\leq n^{-\Omega(1)}.
\end{align*}
\end{proof}

We say a function $f:\{0,1\}^M\to\mathbb{R}$ is \emph{monotone increasing} if $f(\mathbf{x})\leq f(\mathbf{y})$ whenever $\mathbf{x}=(x_1,\ldots,x_M),\mathbf{y}=(y_1,\ldots,y_M)\in\{0,1\}^M$ satisfies $x_i\leq y_i$ for every $i=1,\ldots,M$.

\begin{lemma}[FKG inequality, Theorem 21.5 of \cite{FK16}]\label{lem:FKG}
For a given set $[M]=\{1,2, \ldots, M\}$, let $\Ind_1, \Ind_2, \ldots, \Ind_M$ be independent binary random variables.
Then for any two monotone increase functions $f,g: \{0,1\}^{M}\to \mathbb{R}$, it holds that
\begin{align*}
\E[f(\mathbf{\Ind})g(\mathbf{\Ind})]&\geq \E[f(\mathbf{\Ind})]\E[g(\mathbf{\Ind})]
\end{align*}
where $\mathbf{\Ind}=(\Ind_1, \Ind_2, \ldots, \Ind_M)\in \{0,1\}^{M}$.
\end{lemma}
\begin{lemma}[Janson's inequality (Theorem 21.12 of \cite{FK16})]\label{lem:Janson}
For a given set $[M]=\{1,2, \ldots, M\}$, let $\Ind_1, \Ind_2, \ldots, \Ind_M$ be independent binary random variables.
%and let $E=\{e\in [M]\mid J_e=1\}$ be a random subset of $[M]$.
%
Now, let $(\Es_1, \Es_2, \ldots, \Es_N)$ be a family of $N$ subsets of $[M]$ ($\Es_i\subseteq [M]$ for every $i\in [N]$) and let
\begin{align*}
Y
%=\sum_{i\in [N]}\mathbbm{1}_{\Es_i\subseteq E}
\defeq \sum_{i\in [N]}\prod_{e\in \Es_i}\Ind_e.
\end{align*}
Then, it holds for any $t \leq \E[Y]$ that
\begin{align*}
\Pr\left[Y\leq \E[Y]-t\right]&\leq \exp\left(-\frac{t^2}{2\nabla}\right)
\end{align*}
where
\begin{align*}
\nabla&\defeq \sum_{\substack{i\in N, j\in N:\\\Es_i\cap \Es_j\neq \emptyset}}
%\E[\mathbbm{1}_{\Es_i\subseteq E} \mathbbm{1}_{\Es_j\subseteq E}].
\E\left[\left(\prod_{e\in \Es_i} \Ind_e\right)\left( \prod_{e'\in \Es_j} \Ind_{e'}\right)\right].
\end{align*}
%%%
\end{lemma}
\begin{lemma}[The Kim-Vu concentration (Main Theorem of \cite{KV00})] \label{lem:Kim-Vu}
For a given set $[M]=\{1,2, \ldots, M\}$, let $\Ind_1, \Ind_2, \ldots, \Ind_M$ be independent binary random variables.
%and let $E=\{e\in [M]\mid J_e=1\}$ be a random subset of $[M]$.
%
Now, let $\mathcal{E}\subseteq 2^{[M]}$ be a collection of subsets of $[M]$ ($\Es\subseteq [M]$ for all $\Es\in \mathcal{E}$) and let
\begin{align*}
Y=\sum_{\Es\in \mathcal{E}}w(\Es)\prod_{e\in \Es}\Ind_e,
\end{align*}
where $w(\Es)$ are positive coefficients.
For a subset $A\subseteq [M]$, define $Y_{A}$ as
\begin{align*}
Y_{A}=\sum_{\substack{\Es\in\mathcal{E}:\\\Es\supseteq A}} w(\Es)\prod_{e\in \Es\setminus A}\Ind_e.
\end{align*}
If the polynomial $Y$ has degree at most $k$ (i.e.~$\max_{F\in\mathcal{E}}|F|\leq k$), then for any positive $\lambda>1$, it holds that
\begin{align*}
\Pr\left[|Y-\E[Y]|\geq \sqrt{k! \max_{A\subseteq [M]}\E[Y_{A}]\max_{A\subseteq [M]:A\neq \emptyset}\E[Y_{A}]}(8\lambda)^k\right]\leq 2\exp(2+(k-1)\log M-\lambda).
\end{align*}
\end{lemma}

\subsection[Competitive dynamical systems]{Competitive dynamical systems on $\mathbb{R}^2$} \label{sec:competitive_system}
In this paper, we consider discrete-time dynamic systems given by a map: For a map $T:S\to S$ and $\mathbf{x}\in S$, we discuss whether the orbit $\{T^n(\mathbf{x})\}_{n\geq 0}$ converges to a fixed point or not.
For general dynamical systems, it is typically difficult to predicate the asymptotic behavior of an orbit since it sometimes exhibits chaos.
This section is devoted to introduce \emph{planar competitive dynamical systems}, which includes the induced dynamical systems explored in this paper.
Our dynamical systems are defined on $\mathbb{R}^2$.
The definitions and results in this section follow from~\cite{HS05}.

For two points $\mathbf{x}=(x_1,x_2)$ and $\mathbf{y}=(y_1,y_2)$, we write $\mathbf{x}\leq_K \mathbf{y}$ if both $x_1\leq y_1$ and $y_2 \leq x_2$ hold.
For $S\subseteq \mathbb{R}^2$, a map $T:S \to S$ is \emph{competitive} if $T$ is monotone with respect to the relation $\leq_K$ (i.e.~$T(\mathbf{x})\leq_K T(\mathbf{y})$ whenever $\mathbf{x}\leq_K \mathbf{y}$).
We use $\mathbf{x}\leq \mathbf{y}$ for the usual component-wise comparison (i.e.~$\mathbf{x}\leq\mathbf{y}$ if $x_i\leq y_i$ for $i=1,2$).
Throughout this section, we let
\begin{align*}
S=\{(x,y)\in\mathbb{R}^2:x\geq 0,y\geq 0,x+y\leq 1\}
\end{align*}
and our map $T:S\to S$ is supposed to satisfy the following conditions: 
\begin{enumerate}[label=(C\arabic*)]
\item $T$ is competitive. \label{cond:map_condition1}
\item $T$ is injective. \label{cond:map_condition2}
\item $T$ is $C^1$. \label{cond:map_condition3} %differential at any point $\mathbf{x}\in S$, that is, the Jacobian matrix $J$ is defined on any point of $S$.
\item The inverse map $T^{-1}$ is monotone with respective to the ordinal relation $\leq$. \label{cond:map_condition4}
\end{enumerate}

A point $\mathbf{x}$ is a \emph{fixed point} if $T(\mathbf{x})=\mathbf{x}$.
%Observe that $\mathbf{x}\leq_K (0,1)$ for any $\mathbf{x}\in S$.
%Thus, for any competitive map $T:S\to S$, the point $(0,1)$ is a fixed point.
The following result asserts the convergence of the orbit $\{T^n(\mathbf{x})\}_{n\geq 0}$ for any initial point $\mathbf{x}\in S$:
\begin{theorem} \label{thm:orbit_convergence}
Suppose that a map $T:S\to S$ satisfies the conditions \ref{cond:map_condition1} to \ref{cond:map_condition4}.
Then, for any point $\mathbf{x}\in S$, the limit $\lim_{n\to\infty}T^n(\mathbf{x})$ exists and the limit is a fixed point of $T$.
\end{theorem}

\begin{proof}
For two points $\mathbf{x}=(x_1,x_2)$ and $\mathbf{y}=(y_1,y_2)$, we write $\mathbf{x}\ll \mathbf{y}$ if $x_1<y_1$ and $x_2<y_2$.
We use the following known result:
\begin{theorem}[Theorem 5.28 of \cite{HS05}] \label{thm:competitive_map_known_result}
Suppose that a competitive map $T:S\to S$ satisfies the following:
\begin{center}
$\mathbf{x}\leq \mathbf{y}$ for any $\mathbf{x},\mathbf{y}\in S$ of $T(\mathbf{x}) \ll T(\mathbf{y})$.
\end{center}
Then,  for any $\mathbf{x}\in S$, the sequence $(T^n(\mathbf{x}))_{n\geq 0}$ converges to some fixed point of $T$
\end{theorem}

For any points $\mathbf{x},\mathbf{y}\in S$ of $T(\mathbf{x}) \ll T(\mathbf{y})$, the conditions \ref{cond:map_condition2} and \ref{cond:map_condition4} yield that $\mathbf{x}=T^{-1}(T(\mathbf{x})) \leq T^{-1}(T(\mathbf{y})) = \mathbf{y}$.
Therefore, we can directly apply \cref{thm:competitive_map_known_result}.
\end{proof}

The following results provide useful criterions for the conditions \ref{cond:map_condition1} and \ref{cond:map_condition4} of $T$:
\begin{lemma} \label{lem:criterion_C1}
A map $T:S\to S$ is competitive if, for any point $\mathbf{x}\in S$, the Jacobian matrix $J$ at $\mathbf{x}$ is of the form
\begin{align*}
J=\left(\begin{array}{cc}
j_{11} & j_{12} \\
j_{21} & j_{22}
\end{array} \right)
\end{align*}
with $j_{11},j_{22}\geq 0$ and $j_{12},j_{21}\leq 0$.
\end{lemma}
\begin{proof}
Let $T(\mathbf{x})=(T_1(\mathbf{x}),T_2(\mathbf{x}))$ for $\mathbf{x}=(x_1,x_2)\in S$.
From the assumption on the Jacobian matrix, the function $T_i$ is nondecreasing on $x_i$ and is nonincreasing on $x_{3-i}$.
With this in mind, for any $(a,b),\,(c,d)\in S$ of $(a,b)\leq_K (c,d)$, we have
\begin{align*}
&T_1(a,b) \leq T_1(c,b) \leq T_1(c,d),\\
&T_2(a,b) \geq T_2(a,d) \geq T_2(c,d).
\end{align*}
In other words, $T(a,b)\leq_K T(c,d)$.
\end{proof}

\begin{lemma} \label{lem:criterion_C4}
Suppose that a map $T:S\to S$ satisfies \ref{cond:map_condition1} to \ref{cond:map_condition3}, and that the Jacobian matrix of $T$ at $\mathbf{x}$ has a positive determinant for any $\mathbf{x}\in S\setminus \{(0,1)\}$.
Then, $T$ satisfies \ref{cond:map_condition4}.
\end{lemma}
\begin{proof}
Let $U\defeq T^{-1}$ and $U(\mathbf{x})=(U_1(\mathbf{x}),U_2(\mathbf{x}))$ for $\mathbf{x}=(x_1,x_2)\in S$.
By the Inverse Function Theorem (\cref{thm:IFT}), the Jacobian matrix $K$ of $U$ at $\mathbf{x}\in S\setminus\{(0,1)\}$ is given by the inverse of that of $T$.
Thus, we have
\begin{align*}
\frac{\partial U_i}{\partial x_j}(\mathbf{x})\geq 0
\end{align*}
for any $\mathbf{x}\in S\setminus \{(0,1)\}$.
Hence the functions $U_1(x_1,x_2)$ and $U_2(x_1,x_2)$ are nondecreasing on both $x_1$ and $x_2$.
Therefore for any two points $(a,b),(c,d)\in S$ of $(a,b)\leq (c,d)$, we have $U_1(a,b) \leq U_1(c,d)$ and $U_2(a,b)\leq U_2(c,d)$ (note that if $(a,b)=(0,1)$ then $(c,d)$ must be $(0,1)$ and we are done).
\end{proof}

\end{document}